%% file: Report.tex
\documentclass[12pts]{article}

\usepackage{geometry}
\input{Packages.tex}
\usepackage{authblk}

\title{A functorial approach to Kashiwara-Vergne}
\author[1]{Rodrigo Navarro-Betancourt}
\affil[1]{School of Mathematics, 17 Westland Row, Trinity College Dublin, Ireland}
\date{}

\counterwithin{equation}{section}

\begin{document}

\maketitle

\begin{abstract}
    As a consequence of the proof of the Kashiwara-Vergne conjecture of Alekseev and Torossian \cite{AlekToro}, the authors obtained an injection $\mathrm{GRT} \hookrightarrow \mathrm{KRV}$. The group $\mathrm{GRT}$ can be regarded as the group of automorphisms of the operad of parenthesized chord diagrams, while $\mathrm{KRV}$ can be recovered from the automorphism group of the Goldman-Turaev Lie bialgebra of a thrice-punctured sphere. This suggests the existence of a \textit{natural} way to derive Lie bialgebras from operads, and we verify this is the case. That is, we reproduce the Alekseev-Torossian injection by functorially constructing bracket and cobracket operations out of operad modules. This framework is enough to establish a relationship between Gonzalez' higher genus $\mathrm{GRT}_g$ groups from \cite{gonz}, and the higher genus $\mathrm{KRV}_g$ groups of \cite{Flor1}. Our construction is informed by Massuyeau \cite{Mass} and Turaev's \cite{MasTu} work on Fox pairings and quasi-derivations.
     
\end{abstract}

\tableofcontents

\section{Introduction}

The Kashiwara-Vergne (KV) conjecture is a landmark of Lie theory. In a modern formulation, it posits the existence of an automorphism $F$ on the degree completion of the free Lie algebra in generators $x$ and $y$ such that, among other relations, $F$ satisfies
\[
    F(x+y)= \log(e^x e^y e^{-x} e^{-y})=ch(x,y),
\]
where $ch(x,y)=x+y+\frac{1}{2}[x,y]+ \dots$ is the Campbell-Hausdorff series. The KV conjecture was proved in full generality in \cite{AlekMein}. It was further elucidated in \cite{AlekToro}, where Alekseev and Torossian defined the groups $\mathrm{KV}$ and $\mathrm{KRV}$, which act freely and transitively on solutions of the KV problem. In \cite{AlekToro}, the authors also proved that the graded Grothendieck-Teichm\"uller group $\mathrm{GRT}$ injects into $\mathrm{KRV}$.

In \cite{BarNat}, Bar-Natan gives a topological description of $\mathrm{GRT}$ as the group of automorphisms of the operad of parenthesized chord diagrams $\mathbf{PaCD}$. Surprisingly, the group $\mathrm{KRV}$ also admits a geometric description in terms of the Goldman-Turaev Lie bialgebra of a compact surface.

Let $\Sigma_{g,n+1}$ be a compact oriented surface of genus $g$ with $n+1$ boundary components. In \cite{Goldman1986}, Goldman defined a Lie bracket $[-,-]^G$ on the $\mathbb{K}$-span of conjugacy classes in the fundamental group $\pi_{g,n+1} \defeq \pi_1(\Sigma_{g,n+1})$,
\[
    |\mathbb{K}\pi_{g,n+1}| \defeq \mathbb{K}\pi_{g,n+1}/[\mathbb{K}\pi_{g,n+1}, \mathbb{K}\pi_{g,n+1}].
\]
Later on, in \cite{Turaev1991SkeinQO}, Turaev constructed a Lie cobracket on the space $\mathbb{K}\pi/ \mathbb{K} \textbf{1}$, where $\textbf{1}$ stands for the homotopy class of a contractible loop. This construction can be refined into a cobracket $\delta^\gamma$ on $|\mathbb{K}\pi|$ given a trivialization of the tangent bundle of the surface $\gamma$, that is, a \textit{framing}. The Goldman bracket and Turaev cobracket are compatible in the sense that $(|\mathbb{K}\pi_{g,n+1}|, [-,-]^G, \delta^\gamma)$ describes a Lie bialgebra structure. In \cite{Flor2}, the authors prove that $\mathrm{KRV}$ roughly matches the automorphism group of the associated graded\footnote{With respect to the weight filtration given in Definition \ref{wght_pi}.} Goldman-Turaev Lie bialgebra of a thrice-punctured sphere with respect to the \textit{adapted} framing (see Section \ref{GT_bialg}). 

Our overarching goal is to reproduce the \cite{AlekToro} injection of $\mathrm{GRT}$ into $\mathrm{KRV}$ functorially. To phrase this more precisely, let $\mathbf{GoTu}$ be the category in which objects consist of triplets
\begin{equation}
    (H, [-,-], \delta),
\end{equation}
where $H$ is a Hopf algebra and $[-,-]: |H| \otimes |H| \rightarrow |H|$ and $\delta: |H| \rightarrow |H| \otimes |H|$ are linear maps. Morphisms in $\mathbf{GoTu}$ consist of algebra maps that intertwine with the corresponding brackets and cobrackets. Let us also denote
$\mathcal{C}_{g,n+1}^\gamma \defeq (\mathrm{gr}(\mathbb{K}\pi_{g,n+1}), [-,-]^G_{\mathrm{gr}}, \delta_{\mathrm{gr}}^{\gamma})$ for the associated graded of the Goldman-Turaev bialgebra on $\Sigma_{g,n+1}$ with respect to the framing $\gamma$. The injection of $\mathrm{GRT}$ into $\mathrm{KRV}$ can then be phrased as a morphism 
\begin{equation*}
    \mathrm{Aut}_{\mathrm{Operads}}(\mathbf{PaCD}) \hookrightarrow \mathrm{Aut}_{\mathbf{GoTu}}(\mathcal{C}_{0,3}^{\mathrm{adp}}).
\end{equation*}
From this presentation, it seems natural to guess that there exists a functor $\Lambda: \mathrm{Operads} \rightarrow\mathbf{GoTu}$ such that $\Lambda(\mathbf{PaCD})=\mathcal{C}_{0,3}^\mathrm{adp}$. The main result of this paper is to construct such a functor for a suitable variant of operads. Namely, we consider the category $\mathrm{OpR}^\Delta_f(n)$ of (suitably marked) pairs $(\mathcal{P}, \mathcal{M})$, where $\mathcal{M}$ is an operadic right $\mathcal{P}$-module over the category of small $\mathrm{CoAlg}$-enriched categories. Whenever $(\mathcal{P},\mathcal{M})\in \mathrm{OpR}^\Delta_f(n)$, certain objects of $\mathcal{M}$ can be regarded as \textit{parenthesizations}, and the label $n$ indicates the endomorphism coalgebras of parenthesizations of length $n$ are constrained. Furthermore, objects in $\mathrm{OpR}^\Delta_f(n)$ must be enhanced with data corresponding to a \textit{framing} $\gamma$ (see Definition \ref{op_delta}). For simplicity, we will omit all extra markings on objects in $\mathrm{OpR}^\Delta_f(n)$ for the Introduction. Let $\mathbf{PaCD}^f$ be Gonzalez' \textit{framed} analogue of $\mathbf{PaCD}$ (see Definition \ref{def_PaCD}). The pair $(\mathbf{PaCD}^f, \mathbf{PaCD}^f)$ naturally lifts to an object in $\mathrm{OpR}^\Delta_f(n)$, and we prove that
\begin{theorem}[Theorem \ref{big_adp}]
\label{intro_D}
    For $n\geq 1$, there exist functors $\Lambda_{n+1}: \mathrm{OpR}^\Delta_f(n+1) \rightarrow \mathbf{GoTu}$ such that $\Lambda_{n+1}(\mathbf{PaCD}^f, \mathbf{PaCD}^f)$ is isomorphic to the associated graded of the Goldman-Turaev bialgebra of the compact surface of genus $0$ with $n+1$ boundary components.
\end{theorem}
The construction also applies to surfaces of higher genera. Following \cite{gonz}, let $(\mathbf{PacD}^f, \mathbf{PacD}_g^f)$ denote the operad module of parenthesized framed genus $g$ chord diagrams. 
\begin{theorem}[Theorem \ref{PaCD_is_GT}]
\label{intro_D1}
    For $n\geq 1$, there is a natural isomorphism between $\Lambda_{n+1}(\mathbf{PaCD}^f, \mathbf{PaCD}_g^f)$ and the associated graded of the Goldman-Turaev bialgebra of the compact surface of genus $g$ with $n$ boundary components.
\end{theorem}
Analogously, in \cite{Flor1}, the authors define the higher genus groups $\mathrm{KRV}^\gamma_{g,n+1}$, and prove these act on $\mathcal{C}^\gamma_{g,n+1}$ by bialgebra automorphisms. From the previous theorem, we immediately obtain the following corollary:
\begin{theorem}[Theorem \ref{GRT_in_KRV}]
\label{intro_GRT}
    For $n \geq 0$, there exist natural (non-trivial) homomorphisms
    \begin{equation}\label{main_intro}
        \mathrm{GRT}_g^\gamma = \mathrm{Aut}(\mathbf{PaCD}^f,\mathbf{PaCD}^f_g)  \rightarrow \mathrm{Aut}(\mathcal{C}_{g,n+1}^\gamma) =: \overline{\mathrm{KRV}}_{g,n+1}^\gamma.
    \end{equation}
\end{theorem}
Here, $\mathrm{GRT}_g$ is the genus $g$ version of the Grothendieck-Teichm\"uller group as defined in \cite{gonz}, and $\mathrm{GRT}_g^{\gamma}$ is the subgroup preserving the framing $\gamma$. By results of \cite{Flor3}, the group $\overline{\mathrm{KRV}}_{g,n+1}^\gamma$ is roughly an extension of $\mathrm{KRV}_{g, n+1}^\gamma$ by inner derivations (see Proposition \ref{acc_ext}).

We think of the functors $\Lambda_{n+1}$ as extracting the Goldman-Turaev Lie bialgebra of an $n+1$-fold punctured surface (encoded by its operadic right module of configuration of points) and they are constructed as follows. We split the problem into the two steps indicated by the factorization of $\Lambda_{k}$ as
\begin{equation*}
   \begin{tikzcd}
        & \mathbf{Fox} \arrow[dr, "\Psi"] & \\
    \mathrm{OpR}^\Delta_f(k) \arrow[ur, "\mathrm{IC}_k"] \arrow[rr, "\Lambda_k"] & & \mathbf{GoTu}.
   \end{tikzcd}
\end{equation*}

The category $\mathbf{Fox}$ admits two equivalent descriptions, each of which is useful for constructing one of the two functors $\mathrm{IC}_k$ and $\Psi$. On the one hand, an object in  $\mathbf{Fox}$ consists of a free Lie algebra $\mathfrak{g}$ together with a quasi-derivation $q \colon U\mathfrak{g} \to U\mathfrak{g}$ in the sense of \cite{Mass}. The functor $\Psi$ is then essentially given by the work of Massuyeau and Turaev, deriving bracket operations from Fox pairings \cite{MasTu}, and cobracket operations from quasi-derivations \cite{Mass}. On the other hand, we show that $\mathbf{Fox}$ is equivalent to the category of \textit{relative} abelian Lie algebra extensions $\tilde{\mathfrak{g}}$ of the form  
\[
0 \rightarrow U\mathfrak{g} \to \tilde{\mathfrak{g}} \to \mathfrak{g} \oplus \mathfrak{g} \rightarrow 0.
\]
These are still abelian extensions in the classical sense, in that the previous sequence is short-exact. They are \textit{relativized} by imposing compatibility with a section over $\mathfrak{g} \subset \mathfrak{g} \oplus \mathfrak{g}$ (see Definition \ref{def_relE}). The classical correspondence between abelian extensions and Lie 2-cocycles generalizes to the relative setting, where our notion of relative Lie cohomology is now derived from \cite{KaYu}. In particular, we obtain the following cohomological description of quasi-derivations.
\begin{theorem}[Theorem \ref{quasi-iso}]
    There is a natural bijection between $H^2(\mathfrak{g} \oplus \mathfrak{g}, \mathfrak{g} ; U \mathfrak{g})$ and the space of quasi-derivations on $U\mathfrak{g}$ up to exact ones.
\end{theorem}
Exact quasi-derivations (see Definition \ref{def_ex_qder}) are defined such that they induce zero bracket and cobracket in $\mathbf{GoTu}$. The functors $\mathrm{IC}_k$ capitalize on this second cohomological description of $\mathbf{Fox}$. Although our approach is entirely algebraic, the maps $\mathrm{IC}_k: \mathrm{OpR}^\Delta_f(k) \rightarrow \mathbf{Fox}$ have the following topological interpretation. Thinking of elements in $\mathbf{Fox}$ as intersection pairings, we can construct such a pairing for an $k$-fold punctured surface by considering configuration spaces of (up to $k+2$) points on the unpunctured surface. The collection of (the fundamental groupoids of) such configuration spaces are exactly an object in $\mathrm{OpR}^\Delta_f(k)$. In the language of \cite{3flav}, the functor $\mathrm{IC}_k$ extracts the \textit{intersection contexts} of the $k$-punctured surface from which the Goldman-Turaev bialgebra can be realized as the string bracket and cobracket (in the sense of string topology). We do not explore this connection further, and refer the reader to \cite{3flav} for more details.

Lastly, the reader may be curious about how the morphism of equation $(\ref{main_intro})$ factors through the category $\mathbf{Fox}$. Even at this intermediate stage, we can almost recover the entirety of $\overline{\mathrm{KRV}}_{g,n+1}^\gamma$. Let us define $\mathcal{C}_{g,n+1}^\gamma \defeq \mathrm{IC}_{n+2}(\mathbf{PaCD}^f, \mathbf{PaCD}^f_{g})$. We prove
\begin{theorem}[Theorem \ref{diff_char}]
\label{intro_KRV}
    There exists an isomorphism between $\mathrm{Aut}_{\mathbf{Fox}}(\mathcal{C}_{g,n+1}^\gamma)/\mathbb{K}$ and $\overline{\mathrm{KRV}}_{g,n+1}^\gamma$.
\end{theorem}

\subsection{Organization of the paper}

In Chapter 2, we review Fox pairings and quasi-derivations following Massuyeau and Turaev in \cite{MasTu}. We define notions of \textit{exact} Fox pairings and \textit{exact} quasi-derivations, and we prove that the bracket associated to an exact Fox pairing (in the sense of \cite{MasTu}) and the cobracket associated to an exact quasi-derivation (in the sense of \cite{Mass}) are both null. 

In Chapter 3, we review relative Lie cohomology as defined in \cite{KaYu}. Given a free Lie algebra $\mathfrak{f}$, we prove that elements in $H^2(\mathfrak{f} \oplus \mathfrak{f}, \mathfrak{f}; U\mathfrak{f})$ are in bijection with both \textit{relative} Lie algebra extensions and the set of quasi-derivations on $U\mathfrak{f}$ modulo exact ones.

In Chapter 4, we derive a special category $\mathrm{OpR}^\Delta_f(n)$ from $\mathrm{OpR}\, \mathbf{Cat}(\mathbf{CoAlg})$, the category of operadic right modules over small $\mathrm{CoAlg}$-enriched categories, and sketch how one can tentatively produce Lie bialgebras out of its objects. As a special instance of this, we compute the bracket and cobracket operations that can be extracted from the $\mathbf{PaCD}^f$-operad module $\mathbf{PaCD}^f_g$.

In Chapter 5, we prove the main results of this paper. After reviewing the higher genus graded Kashiwara-Vergne groups following \cite{Flor1}, we prove the bracket and cobracket operations functorially derived from $\mathbf{PaCD}^f_g$ coincide with those of the Goldman-Turaev bialgebra of a higher genus surface. As a consequence of this fact, we obtain a non-trivial map from Gonzalez' higher genus versions of the graded Grothendieck-Teichm\"uller group into the corresponding higher genus Kashiwara-Vergne groups.

\subsection*{Acknowledgments}

The author is grateful to Anton Alekseev, Yusuke Kuno, and Muze Ren, for fruitful discussions and for sharing preliminary versions of their work. The author would like to extend special thanks to his advisor Florian Naef, for his guidance along the course of this project and for valuable comments on earlier versions of this text.

\subsection*{Conventions}
\begin{itemize}
    \item Throughout this paper, we work with respect to a field of characteristic zero $\mathbb{K}$.
    \item Given an associative algebra $A$ over $\mathbb{K}$, we denote
    \begin{equation*}
        |A|\defeq A/\langle (ab-ba)\, |\,  a,b\in A \rangle,
    \end{equation*}
    the $\mathbb{K}$-vector space of cyclic words on $A$.
    \item We will only be dealing with cocommutative, involutive Hopf algebras. We typically use $S$ and $\varepsilon$ to denote the antipode and augmentation maps, respectively. We will use Sweedler notation when dealing with the coproduct:
    \begin{equation*}
        \Delta(a)=a \otimes a''.
    \end{equation*}
    \item For a given Lie algebra $\mathfrak{h}$, the associative algebra $U\mathfrak{h}$ has the structure of a Hopf algebra. We identify $\mathfrak{h}$ with the degree zero component of $U\mathfrak{h}$. For any element $x \in \mathfrak{h}$, the coproduct, antipode, and augmentation map are given by
    $\Delta(x)=1 \otimes x + x \otimes 1$, $S(x)=-x$, and $\varepsilon(x)=0$, respectively.
    \item If $V$ and $W$ are filtered vector spaces, their tensor product carries the filtration
    \begin{equation*}
        V \otimes W(m)=\bigoplus_{m=k+l}V(k) \otimes W(l).
    \end{equation*}
    Unless it can lead to confusion, we will not indicate the completion with respect to this filtration in our notation, so that $V \hat{\otimes} W=V \otimes W$.
    \item We denote by $\mathbf{Cat}(\mathbf{CoAlg)}$ the category of small $\mathrm{CoAlg}$-enriched categories. We consider it a symmetric monoidal category with respect to the product given by:
    \begin{itemize}
        \item $\mathrm{Ob}(\mathcal{C} \otimes \mathcal{C}') \defeq \mathrm{Ob}(\mathcal{C}) \times \mathrm{Ob}(\mathcal{C}')$;
        \item $\mathrm{mor}_{\mathcal{C}\otimes \mathcal{C}}((c,c'), (d,d'))\defeq \mathrm{mor}_\mathcal{C}(c,d)\, \hat{\otimes}\, \mathrm{mor}_\mathcal{C'}(c',d')$,
    \end{itemize}
    where the tensor product is completed with respect to the descending filtrations defined by powers of the augmentation ideal.
    
\end{itemize}

\section{Fox pairings and quasi-derivations}\label{fox}

\begin{definition}    
    Let $\mathbf{GoTu}$ be the category with objects given by triplets
    \begin{equation*}
        ( H, [-,-], \delta ),
    \end{equation*}
    where $H$ is a cocommutative, involutive Hopf algebra. The other components are $\mathbb{K}$-linear maps 
    \begin{equation*}
        [-,-]: |H| \otimes |H| \rightarrow |H|, \hspace{1 em} \delta:|H| \rightarrow |H| \otimes |H|
    \end{equation*}
    defined on the space of cyclic words in $H$. Morphisms in $\mathbf{GoTu}$ consist of Hopf algebra maps that preserve the corresponding operations. Concretely, the space
    \begin{equation*}
        \mathbf{GoTu}((H, [-,-], \delta),(H', [-,-]',\delta'))    
    \end{equation*}
    consists of Hopf algebra maps $f: H \rightarrow H'$ that define commutative diagrams\footnote{ We do not distinguish between the Hopf morphism $f: H \rightarrow H'$ and its induced map $f:|H| \rightarrow |H'|$.}
    \begin{equation*}
        \begin{tikzcd}
            \vert H \vert \otimes \vert H \vert \arrow[r, "{[-,-]}"] \arrow[d, "f \otimes f"] & \vert H \vert \arrow[d, "f"]\\
            \vert H' \vert \otimes \vert H' \vert \arrow[r, " {[-,-]'}"]& \vert H'
        \end{tikzcd}, \hspace{3 em}
        \begin{tikzcd}
            \vert H \vert \arrow[r, "\delta"] \arrow[d, "f"] & \vert H \vert \otimes \vert H \vert \arrow[d, "f \otimes f"]\\
            \vert H' \vert \arrow[r, "\delta'"] & \vert H'\vert \otimes \vert H' \vert
        \end{tikzcd}.
    \end{equation*}
\end{definition}    
In this chapter, we define the category $\mathbf{Fox}$; its objects similarly consist of triplets $(H, \rho, q)$, where $\rho$ and $q$ are a Fox pairing and a quasi-derivation on $H$ in the sense of Massuyeau-Turaev and Massuyeau (see definitions below), respectively. Our main goal is to define a functor
\begin{equation*}
    \begin{split}
        \Psi: \hspace{2em} \mathbf{Fox} &\longrightarrow \mathbf{GoTu}\\
        (H, \rho, q) & \longmapsto (H, [-,-]^\rho, \delta_q).
    \end{split}
\end{equation*}
This is essentially a recasting of Massuyeau and Turaev's work, deriving bracket operations from Fox pairings \cite{MasTu}, and cobracket operations from quasi-derivations \cite{Mass}. Our main contribution has to do with subtleties in the definition of $\mathbf{Fox}$: while naively one could define morphisms in $\mathbf{Fox}$ as the Hopf algebra maps preserving Fox pairings and quasi-derivations on the nose, we will instead ask these operations to be preserved up to (specified) \textit{exact} terms. That is, we will define \textit{exact} Fox pairings and \textit{exact} quasi-derivations, and prove their associated brackets and cobrackets are null. The very suggestive "exactness" moniker comes from ties with Lie algebra cohomology we will explore in Chapter \ref{chap_coho}.

\subsection{Fox pairings}

In this section, we recall Massuyeau and Turaev's definition of Fox pairings and their associated brackets following \cite{MasTu}.

Let $f:H \rightarrow I$ be a map between Hopf algebras over $\mathbb{K}$. The map $f$ induces an $H$-bimodule structure on $I$ via 
\begin{equation*}
    h \cdot a \defeq f(h)a, \hspace{1em} a \cdot h \defeq af(h),
\end{equation*}
for all $h\in H$ and $a\in I$. To simplify the notation throughout this paper, we somewhat abusively abbreviate $h\cdot a=ha$, $a \cdot h=ah$. That is, we omit any mention of the map $f$ or the $H$-actions on $I$, except when it would lead to confusion. Furthermore, our notation will usually not distinguish between the Hopf algebra operations in $H$ and $I$.

\begin{definition}
     A linear map $\partial: H \rightarrow I$ is a left (resp. right) \textit{Fox derivative} if
    \[
        \partial(ab)=\partial(a) \varepsilon(b)+a\partial(b) \quad (\text{resp.}\,\, \partial(ab)=\partial(a)b+\varepsilon(a)\partial(b))
    \]
    for all $a,b \in H$.
    This definition depends on the morphism $f:H \rightarrow I$. We denote the space of left (resp. right) Fox derivatives from $H$ into $I$ mediated by $f$ by $\mathrm{LFox}_f(H,I)$ (resp. $\mathrm{RFox}_f(H,I)).$ In what follows, we will mainly be interested in the case when $H=I$, so we make special distinction of the spaces
    \begin{equation*}
        \mathrm{LFox}(H)\defeq \mathrm{LFox}_{\mathrm{id}}(H,H), \hspace{1em} \mathrm{RFox}(H)\defeq \mathrm{RFox}_{\mathrm{id}}(H,H).
    \end{equation*}
\end{definition}
When dealing with Fox derivatives in a fixed Hopf algebra, unless otherwise stated, we will always assume they are mediated by the identity map.
\begin{example}
    Given an arbitrary Hopf algebra $H$, the map $D: H \rightarrow H$ defined by $D(a)=a-\varepsilon(a)$ is both a left and a right Fox derivative.
\end{example}

\begin{example}
    Let $H$ be the free associative algebra in the generators $\{x_i\}_{i=1, \dots, n}$. Then we may define Fox derivatives solely by specifying their values on the generators. Let $\overrightarrow{\partial_i}$ (resp. $\overleftarrow{\partial_i}$) be the left (resp.right) Fox derivative uniquely characterized by setting
    \begin{equation} \label{fox_base}
        \overrightarrow{\partial_i}(x_j)=\overleftarrow{\partial_i}(x_j)=\delta_{ij}.
    \end{equation}
\end{example}
\begin{lemma}\label{LcapR=D}
    Let $H$ be the degree completion of the free associative algebra in the generators $\{x_i\}_{i=1, \dots,n}$, for $n \geq$ 2. Then
    \begin{equation*}
        \mathrm{LFox}(H) \cap \mathrm{RFox}(H)= \mathbb{K}\{D\}.
    \end{equation*}
    That is, up to scalars, $D$ is the only map that is a left and right Fox derivative at once.
\end{lemma}
\begin{proof}
    Let $\partial \in \mathrm{LFox}(H) \cap \mathrm{RFox}(H)$. We define elements $a_i \defeq \partial(x_i)$ for all $i=1, \dots n$. Notice that the equality
    \begin{equation*}
        \partial(x_ix_i)=x_i a_i=a_ix_i
    \end{equation*}
    implies that $a_i$ is a power series in the generator $x_i$; that is, $a_i \in \mathbb{K}\langle \!\langle x_i \rangle \! \rangle$. As $n \geq 2$, we may choose $j \neq i$. Then equalities
    \begin{equation*}
        \partial(x_ix_jx_i)=a_ix_jx_i=x_ix_ja_i
    \end{equation*}
    mean that $\partial(x_i)=\lambda_ix_i$ for some $\lambda_i \in \mathbb{K}$. Furthermore, since
    \begin{equation*}
        \partial(x_ix_j)=(\lambda_ix_i)x_j=x_i(\lambda_jx_j),
    \end{equation*}
    it must be that $\lambda=\lambda_i=\lambda_j$ for all $i,j$. This concludes the proof, since the map $\lambda D(a)= \lambda(a-\varepsilon(a))$ takes the same values as $\partial$ on the generators of $H$.
\end{proof}

\begin{definition}
    Given a left or right Fox derivative $\partial: H \rightarrow I$, we define its \textit{transpose} by $\partial^t \defeq S_I \circ \partial \circ S_H$.
\end{definition}
\begin{lemma}\label{Fox-transpose}
    If $\partial: H \rightarrow I$ is a left (resp. right) Fox derivative, then its transpose $\partial^t$ is a right (resp. left) Fox derivative.
\end{lemma}
\begin{proof}
    Suppose $\partial \in \mathrm{LFox}_f(H,I)$, then
    \begin{equation*}
        \begin{split}
            \partial^t(ab)&=S \partial(S(b)S(a))\\
            &=S(\partial(S(b))\varepsilon(S(a)) + S(b) \partial(S(a)))\\
            &=\partial^t(b) \varepsilon(a) + \partial^t{\partial}(a)b,
        \end{split}
    \end{equation*}
    for any $a,b \in H$. The proof is analogous when dealing with a right Fox derivative.
\end{proof}

\begin{definition}
    A \textit{Fox pairing} from $H$ into $I$ is a bilinear map $\rho: H \times H \rightarrow I$ that is a left Fox derivative with respect to the first variable and a right Fox derivative with respect to the second variable. In formulas,
    \[
        \begin{split}
            \rho(a_1 a_2, b) &= \rho(a_1, b) \varepsilon(a_2) + a_1 \rho(a_2, b),\\
            \rho(a, b_1b_2) &= \rho(a, b_1) b_2 + \varepsilon(b_1)\rho(a, b_2).
        \end{split}
    \]
    We denote the space of Fox pairings associated to $f$ by $\mathrm{Fox}_f(H,I)$. We will mainly deal with the instance
    \begin{equation*}
        \mathrm{Fox}(H)\defeq \mathrm{Fox}_{\mathrm{id}}(H,H).
    \end{equation*}
\end{definition}

\begin{definition}\label{fox-skew-def}
    Let $\rho \in \mathrm{Fox}_f(H, I)$. The \textit{transpose} of $\rho$ is given by
    \begin{equation*}
        \rho^t(a, b) \defeq S_I \rho (S_H(b), S_H(a))
    \end{equation*}
    for all $a, b \in H$. The Fox pairing $\rho$ is \textit{skew-symmetric} if $\rho = - \rho^t$.
\end{definition}

We now move on to describing Massuyeau and Turaev's construction of the bracket associated to a Fox pairing. We first need to discuss the notion of \textit{double bracket}.

Let $H$ be a Hopf algebra. We can endow $H^{\otimes n}$ with the \textit{outer} bimodule structure, 
\[
    aub \defeq au_1 \otimes \cdots \otimes u_n b 
\]
for $u= u_1 \otimes \cdots \otimes u_2 \in H^{\otimes n}$ and $a, b \in H$; but we can also choose the \textit{inner} bimodule structure, 
\[
    a*u*b \defeq u_1b \otimes \cdots \otimes a u_n.
\]

\begin{definition}
    Let $f: H \rightarrow I$ be a Hopf algebra map. A linear map $\dbl -,-\dbr: H \otimes H \rightarrow I \otimes I $ is a \textit{double bracket} if 
    \begin{gather} \label{double}
        \dbl a, b_1 b_2 \dbr= b_1 \dbl a, b_2 \dbr + \dbl a, b_1 \dbr b_2,\\
        \dbl a_1 a_2, b\dbr = a_1 *\dbl a_2, b \dbr + \dbl a_1, b\dbr*a_2,
    \end{gather}
    where as before we have left implicit the $H$ bimodule structure on $I$. We say the double bracket is mediated by the Hopf algebra map $f$. Whenever $H=I$ and no mention of a map is made, we will assume $f= \mathrm{id}$.
\end{definition}

\begin{proposition}\label{brac_from_double}
    A double bracket $\dbl -,- \dbr:H \otimes H \rightarrow I \otimes I$ induces an operation $[-,- ]: |H| \otimes |H| \rightarrow |I|$ defined by
    \[
        [ |a|, |b| ] \defeq | \dbl a, b \dbr^{(1)}  \dbl a, b \dbr^{(2)}|,
    \]
    where we have used Sweedler notation to write the components of the double bracket as $\dbl-,- \dbr = \dbl-,-\dbr^{(1)} \otimes \dbl-,- \dbr^{(2)}$. We call $[-,-]$ the \textit{bracket} associated to $\dbl -,- \dbr$. 
\end{proposition}
\begin{proof}
        By abuse of notation, we use the same symbol to denote the map $[-,-]:H \otimes H \rightarrow |I|$. Then, for a fixed $a \in H$, equation (\ref{double}) gives
    \[
        [ a, b_1 b_2] = |\dbl a, b_2 \dbr^{(2)} b_1 \dbl a, b_2 \dbr^{(1)}| + |\dbl a, b_1 \dbr^{(2)} b_2 \dbl a, b_1\dbr^{(1)}|,
    \]
    which we can see is symmetric in $b_1$ and $b_2$. This means that commutators in the second variable are mapped to zero. We can prove an analogous statement when fixing the first variable instead.
\end{proof}

In \cite{MasTu}, to any Fox pairing $\rho \in \mathrm{Fox}_{f}(H,I)$ between involutive and cocommutative Hopf algebras, Massuyeau and Turaev associate a linear map $\dbl -,- \dbr ^\rho: H \otimes H \rightarrow I \otimes I$ given by
\begin{equation} \label{masbr}
        \dbl a, b \dbr^{\rho}= b'S(\rho(a'', b'')')a' \otimes \rho(a'',b'')''.
    \end{equation}
The authors then prove that $\dbl -,- \dbr^\rho$ is a double bracket whenever $\rho$ is skew-symmetric. This indirectly gives us a way to construct brackets out of skew-symmetric Fox pairings. However, the assignment $\rho \mapsto [-,-]^\rho$ is well defined even without the skew-symmetry assumption; the following proposition emphasizes this fact:
\begin{proposition}[\cite{MasTu}, Lemma 6.2]
\label{brac_from_fox}
    Let $f:H \rightarrow I$ be a morphism between cocommutative, involutive Hopf algebras. Then there is a map
    \begin{equation*}
        \begin{split}
            \mathrm{Fox}_f(H, I) &\rightarrow \mathrm{Hom}_\mathbb{K}(|H| \otimes |H|, |I|)\\
        \rho & \mapsto [-,-]^\rho
        \end{split},
    \end{equation*}
    where $[|a|, |b| ]^\rho=|b'S(\rho(a'', b'')')a'\rho(a'',b'')''|$.
\end{proposition}
\begin{proof}
    We need to show that the bracket $[-,-]^\rho$ is well defined on the quotient of cyclic words. Let $a,b,c \in H$. By abuse of notation, we denote the map $[-,-]^\rho:H \otimes H \rightarrow |I|$ by the same symbol. Notice that, since $\rho$ is a left Fox derivative in its first variable,
    \begin{equation*}
        \begin{split}
            [ab, c] &= |c' S(a''\rho(b'', c'')')a'b'a'''\rho(b'', c'')'' + c'S(\varepsilon(b'') \rho(a'', c'')')a'b'\epsilon(b''')\rho(a'', c'')''|\\
            &=|c'S(\rho(b'', c'')')b'a\rho(b'', c'')''| +|c'S(\rho(a'', c'')')a'b\rho(a'', a'')|.
        \end{split}
    \end{equation*}
    Since the final equality is symmetric in $a$ and $b$, commutators in the first variable must be mapped to zero. We can analogously show $[a, bc-cb]=0$.
\end{proof}
\begin{remark}
    The bracket $[-,-]^\rho$ is skew-symmetric whenever $\rho$ is.
\end{remark}

\begin{definition}\label{FoxP_exact}
    Let $\rho \in \mathrm{Fox}_f(H, I)$. We say $\rho$ is \textit{exact} if it is expressible as
    \begin{equation}
        \rho(a,b)= \partial_L(a)D(b)+D(a)\partial_R(b), \ a,b \in H,
    \end{equation}
    where $\partial_L \in \mathrm{LFox}_f(H,I)$ and $\partial_R \in \mathrm{RFox}_f(H,I)$. This motivates us to define a map
\begin{equation}\label{tau_fox}
    \begin{split}
         \tau: \hspace{1em} \mathrm{LFox}_f(H,I) \oplus \mathrm{RFox}_f(H,I) & \longrightarrow \mathrm{Fox}_f(H,I)\\
         \partial_L \oplus \partial_R & \longmapsto ((a, b) \mapsto\partial_L(a)D(b) + D(a)\partial_R(b)).
    \end{split}
\end{equation}
\end{definition}

\begin{proposition}\label{FP_exact=null}
    Let $\rho \in \mathrm{Fox}_f(H, L)$. The bracket $[-,-]^\rho: |H| \times |H| \rightarrow |I|$ is identically zero whenever $\rho$ is exact.
\end{proposition}
\begin{proof}
    Let $\rho$ be an exact Fox pairing, with
    \[
        \rho(a, b)= \partial_L(a)D(b) + D(a) \partial_R(b)
    \]
    for all $a, b \in H$. Because the assignment $\rho \mapsto [ -,- ]^{\rho}$ is linear in $\rho$, we will first prove $[-,-]^{\gamma}=0$ for the Fox pairing $\gamma(a, b)= \partial_L(a)D(b)$. We begin by computing
    \[
        \Delta(\gamma(a'', b''))= \partial_L (a'')' b'' \otimes \partial_L(a'')'' b'''- \partial_L(a'')'\varepsilon(b'')\otimes \partial_L(a'')''.
    \]
    Now we can plug this into the definition of $[ -,-]^{\gamma}$:
    \begin{equation*}
        \begin{split}
        [ a, b ] ^{\gamma}&= |b'S(b'')S(\partial_L(a'')')a' \partial_L(a'')''b'' -b'\varepsilon (b'') S(\partial_L(a'')')a' \partial_L(a'')''|\\
        &= |S(\partial_L(a'')')a' \partial_L(a'')''b| - |b S(\partial_L (a'')')a'  \partial_L(a'')''|=0\\
        \end{split}
    \end{equation*}
    We can analogously show that $[ -,-]^{D \cdot \partial_R}=0$ for any right Fox derivative $\partial_R$.
\end{proof}

\begin{definition}\label{inner}
    A Fox pairing $\rho: H \rightarrow H$ is \textit{inner} if there exists an $e\in H$ such that
    \begin{equation*}
        \rho(a,b)=(a-\varepsilon(a))e(b-\varepsilon(b)) = D(a) e D(b)
    \end{equation*}
    for all $a,b\in H$.
\end{definition}
In \cite[Lemma 6.3]{MasTu}, the authors prove that the bracket $[-,-]^\rho$ associated to any inner Fox pairing is null. Notice that we may write
\begin{equation*}
    \begin{split}
        D(a)eD(b)=\left(\frac{1}{2} D(a) e \right)D(b)+D(a)\left(\frac{1}{2}eD(a) \right),
    \end{split}
\end{equation*}
where the maps
\begin{equation*}
    a\mapsto \frac{1}{2}D(a)e \quad \text{and} \quad a\mapsto \frac{1}{2}eD(a)
\end{equation*}
are left and right Fox derivatives, respectively. It follows that all inner Fox parings are exact.

\subsection{Quasi-derivations}\label{quasisec}

In this section we follow \cite{Mass} in order to produce a cobracket in $|H|$ from a quasi-derivation in $H$. Similarly to the previous section, we fix a Hopf algebra map $f: H \rightarrow I$ and capitalize on the $H$-bimodule structure it induces on $I$ with our notation.

\begin{definition}
    A linear map $q: H \rightarrow I$ is a \textit{quasi-derivation} if the map
    \begin{equation} \label{quasi}
        \begin{split}
            \rho : \hspace{1em} H \times H & \rightarrow I\\
            (a,b) & \mapsto q(a)b + aq(b) - q(ab)
        \end{split}
    \end{equation}
    is a Fox pairing in $\mathrm{Fox}_f(H,I)$. We say $q$ is a quasi-derivation associated to the Fox pairing $\rho$, and write $q\in \mathrm{Qder}_f(\rho)$. As before, we will mainly concern ourselves with quasi-derivations in
    \begin{equation*}
        \mathrm{Qder}(\rho) \defeq \mathrm{Qder}_\mathrm{id}(\rho).
    \end{equation*}
    Whenever we do not want to fix a specific Fox pairing, we will write $\mathrm{Qder}_f(H,I)$ for the space of quasi-derivations mediated by the map $f$. We also denote $\mathrm{Qder}(H) \defeq \mathrm{Qder}_{\mathrm{id}}(H)$.
\end{definition}
\begin{example}
    Given a left (resp. right) Fox derivative $\partial_L$ (resp. $\partial_R$), the map $q(a)= \partial_L(a) + \partial_R(a)$ is a quasi-derivation associated to the Fox pairing
    \[
        \rho(a,b)=-(\partial_L(a)D(b) + D(a)\partial_R(b)).
    \]
\end{example}

Similarly to before, we will consider \textit{skew-symmetric} quasi-derivations.
\begin{definition}\label{qder-skew-def}
    Let $q \in \mathrm{Qder}_f(\rho)$. The \textit{transpose} $q^t$ of $q$ is given by
    \begin{equation*}
        q^t(a) \defeq S_I q (S_H(a)) \hspace{1em} \text{for all }a\in H.
    \end{equation*}
    We say $q$ is \textit{skew-symmetric} if $q= -q^t$.
\end{definition}
\begin{lemma}[\cite{Mass}, Lemma 2.3]
    If $q: H \rightarrow I$ is a quasi derivation associated to a Fox pairing $\rho: H \times H \rightarrow I$, then its transpose $q^t$ is a quasi-derivation associated to $\rho^t$.
\end{lemma}

Our immediate goal is to define a cobracket $\delta_q: |H| \rightarrow |I| \otimes |I|$ associated to any quasi-derivation $q$ mediated by a Hopf algebra map $f: H \rightarrow I$. In \cite{Mass}, Massuyeau defines a linear map $d_q: H \rightarrow I \otimes I$ by
\begin{equation}
    d_q(a)= a'S(q(a'')') \otimes q(a'')''
\end{equation}
for any quasi-derivation $q$. The author then proves that the map
\begin{equation}\label{masdelta_def}
    \begin{split}
        \delta_q':\quad H & \rightarrow I \otimes I\\
        a & \mapsto d_q(a) - P d_q(a)
    \end{split}
\end{equation}
induces a linear map $|\delta_q'|: |H| \rightarrow |I| \otimes |I|$ whenever the quasi-derivation $q$ is skew-symmetric (cf. \cite[Lemma 2.6]{Mass}). Here, $P: I \otimes I \rightarrow I \otimes I$ is the permutation map given by $P(a\otimes b)=b\otimes a$ for all $a, b \in I$.  We need to extend this result to non-symmetric quasi-derivations. For this reason, we deviate from \cite{Mass} and instead define
\begin{equation}\label{delta_def}
    \delta_q(a) \defeq d_q(a)+Pd_{q^t}(a).
\end{equation}
Notice that $\delta_q'=\delta_q$ whenever $q$ is skew-symmetric.

\begin{proposition} \label{cobracyc}
    Suppose $q: H \rightarrow I$ is a quasi-derivation between involutive, cocommutative Hopf algebras. Then $\delta_q$ descends to a map $\delta_q: |H| \rightarrow |I| \otimes |I|$, which we call the cobracket associated to $q$. 
\end{proposition}
\begin{proof}
    By abuse of notation, we will use the same symbol for the map $\delta_q: H \rightarrow |I| \otimes |I|$. We need to show that $\delta_q(ab -ba)=0$ for any $a,b \in H$. Let $q \in \mathrm{Qder}_f(\rho)$, so that
    \begin{equation*}
        q(ab)= q(a)b+aq(b) + \rho(a,b).
    \end{equation*}
    Note this assumption also implies that $q^t$ is associated to $\rho^t$. On the one hand,
    \begin{equation*}
        \begin{split}
            \delta_q(a b) &= d_q(ab) + P(d_{{q}^t}(ab))\\
            &= a'b'S(q(a''b'')') \otimes q(a''b'')''+ P(d_{q^t}(ab))\\
            &=a'b'S(q(a'')'b'')\otimes q(a'')''b'''+a'b'S(a''q(b'')') \otimes a'''q(b'')'' + a'b' S(\rho(a'',b'')') \otimes \rho(a'', b'')'' + P(d_{q^t}(ab))\\
            &= a' S(q(a'')') \otimes q(a'')'' + b' S(q(b'')') \otimes a q(b'')'' + P(a' S(q^t(a'')') \otimes q^t(a'')'' + b' S(q^t(b'')') \otimes a q^t(b'')'')\\
            &+ a' b' S(\rho(a'',b'')') \otimes \rho(a'',b'')'' + P(a' b' S(\rho^t(a'',b'')') \otimes \rho^t(a'',b'')'').
        \end{split}
    \end{equation*}
    Notice the first two terms in the last equality are symmetric under the exchange of $a$ and $b$, and likewise for the third and fourth terms. This means
    \begin{equation}\label{qauxf}
        \begin{split}
            \delta_q(ab-ba)&=a'b' S(\rho(a'',b'')') \otimes \rho(a'',b'')'' -b'a' S(\rho(b'',a'')') \otimes \rho(b'', a'')''\\
            &+P(a'b' S(\rho^t(a'',b'')') \otimes \rho^t(a'',b'')'' -b'a' S(\rho^t(b'',a'')') \otimes \rho^t(b'', a'')'').
        \end{split}
    \end{equation}
    Now it is enough to observe, following \cite[Lemma 2.6]{Mass}, that
    \begin{equation*}
        y'S(\rho(x'', y'')')x' \otimes \rho(x'', y'')''=P(x' S(\rho^t(y'', x'')')y' \otimes \rho^t(y'', x'')'')
    \end{equation*}
    for any $x,y \in H$. This implies that the first and fourth term in equation (\ref{qauxf}) cancel each other out, and similarly for the second and third terms.
\end{proof}

Now we will define \textit{exact} quasi-derivations; these naturally correspond to exact relative Lie cocycles, as we will see in Section \ref{coho}.
\begin{definition}\label{def_ex_qder}
    A quasi-derivation $q:H \rightarrow I$ is \textit{exact} if it is expressible as
    \begin{equation}
        q(a)= \partial_L(a) + \partial_R(a), \, a \in H,
    \end{equation}
    for some $\partial_L \in \mathrm{LFox}_f(H, I)$ and $\partial_R \in \mathrm{RFox}_f(H, I)$. Notice that in this case, $q$ is associated to the Fox pairing
    \begin{equation*}
        (a,b) \longmapsto -\partial_L(a)D(b) - D(a)\partial_R(b)=-\tau(\partial_L \oplus \partial_R)(a, b). 
    \end{equation*}
    We consolidate this information by defining a map
    \begin{equation} \label{mu_qder}
        \begin{split}
            \mu: \hspace{1em} \mathrm{LFox}_f(H, I) \oplus \mathrm{RFox}_f(H,I) & \longrightarrow \mathrm{Qder}_f(H,I)\\
            \partial_L \oplus \partial_R & \longmapsto(a \mapsto\partial_L(a) + \partial_R(a)).
        \end{split}
    \end{equation}
\end{definition}

We aim to prove that the cobracket associated to exact quasi-derivations is null. We will first need an auxiliary lemma:
\begin{lemma} \label{Fox_transpose_alt}
    Let $\partial \in \mathrm{LFox}_f(H,I)$ (resp. $\partial \in \mathrm{RFox}_f(H,I)$), where both $H$ and $I$ are involutive Hopf algebras. Then  
    \begin{equation*}
        \partial^t(a)=-S\partial(a'')a' \, \,(\text{resp.}\ \partial^t(a)= -a' S\partial(a''))
    \end{equation*}
    for any $a \in H$.
\end{lemma}
\begin{proof}
    Let $\partial: H \rightarrow I$ be a left Fox derivative. For any $a \in H$, $\epsilon(a)1= S(a')a''$, so that
    \begin{equation*}
        \begin{split}
            0=\partial(\varepsilon(a)1)&=\partial(S(a')a'')\\
            &=S(a')\partial(a'') + \varepsilon(a'')\partial(S(a'))\\
            &=S(a')\partial(a'') + \partial(S(a)).
        \end{split}
    \end{equation*}
    We get the desired result by evaluating the last equality on the antipode $S_I$. The proof is analogous when dealing with right Fox derivatives.
\end{proof}

\begin{proposition}  \label{big2}
    Let $q: H \rightarrow I$ be an exact quasi-derivation. Then its associated cobracket $\delta_q : |H| \rightarrow |I| \otimes |I|$ is identically zero.
\end{proposition}
\begin{proof}
    Let $q = \partial_L + \partial_R$ be an exact quasi-derivation, with $\partial_L \in \mathrm{LFox}_f(H, I)$ and $\partial_R \in \mathrm{RFox}_f(H, I)$. We can write
     \begin{equation*}
         \begin{split}
             \delta_q &= d_q + P(d_{q^t})\\
             &= [d_{\partial_L}+ P(d_{\partial_L^t})]+[d_{\partial_R}+P(d_{\partial_R^t})].
         \end{split}
     \end{equation*}
     Looking at the preceding equation, notice we will finish the proof if we can show $d_{\partial}+ P(d_{\partial^t})=0$ for any left or right Fox derivative $\partial$. Suppose $\partial$ is a left Fox derivative, then Lemma \ref{Fox_transpose_alt} implies that, given $a \in H$,
     \begin{equation*}
         d_{\partial^t}(a)=-\partial(a'')' \otimes S(\partial(a'')')a'= -P(d_\partial(a)),
     \end{equation*}
     where the last equality holds after the projection $H \otimes H \to |H| \otimes |H|$.
     Showing the preceding equality is analogous for right Fox derivatives.
\end{proof}

\subsection{Brackets and cobrackets}

\begin{definition}
    Let $\mathbf{Fox}$ be the category consisting of objects of the form $( \mathfrak{h}, q \oplus \rho)$, where $\mathfrak{h}$ is a Lie algebra, $\rho \in \mathrm{Fox}( U\mathfrak{h})$, and $q \in \mathrm{Qder}(-\rho)$. A morphism between two objects $(\mathfrak{h}, q \oplus \rho)$ and $(\mathfrak{h}', q' \oplus \rho')$ is specified by a Lie algebra map $f: \mathfrak{h} \rightarrow \mathfrak{h}'$ and a pair of Fox derivatives $\partial_L \oplus \partial_R \in \mathrm{LFox}_f(U\mathfrak{h},U\mathfrak{h}') \oplus \mathrm{RFox}_f(U\mathfrak{h}, U\mathfrak{h}')$, subject to the constraint\footnote{We do not distinguish between the Lie map $f: \mathfrak{h} \rightarrow \mathfrak{h}'$ and its associated algebra map $f: U\mathfrak{h} \rightarrow U\mathfrak{h}'$.}
    \small
    \begin{equation}\label{mor_fox}
        \mathbf{Fox}((\mathfrak{h}, q \oplus \rho), (\mathfrak{h}', q' \oplus \rho'))=
        \left\{ 
            \begin{tikzcd}
                \mathfrak{h} \arrow[d, "f"]\\
                \mathfrak{h}'
            \end{tikzcd},
            \partial_L \oplus \partial_R\, \middle| \, \mu(\partial_L \oplus \partial_R) \oplus\tau(\partial_L \oplus \partial_R) =(q'f-fq) \oplus (\rho'(f\otimes f)-f\rho)
        \right\},
    \end{equation}
    \normalsize
    Composition in $\mathbf{Fox}$ is defined by
    \begin{equation*}
        \left(
            \begin{tikzcd}
                \mathfrak{h} \arrow[d, "f"] \\
                \mathfrak{h}'
            \end{tikzcd},
            \partial_L \oplus \partial_R
        \right) \circ_{\mathbf{Fox}}
        \left(
            \begin{tikzcd}
                \mathfrak{h}' \arrow[d, "f'"] \\
                \mathfrak{h}''
            \end{tikzcd},
            \partial_L' \oplus \partial_R'
        \right) \defeq
        \left(
            \begin{tikzcd}
                \mathfrak{h} \arrow[d, "f' \circ f"] \\
                \mathfrak{h}''
            \end{tikzcd},
            (f'\partial_L + \partial_L'f) \oplus (f'\partial_R + \partial_R'f)
        \right).
    \end{equation*}
\end{definition}
    One can check by direct computation that the composition $\circ_\mathbf{Fox}$ is associative. The maps $\mu$ and $\tau$ are defined in equations (\ref{mu_qder}) and (\ref{tau_fox}), respectively; they yield exact quasi-derivations and exact Fox pairings from Fox derivatives. The interpretation of constraint (\ref{mor_fox}) is that a morphism $f:U\mathfrak{h} \rightarrow U\mathfrak{h}'$ in $\mathbf{Fox}$ is not required to precisely fix the structures in the respective Hopf algebras; we allow deviations up to exact terms.
\begin{remark}
    Notice that, given any $(\mathfrak{h}, q \oplus \rho) \in \mathrm{Ob}(\mathbf{Fox})$, the presence of $\rho$ is entirely redundant in the sense that if $q \in \mathrm{Qder}(-\rho)$, then $\rho$ can be recovered from $q$. We insist on this notation because it will help us translate objects in $\mathbf{Fox}$ into relative Lie cocycles (cf. Section \ref{coho})).
\end{remark}

\begin{theorem}\label{RelCo--GoTu}
    The assignment $\Psi: \mathbf{Fox} \rightarrow \mathbf{GoTu}$ defined by
    \begin{equation*}
        (\mathfrak{h}, q \oplus \rho ) \mapsto ( U\mathfrak{h}, [-,-]^{\rho},\delta_q ),
    \end{equation*}
    \begin{equation*}
         \left( 
        \begin{tikzcd}
            \mathfrak{h} \arrow[d, "f"]\\
            \mathfrak{h}'
        \end{tikzcd},\,
        \partial_L \oplus \partial_R
        \right) \mapsto
        (f: U\mathfrak{h} \rightarrow U\mathfrak{h}'),
    \end{equation*}
    on objects and on morphisms, respectively, is a functor. 
\end{theorem}
\begin{proof}
    Here, $[-,-]^\rho$ is the bracket induced by the Fox pairing $\rho$, in the sense of Proposition \ref{brac_from_fox}. Likewise, $\delta_q$ is the cobracket induced by the quasi-derivation $q$, in the sense of Proposition \ref{cobracyc}. 

    Note that if $f: \mathfrak{h} \rightarrow \mathfrak{h}'$ is a Lie algebra map, then its unique extension to an algebra map $U\mathfrak{h} \rightarrow U\mathfrak{h}'$ is also a Hopf algebra map, and it trivially descends to a map on cyclic words. For this proof, we drop the convention of leaving implicit the $U\mathfrak{h}$-bimodule structure on $U\mathfrak{h}'$. Instead, we will abbreviate $x=f(\tilde{x})$ for any $\tilde{x}\in \mathfrak{h}$.

    To show $\mathrm{\Psi}$ is well-defined , we need to prove that the map $f: |U\mathfrak{h}| \rightarrow |U\mathfrak{h}'|$ intertwines with the brackets and cobrackets of the corresponding spaces. Explicitly, we want to prove the equalities
    \begin{equation} \label{fix_brac}
        f\circ [-,-]^{\rho}=[-,-]^{\rho'} \circ f,
    \end{equation}
    \begin{equation} \label{fix_cobra}
        f \circ \delta_q=\delta_{q'} \circ f.
    \end{equation}
    By assumption, there exist Fox derivatives $r= \partial_L \oplus \partial_R \in \mathrm{LFox}_f(U\mathfrak{h}, U\mathfrak{h}') \oplus \mathrm{RFox}_f(U\mathfrak{h}, U\mathfrak{h}')$ such that
    \begin{equation} \label{delta_r1}
        (fq-q'f) \oplus(f \rho - \rho'f)= -\mu r \oplus -\tau r.
    \end{equation}
    Let $a,b \in U\mathfrak{h}$. To ease the notation, we will not distinguish between the Hopf algebra operations on $U\mathfrak{h}$ and $U\mathfrak{h}'$. Firstly, notice that the assignment $\rho \mapsto [-,-]^\rho$ is natural in the sense that $f \circ [-,-]^\rho=[-,-]^{f\rho}$. Indeed,
    \begin{equation*}
        \begin{split}            f[|\tilde{a}|,|\tilde{b}|]^{\rho}&=|f(\tilde{b}'S(\rho(\tilde{a}'',\tilde{b}'')')\tilde{a}'\rho(\tilde{a}'',\tilde{b}'')'')|\\
            &=|b'S((f\rho)(\tilde{a}'',\tilde{b}'')')a'(f\rho)(\tilde{a}'',\tilde{b}'')''|\\
            &= [|\tilde{a}|, |\tilde{b}|]^{f\rho}
        \end{split}
    \end{equation*}
    But then equation (\ref{delta_r1}) implies,
    \begin{equation*}
        f[-,-]^\rho=[-,-]^{\rho'f}-[-,-]^{\tau r}
    \end{equation*}
    The Fox paring $\tau r: U\mathfrak{h} \otimes U \mathfrak{h} \rightarrow U \mathfrak{h}'$ is exact in the sense of Definition \ref{FoxP_exact}. According to Proposition $\ref{FP_exact=null}$, $[-,-]^{\tau r}$ must be null, so that equation (\ref{fix_brac}) is satisfied.

    Notice now that
    \begin{equation*}
        \begin{split}
            fd_q(\tilde{a})&=f(\tilde{a}'S(q(\tilde{a}'')')) \otimes f(q(\tilde{a}'')'')\\
            &=a'S(fq(\tilde{a}'')') \otimes fq(\tilde{a}'')''\\
            &=d_{fq}(\tilde{a})
        \end{split}
    \end{equation*}
    for any $\tilde{a}\in U\mathfrak{h}$, and similarly for $fd_{q^t}$. Analogously to before, this implies that the assignment $q \rightarrow \delta_q$ is natural, in the sense that $f\delta_q=\delta_{fq}$. Equation (\ref{delta_r1}) now yields that
    \begin{equation*}
        f\delta_q=\delta_{q'f}-\delta_{\mu r}.
    \end{equation*}
    The quasi-derivation $\mu r:U\mathfrak{h} \rightarrow U\mathfrak{h}'$ is exact in the sense of Definition \ref{def_ex_qder}, which means that the cobracket $\delta_{\mu r} $ must be null (cf. Proposition \ref{big2}). This last statement implies that equation (\ref{fix_cobra}) is indeed satisfied.
\end{proof}

\section{Lie cohomology}\label{chap_coho}

Our goal for this chapter is to make precise the connection between the category $\mathbf{Fox}$ and Lie algebra cohomology.
We introduce the cateogry $\mathbf{RelCo}$ of \textit{relative} Lie cocycles, and the category $\mathbf{RelE}$ of \textit{relative} abelian extensions of Lie algebras. Schematically, we prove special (non-full) subcategories of $\mathbf{Fox}$, $\mathbf{RelCo}$, and $\mathbf{RelE}$ are all equivalent. Our notion of relative Lie cohomology is not related to Chevalley and Eilenberg's \cite[\S 22]{Chevalley:1948zz}, but rather to Kawazumi and Kuno's \cite{KaYu}. In this way, the equivalence $\mathbf{RelCo}\cong\mathbf{RelE}$ can be seen as a generalization of the classical result equating regular Lie cohomology with abelian extensions, which we review presently.

\subsection*{Motivation}

Let $\mathfrak{g}$ be a Lie algebra and $M$ an abelian $\mathfrak{g}$-module. Recall how a Lie algebra $\mathfrak{e}$ is an abelian extension of $\mathfrak{g}$ by $M$ if it slots into a short exact sequence:
\begin{equation*}
    \begin{tikzcd}[column sep = small]
        0 \arrow[r] & M \arrow[r] & \mathfrak{e} \arrow[r] & \mathfrak{g} \arrow[r] & 0
    \end{tikzcd}.
\end{equation*}
Two abelian extensions $\mathfrak{e}$ and $\mathfrak{e}'$ are equivalent if there exists an isomorphism $\varphi:\mathfrak{e} \rightarrow \mathfrak{e}'$ such that the following is a commutative diagram:
\begin{equation*}
        \begin{tikzcd}
        M \arrow[r, "i"] \arrow[d, "\text{id}_M"] & \mathfrak{e} \arrow[r, "p"] \arrow[d, "\varphi"]& \mathfrak{g} \arrow[d, "\text{id}_\mathfrak{g}"]\\
        M \arrow[r, "i'"]& \mathfrak{e}' \arrow[r, " p' "]& \mathfrak{g}
    \end{tikzcd}.
    \end{equation*}
We denote the set of such equivalence classes by $\mathrm{Ext}(\mathfrak{g},M)$. Given a Chevalley-Eilenberg 2-cocycle $[c] \in H^2(\mathfrak{g}, M)$, we can construct an abelian extension of $\mathfrak{g}$
\begin{equation*}
    \begin{tikzcd} [sep=small]
        0 \arrow[r]& M \arrow[r]& \mathfrak{g} \times_c M \arrow[r]& \mathfrak{g} \arrow[r]& 0,
    \end{tikzcd}
\end{equation*}
where the Lie bracket on $\mathfrak{g} \times_c M$ is given by
\begin{equation*}
    [(x_1, v_1), (x_2, v_2)]=([x_1, x_2 ], x_2\cdot v_1 -x_1 \cdot v_2 +c(x_2, x_2)).
\end{equation*}
It is a well-known fact that 
\begin{theorem}[\cite{Chevalley:1948zz}] \label{ext_coho}
    There is a one-to-one correspondence between $H^{2}(\mathfrak{g}, M)$ and equivalence classes of abelian extensions of $\mathfrak{g}$ by $M$; it is induced by the map
    \begin{equation*}
            \mathcal{D}:\quad [c] \mapsto \begin{tikzcd}[column sep=small]
                M \arrow[r] & \mathfrak{g} \times_c M  \arrow[r] & \mathfrak{g}
            \end{tikzcd}
    \end{equation*}
    from $H^{2}(\mathfrak{g}, M)$ into $\mathrm{Ext}(\mathfrak{g}, M)$.
\end{theorem}
It is understood that this correspondence is actually functorial (see \cite[\S 7.6]{weibel}. A way to view this explicitly is by phrasing Theorem \ref{ext_coho} as an equivalence between appropriately defined categories. We will do this schematically:
\begin{itemize}
    \item Let $\mathbf{Lie E}$ stand for the category of abelian Lie extensions. Its objects consist of short exact sequences of Lie algebras
    \begin{equation*}
        \begin{tikzcd}[column sep=small]
            M \arrow[r] & \mathfrak{e} \arrow[r] & \mathfrak{g}
        \end{tikzcd},
    \end{equation*}
    where $M$ is an abelian Lie algebra. Morphisms are defined in the natural way.
    \item Let $\mathbf{LieCo}$ stand for the category of Lie cohomology. Its objects consist of triplets $(\mathfrak{g}, M, c)$, where $M$ is an abelian $\mathfrak{g}$-module and $c\in Z^2(\mathfrak{g},M)$ is a Chevalley-Eilenberg cochain in the sense of \cite{Chevalley:1948zz}. Morphisms are also triplets
    \begin{equation*}
        \left(\begin{tikzcd}[row sep= scriptsize]
            \mathfrak{g} \arrow[d, "G"] \\
            \mathfrak{g}'
        \end{tikzcd}, 
        \begin{tikzcd}[row sep= scriptsize]
            M \arrow[d, "\alpha"] \\
            M'
        \end{tikzcd}, r\right),
    \end{equation*}
    where $r \in C^1(\mathfrak{g},M')$ satisfies $dr=\alpha c - c' G$.
\end{itemize}

\begin{theorem}
    The functor $\mathcal{D}: \mathbf{LieCo} \rightarrow \mathbf{LieE}$ defined on objects and morphisms, respectively, by
    \begin{equation*}
        (\mathfrak{g}, M, c)  \longmapsto
             \left(\begin{tikzcd}[column sep=small]
                 M \arrow[r] & \mathfrak{g} \times_c M \arrow[r] & \mathfrak{g}
             \end{tikzcd}\right) ,
    \end{equation*}
    \begin{equation*}
        \left( \begin{tikzcd}[row sep= scriptsize]
                \mathfrak{g} \arrow[d, "G"]\\
                \mathfrak{g}'
            \end{tikzcd}, 
            \begin{tikzcd}[row sep= scriptsize]
                M \arrow[d, "\alpha"]\\
                M'
            \end{tikzcd}, r \right)  \longmapsto
            \left( \begin{tikzcd}[row sep= scriptsize]
                M \arrow[d, "\alpha"]\\
                M'
            \end{tikzcd},
            \begin{tikzcd}[row sep= scriptsize]
                \mathfrak{g} \times_c M \arrow[d, "\sigma"]\\
                \mathfrak{g}' \times_{c'} M'
            \end{tikzcd},
            \begin{tikzcd}[row sep= scriptsize]
                \mathfrak{g} \arrow[d, "G"]\\
                \mathfrak{g}'
            \end{tikzcd}\right),
    \end{equation*}
    where $\sigma(x,m) \defeq (G(x), \alpha(m) + r(x))$, is an equivalence of categories.
\end{theorem}
Although the previous theorem is phrased in an unusual way, it contains the same information as Theorem \ref{ext_coho}. We insist on this presentation because it is easier to generalize to relative Lie cohomology, which is our goal for the coming section. As mentioned earlier, our definition of relative Lie cohomology follows \cite{KaYu}, so we borrow their notation.

\subsection{Relative cohomology}

Let $f:\mathfrak{h}\rightarrow \mathfrak{g}$ be a Lie algebra map and $M$ a left $\mathfrak{g}$ module. The map $f$ induces an $\mathfrak{h}$-module structure on $M$ via $x \cdot m \defeq f(x) \cdot m$ for all $x \in \mathfrak{h}$ and $m \in M$. 
\begin{definition}\label{def_rel_coho}
    Given a Lie map $f: \mathfrak{h} \rightarrow \mathfrak{g}$ we define the cohomology of $\mathfrak{g}$ with coefficients in $M$ relative to $f$ as the cohomology of the total complex
    \begin{equation}\label{chevalley_one}
    \begin{tikzcd}
        M \arrow[d, "d"] & M \arrow[l, "-\mathrm{id}"] \arrow[d, "-d"]\\ 
        \mathrm{Hom}_\mathbb{K}(\mathfrak{h}, M) \arrow[d, " d "]& \mathrm{Hom}_\mathbb{K}(\mathfrak{g}, M) \arrow[l, "-f^*"] \arrow[d, "-d"]\\
        \mathrm{Hom}_\mathbb{K}(\bigwedge^2 \mathfrak{h}, M) \arrow[d, "d"] & \mathrm{Hom}_\mathbb{K}(\bigwedge^2 \mathfrak{g}, M) \arrow[l, "-f^*"] \arrow[d, "-d"]\\
        \mathrm{Hom}_\mathbb{K}(\bigwedge^3 \mathfrak{h}, M) \arrow[d] & \mathrm{Hom}_\mathbb{K}(\bigwedge^3 \mathfrak{g}, M) \arrow[l, "-f^*"] \arrow[d]\\
        \vdots & \vdots
    \end{tikzcd},
\end{equation}
where $d$ stands for the appropriate iteration of the Chevalley-Eilenberg differential. For reference, $dm(x)=x\cdot m$ for $m \in M$ and $x\in \mathfrak{h}$, and
\begin{equation*}
    \begin{split}
        d\omega(x_1, \dots, x_{n+1})= &\sum_i(-1)^{i+1}x_i \cdot \omega(x_1, \dots, \hat{x_{i}}, \dots, x_{n+1})\\
        &+ \sum_{i<j}(-1)^{i+j}\omega([x_i,x_j],x_1, \dots, \hat{x_i}, \dots, \hat{x_j}, \dots, x_{n+1})
    \end{split}
\end{equation*}
whenever $\omega \in \mathrm{Hom}_\mathbb{K}(\bigwedge^n\mathfrak{h}, M)$. We use the notation $H^*(\mathfrak{g},\mathfrak{h};M)$ for the $*$-th relative cohomology space.
\end{definition}
In practice, we will truncate the first row of the complex (\ref{chevalley_one}), so that
\begin{gather*}
    \text{Tot}(\mathcal{M})^1= \text{Hom}_\mathbb{K}(\mathfrak{g}, M);\\
     \text{Tot}(\mathcal{M})^2= \text{Hom}_\mathbb{K}(\mathfrak{h}, M) \oplus \text{Hom}_\mathbb{K}(\bigwedge^2 \mathfrak{g}, M);\\
      \text{Tot}(\mathcal{M})^3= \text{Hom}_\mathbb{K}(\bigwedge^2 \mathfrak{h}, M) \oplus \text{Hom}_\mathbb{K}(\bigwedge^3 \mathfrak{g}, M).
\end{gather*}
Our notion of relative Lie cohomology does not match that of Chevalley and Eilenberg \cite{Chevalley:1948zz}, but it coincides with Kawazumi and Kuno's \cite{KaYu}. Given a cochain map $f:\, (\mathbf{A}_*, d) \rightarrow (\mathbf{C}_*, \delta)$ between two cochain complexes, recall that its mapping cone $\mathbf{C}_{*-1} \times_f \mathbf{A}_*$ is the complex whose term of degree $n$ is $( \mathbf{C}_{*-1}\times_f \mathbf{A}_*)_n \defeq \mathbf{C}_{n-1} \oplus \mathbf{A}_{n}$, and with differential 
\begin{equation*}
    D_n(c_{n-1}, a_n) \defeq(\delta_{n-1}c_{n-1}-f_n a_n, -d_n a_n).
\end{equation*}
Let $\begin{tikzcd}[cramped, sep=small]
    P_* \arrow[r, "\varepsilon"] & \mathbb{K}
\end{tikzcd}$ and $\begin{tikzcd}[cramped, sep=small]
    F_* \arrow[r, "\varepsilon"] & \mathbb{K}
\end{tikzcd}$ be $U\mathfrak{g}$ and $U\mathfrak{h}$-projective resolutions of $\mathbb{K}$, such that $P_0=U\mathfrak{g}$ and $F_o=U\mathfrak{h}$. We can choose a chain map $f: F_* \rightarrow P_*$ that respects the morphism $f:U\mathfrak{h}\rightarrow U\mathfrak{g}$ and the respective augmentation maps. Kawazumi and Kuno define the cohomology of $\mathfrak{g}$ relative to $f$ with coefficients $M$ as the cohomology of the mapping cone
\begin{equation*} 
        H^*(\mathfrak{g}, \mathfrak{h};M)\defeq H^*((\text{Hom}_{U\mathfrak{h}}(F_{*-1}, M)) \times_{f^*} \text{Hom}_{U\mathfrak{g}}(P_*, M))
    \end{equation*}
    of the cochain map $f^*: \text{Hom}_{U\mathfrak{g}}(P_{*}, M) \rightarrow \text{Hom}_{U\mathfrak{h}}(F_{*}, M)$. Both definitions coincide if we choose $P_*$ and $F_*$ to be the Chevalley-Eilenberg complexes of $\mathfrak{g}$ and $\mathfrak{h}$, respectively. 

\subsection{Relative extensions}

\begin{definition}\label{def_relE}
    We write $\mathbf{RelE}$ for the category whose objects are commutative diagrams between Lie algebras of the sort:
    \begin{equation} \label{def_rel_e}
        \begin{tikzcd}
            & \mathfrak{h} \arrow[d, "f"] \arrow[ld, "s"']\\
            \mathfrak{e} \arrow[r, "p"]& \mathfrak{g}
        \end{tikzcd},
    \end{equation}
    with the added requirement that the Lie map $p:\mathfrak{e}\rightarrow \mathfrak{g}$ should be surjective with an abelian kernel. The morphisms in $\mathbf{RelE}$ consist of Lie maps $H: \mathfrak{h} \rightarrow \mathfrak{h}'$, $G: \mathfrak{g}\rightarrow \mathfrak{g}'$, and $E: \mathfrak{e}\rightarrow \mathfrak{e}'$, which fit into commutative diagrams
    \[
        \begin{tikzcd}
            \mathfrak{h} \arrow[d, "H"] \arrow[r, "f"]& \mathfrak{g} \arrow[d, "G"]\\
            \mathfrak{h}' \arrow[r, "f'"]& \mathfrak{g}'
        \end{tikzcd},\hspace{1 em}
        \begin{tikzcd}
            \mathfrak{h} \arrow[r, "s"] \arrow[d, "H"]& \mathfrak{e} \arrow[d, "E"]\\
            \mathfrak{h}' \arrow[r, "s'"]& \mathfrak{e}'
        \end{tikzcd}, \hspace{1 em}
        \begin{tikzcd}
            \mathfrak{e} \arrow[r, "p"] \arrow[d, "E"] & \mathfrak{g} \arrow[d, "G"]\\
            \mathfrak{e} \arrow[r, "p'"]& \mathfrak{g}'
        \end{tikzcd}.
    \]
    The composition of morphisms in $\mathbf{RelE}$ is defined in the natural way. Whenever
\begin{equation*}
        \begin{tikzcd}
            & \mathfrak{h} \arrow[d, "f"] \arrow[ld, "s"']\\
            \mathfrak{e} \arrow[r, "p"]& \mathfrak{g}
        \end{tikzcd} \in \mathrm{Ob}(\mathbf{RelE}),
\end{equation*}
we say $\mathfrak{e}$ is an abelian extension of $\mathfrak{g}$ by $\ker p$ relative to $f: \mathfrak{h} \rightarrow \mathfrak{g}$. Relative Lie algebra extensions are also extensions in the classical sense.
\end{definition}

\begin{remark}
    In the notation of diagram (\ref{def_rel_e}, let $u: \mathfrak{g} \rightarrow \mathfrak{e}$ be a section of the projection $p: \mathfrak{e} \rightarrow \mathfrak{g}$. Notice that $\ker p$ is a $\mathfrak{g}$-module with respect to the action
    \begin{equation}\label{g-mod_act}
        g \cdot m \defeq [m, u(g)]
    \end{equation}
    for any $g\in \mathfrak{g}$ and $m\in \ker p$. Since $\ker p$ is abelian, this action is independent of the choice of section.
\end{remark}

\begin{definition}
We write $\mathbf{RelCo}$ for the category whose objects consist of the data:
    \[
        ( f: \mathfrak{h} \rightarrow \mathfrak{g},\, M,\, w\oplus c \in Z^2(\mathfrak{g},\mathfrak{h};M)),
    \]
    where $f: \mathfrak{h} \rightarrow \mathfrak{g}$ is a Lie algebra map and $M$ is an abelian $\mathfrak{g}$-module. Morphisms now consist of Lie maps $H: \mathfrak{h} \rightarrow \mathfrak{h}'$, $G: \mathfrak{g} \rightarrow \mathfrak{g}'$, $\alpha: M \rightarrow M'$, and a cochain $r \in C^1(\mathfrak{g}, \mathfrak{h}; M')=\mathrm{Hom}_\mathbb{K}(\mathfrak{g}, M')$, such that
    \[
        \delta r= (\alpha\omega - \omega'H) \oplus (\alpha c - c'(G\otimes G)).
    \]
    Here, we are implicitly considering $M'$ as a $\mathfrak{g}$-module through $G$:
    \[
        x\cdot m= G(x) \cdot m,\, \text{for all}\, x\in \mathfrak{g}\, \text{and}\, m\in M'.
    \]
    Additionally, we also ask that the Lie maps satisfy the commutativity constraint $Gf=fH$, and also that
    \[
        \alpha(x \cdot m)=G(x)\cdot \alpha(m)
    \]
    for all $x\in \mathfrak{g}$ and $m\in M$.
    Composing the Lie map components of morphisms in $\mathbf{RelCo}$ is straightforward, but if $r \in C^1(\mathfrak{g}, \mathfrak{h}; M')$ and $r' \in C^1(\mathfrak{g}', \mathfrak{h}'; M'')$, then we define
    \[
        r' \circ_{\mathbf{RelCo}}r=\alpha' r +r'G.
    \]
We can check by direct computation that the composition $\circ_{\mathbf{RelCo}}$ is indeed associative.
\end{definition}

\begin{proposition} \label{cat-equiv}
    The functor $\mathrm{F}: \mathbf{RelCo} \rightarrow \mathbf{RelE}$, defined on objects by
    \begin{equation*}
        \left( \begin{tikzcd}
            \mathfrak{h} \arrow[r,"f"]& \mathfrak{g}
        \end{tikzcd},\, M,\, \omega \oplus c \in Z^2(\mathfrak{g}, \mathfrak{h}; M) \right) \mapsto \begin{tikzcd}
            & \mathfrak{h} \arrow[d, "f"] \arrow[dl, "f \times \omega" swap]\\
            \mathfrak{g}\times_c M \arrow[r] & \mathfrak{g}
        \end{tikzcd},
    \end{equation*}
    and on morphisms by
    \begin{equation*}
        \left( \begin{tikzcd}
            \mathfrak{h} \arrow[r, "f"] \arrow[d, "H"] & \mathfrak{g} \arrow[d, "G"]\\
            \mathfrak{h}' \arrow[r, "f'"]& \mathfrak{g}'
        \end{tikzcd},\, 
        \begin{tikzcd}
            M \arrow[d, "\alpha"]\\
            M'
        \end{tikzcd},\,
        r \in C^1(\mathfrak{g}, \mathfrak{h}; M')
        \right) \mapsto
        \left( \begin{tikzcd}
            \mathfrak{h} \arrow[r, "f"] \arrow[d, "H"] & \mathfrak{g} \arrow[d, "G"]\\
            \mathfrak{h}' \arrow[r, "f'"]& \mathfrak{g}'
        \end{tikzcd},\, 
        \begin{tikzcd}
            \mathfrak{h} \arrow[r, "f \times \omega"] \arrow[d, "H"] & \mathfrak{g}\times_c M \arrow[d, "\sigma"]\\
            \mathfrak{h}' \arrow[r, "f' \times \omega'"]  &\mathfrak{g}' \times_{c'} M'
        \end{tikzcd},\,
        \begin{tikzcd}
            \mathfrak{g} \times_c M \arrow[r] \arrow[d, "\sigma"]& \mathfrak{g} \arrow[d, "G"] \\
            \mathfrak{g}' \times_{c'} M' \arrow[r] & \mathfrak{g}'
        \end{tikzcd}
        \right),
    \end{equation*}
    where $\sigma(x,m)= (G(x), \alpha (m) +r(x))$ for all $(x,m)\in \mathfrak{g}\times_c M$, is an equivalence of categories.
\end{proposition}
\begin{proof}
    We prove the functor $\mathrm{F}$ is well defined in Appendix \ref{cat_appendix}. Firstly, we will prove that $\mathrm{F}$ is essentially surjective. Let 
    \begin{equation*}
        \mathcal{O}=\begin{tikzcd}
            & \mathfrak{h} \arrow[d, "f"] \arrow[ld, "s"']\\
            \mathfrak{e} \arrow[r, "p"]& \mathfrak{g}
        \end{tikzcd}
    \end{equation*}
    be an arbitrary object in $\mathbf{RelE}$. We choose a section $u: \mathfrak{g} \rightarrow \mathfrak{e}$ of the map $p$, and define the relative cochain
    \begin{gather*}
        \omega \oplus c \in Z^2(\mathfrak{g}, \mathfrak{h}; \ker p);\\
        c(x,y)=[u(x),u(y)]-u([x,y]),\, \text{for }\, x,y \in \mathfrak{g},\\
        \omega(h)=s(h)-us(h),\, \text{for }\,h\in\mathfrak{h}.
    \end{gather*}
    Here, we regard $\ker p$ as a $\mathfrak{g}$-module with respect to the action
    \[
        x \cdot m = [m, u(x)],
    \]
    for $x\in \mathfrak{g}$ and $m\in \ker p$.
    We can define a morphism from $\mathrm{F}(\mathcal{O})$ to $\mathcal{O}$ in $\mathbf{RelE}$ by the data:
    \begin{equation*}
        \begin{tikzcd}
            \mathfrak{h} \arrow[r, "f"] \arrow[d, "\mathrm{id}"]& \mathfrak{g} \arrow[d, "\mathrm{id}"]\\
            \mathfrak{h} \arrow[r, "f"] & \mathfrak{g}
        \end{tikzcd},\,
        \begin{tikzcd}
            \mathfrak{h} \arrow[r, "f \times \omega"] \arrow[d, "\mathrm{id}"] & \mathfrak{g} \times_c \ker p \arrow[d, "\psi"]\\
            \mathfrak{h} \arrow[r, "s"] & \mathfrak{e}
        \end{tikzcd},\,
        \begin{tikzcd}
            \mathfrak{g}\times_c \ker p \arrow[r] \arrow[d, "\psi"] & \mathfrak{g} \arrow[d, "\mathrm{id}"]\\
            \mathfrak{e} \arrow[r, "p"]& \mathfrak{g}
        \end{tikzcd};
    \end{equation*}
    where $\psi(x,m)=u(x)+m$ for all $(x,m)\in \mathfrak{g}\times_c \ker p$. The map $\psi$ is actually an isomorphism of Lie algebras, and thus defines an isomorphism in $\mathbf{RelE}$. The invertibility of $\psi$ is implied by the commutativity of the diagram
    \begin{equation*}
        \begin{tikzcd}
            0 \arrow[r] & \ker p \arrow[r] \arrow[d, "\mathrm{id}"] & \mathfrak{g} \times_c \ker p \arrow[r] \arrow[d, "\psi"] & \mathfrak{g} \arrow[r] \arrow[d, "\mathrm{id}"] & 0\\
            0 \arrow[r] & \ker p \arrow[r] & \mathfrak{e} \arrow[r, "p"] & \mathfrak{g} \arrow[r] & 0
        \end{tikzcd},
    \end{equation*}
    where each of the rows is a short exact sequence. We will now prove that the functor $\mathrm{F}$ is fully faithful. Denote two objects in $\mathbf{RelCo}$ by
    \begin{equation*}
        X=(f: \mathfrak{h}\rightarrow \mathfrak{g},\, M,\, \omega \oplus c \in Z^2(\mathfrak{g}, \mathfrak{h};M)),\hspace{1 em} 
        Y= (f': \mathfrak{h}'\rightarrow \mathfrak{g}',\, M',\, \omega' \oplus c' \in Z^2(\mathfrak{g}', \mathfrak{h}';M')).
    \end{equation*}
    To prove that $\mathrm{F}: \mathbf{RelCo}(X,Y) \rightarrow \mathbf{RelE}(F(X), F(Y))$ is a bijection, we propose an explicit inverse,
    \begin{equation*}
        \mathcal{G}: \mathbf{RelE}(F(x),  F(Y)) \rightarrow \mathbf{RelCo}(X,Y),
    \end{equation*}
    defined by
    \begin{equation*}
        \left( \begin{tikzcd}
            \mathfrak{h} \arrow[r, "f"] \arrow[d, "H"] & \mathfrak{g} \arrow[d, "G"] \\
            \mathfrak{h}' \arrow[r, "f'"] & \mathfrak{g}'
        \end{tikzcd},\,
        \begin{tikzcd}
            \mathfrak{h} \arrow[r, "f\times \omega"] \arrow[d, "H"] & \mathfrak{g} \times_c M \arrow[d, "\varphi"]\\
            \mathfrak{h}' \arrow[r, "f' \times \omega'"]& \mathfrak{g}' \times_{c'} M'
        \end{tikzcd},\,
        \begin{tikzcd}
            \mathfrak{g}\times_c M \arrow[d, "\varphi"] \arrow[r, "\pi"] & \mathfrak{g} \arrow[d, "G"]\\
            \mathfrak{g}' \times_c' M' \arrow[r, "\pi'"]& \mathfrak{g}'
        \end{tikzcd} \right) \mapsto
        \left( \begin{tikzcd}
            \mathfrak{h} \arrow[r, "f"] \arrow[d, "H"] & \mathfrak{g} \arrow[d, "G"] \\
            \mathfrak{h}' \arrow[r, "f'"] & \mathfrak{g}'
        \end{tikzcd},\,
        \begin{tikzcd}
            M \arrow[d, "\varphi|_M"] \\
            M'
        \end{tikzcd},\, r\in C^1(\mathfrak{g}, \mathfrak{h}; M')
        \right),
    \end{equation*}
    where $r(x)= \varphi(u(x))-u'(G(x))$ for the choice of sections
    \begin{equation*}
        \begin{split}
            u: \hspace{0.5 em} \mathfrak{g} &\rightarrow \mathfrak{g} \times_c M\\
            x & \mapsto (x,0)
        \end{split}\hspace{2 em} \text{and} \hspace{2 em} 
        \begin{split}
             u': \hspace{0.5 em} \mathfrak{g}' &\rightarrow \mathfrak{g}' \times_{c'} M'\\
            x & \mapsto (x,0)
        \end{split}
    \end{equation*}
    of $\pi$ and $\pi'$, respectively.  Note that condition $\pi'\varphi=G\pi$ implies that $\varphi$ truly restricts to a map $\varphi|_M: M \rightarrow M'$. Additionally, 
    \begin{equation*}
        \begin{split}
            \pi'r(x)=(\pi'\varphi)u(x)-(\pi'u')G(x)=G(\pi u)(x)-G(x)=0
        \end{split},
    \end{equation*}
    which verifies $r(x) \in M'=\ker \pi'$ for all $x\in \mathfrak{g}$.

    The map $\mathcal{G}$ is indeed an inverse to $\mathrm{F}$. We can explicitly see this on the relevant morphism data of $\mathbf{RelCo}$ and $\mathbf{RelE}$, respectively:
    \begin{equation}
        \begin{split}
            \mathcal{G}\circ \mathrm{F}: \hspace{1 em} C^1(\mathfrak{g}, \mathfrak{h}; M') \ni r \mapsto (\,x \mapsto& \sigma u(x)-u'G(x)\\
            &=(G(x), r(x))-(G(x),0)\\
            &=(0, r(x))\,);
        \end{split}
    \end{equation}
    \begin{equation} \label{fg=id}
        \begin{split}
            \mathrm{F} \circ \mathcal{G}: \hspace{1 em} \mathrm{Iso}_{\mathrm{Lie}}(\mathfrak{g}\times_c M, \mathfrak{g}'\times_{c'} M') \ni \varphi \mapsto ( \,(x,m)\mapsto &(G(x), \varphi|_M(m)+ r(x))\\
            &= \varphi(x,m)\,).
        \end{split}
    \end{equation}
    To verify the last equality, notice that the conditions $\pi' \varphi = G \pi$ and $\varphi(0,m)\in \ker \pi'$ for all $m\in M$ imply that $\phi$ is of the form
    \begin{equation*}
        \varphi(x, m)=(G(x), \varphi|_M(m)+\phi(x));
    \end{equation*}
    equation (\ref{fg=id}) follows from a straightforward calculation showing $\phi =r$.
\end{proof}

\subsection{Revisiting $\mathbf{Fox}$} \label{coho}

As can be gauged from its definition, the category $\mathbf{RelCo}$ is intimately related to $\mathbf{Fox}$, except the former seemingly needs more information to specify an object. In this subsection, we define a subcategory $\mathbf{RelCo}^\eta$ of $\mathbf{RelCo}$, and explicitly tie it to $\mathbf{Fox}$. For completeness, we also define the subcategory $\mathbf{RelE}^\eta$ of $\mathbf{RelE}$ which corresponds to $\mathbf{RelCo}^\eta$ under the equivalence of Proposition \ref{cat-equiv}.

Let $\mathfrak{h}$ an arbitrary Lie algebra. For the remainder of this paper, we will mainly be dealing with cohomology spaces $H^*(\mathfrak{h} \oplus \mathfrak{h}, \mathfrak{h}; U\mathfrak{h})$, where the $(\mathfrak{h} \oplus \mathfrak{h})$-module structure on $U\mathfrak{h}$ is given by
\begin{equation*}
    (x \oplus y) \cdot a=xa -ay, 
\end{equation*}
and we compute cohomology relative to the diagonal map,
\begin{equation} \label{diag_map}
    \begin{split}
        \Delta: \hspace{1em} \mathfrak{h} & \longrightarrow \mathfrak{h} \oplus \mathfrak{h}\\
        x & \longmapsto x \oplus x
    \end{split}.
\end{equation}

We will not be dealing with the entirety of $\mathbf{RelCo}$, but rather with sufficiently refined subcategories. 

\begin{definition}
Let $\mathbf{RelCo}^{\Delta}$ be the (non-full) subcategory of $\mathbf{RelCo}$ comprised of triplets of the form
    \begin{equation*}
    (\Delta: \mathfrak{h} \rightarrow \mathfrak{h} \oplus \mathfrak{h},\, U\mathfrak{h},\, \omega \oplus c \in Z^2(\mathfrak{h} \oplus \mathfrak{h}, \mathfrak{h}; U\mathfrak{h})),
\end{equation*}
where we restrict the notion of morphisms so that these consist of the data
\begin{equation*}
    (f: \mathfrak{h} \rightarrow \mathfrak{h}',\, r \in C^1(\mathfrak{h} \oplus \mathfrak{h}, \mathfrak{h}; U\mathfrak{h}'))
\end{equation*}
satisfying that
\begin{equation*}
    \delta r=(f \omega - \omega'f) \oplus(fc -c'(f \otimes f)).
\end{equation*}
\end{definition}

To bridge the gap between $\mathbf{Fox}$ and $\mathbf{RelCo^\Delta}$, we can recast the definition of the former so that it more closely resembles the latter. Given a Lie algebra $\mathfrak{h}$, we can essentially treat quasi-derivations and Fox pairings in $U\mathfrak{h}$ as cocycles in $H^2(\mathfrak{h}\oplus \mathfrak{h}, \mathfrak{h}; U \mathfrak{h})$. Let $f: \mathfrak{h} \rightarrow \mathfrak{h}'$ be an arbitrary but fixed Lie map; we define the total complex
\begin{equation} \label{goodone}
        \mathcal{M}(f) \defeq \begin{tikzcd} [sep=scriptsize]
            \mathrm{Qder}_f(U\mathfrak{h}, U\mathfrak{h}') \arrow[d, "\sigma"] & \mathrm{LFox}_f(U\mathfrak{h}, U \mathfrak{h}') \oplus \mathrm{RFox}_f(U\mathfrak{h}, U \mathfrak{h}') \arrow[l, "-\mu"] \arrow[d, "-\tau"]\\
            \mathcal{B}^3(\mathfrak{h} \oplus \mathfrak{h}, \mathfrak{h}; U \mathfrak{h}') \arrow[d] & \mathrm{Fox}_f(U\mathfrak{h}, U\mathfrak{h}') \arrow[l, "-g"] \arrow[d]\\
            0 \arrow[d] & 0 \arrow[l] \arrow[d]\\
            \vdots & \vdots
        \end{tikzcd}
    \end{equation}
where $\mathcal{B}^3(\mathfrak{h} \oplus \mathfrak{h}, \mathfrak{h}; U\mathfrak{h}')$ suggestively denotes the linear subspace defined by
    \begin{equation*}
        \mathcal{B}^3(\mathfrak{h} \oplus \mathfrak{h}, \mathfrak{h}; U\mathfrak{h}') \defeq \{ w \in \mathrm{Hom}_\mathbb{K}(K\otimes K, U\mathfrak{h}')\,|\, w=\sigma(q)+g(\rho)\, \text{for some }q\oplus\rho\in C^2(\mathcal{M}(f)) \}, 
    \end{equation*}
    and $K$ stands for the kernel of the augmentation map, $\varepsilon: U \mathfrak{h} \rightarrow \mathbb{K}$. The relevant maps are given by
    \begin{gather*}
        \mu(\partial_L \oplus \partial_R)(a)= \partial_L(a)+\partial_R(a),\\
        \sigma (q)(a,b)= q(a)b+a q(b) -q(ab),\\
        \tau(\partial_L \oplus \partial_R)(a,b)= \partial_L(a)D(b) + D(a)\partial_R(b),\\
        g(\rho)=\rho,
    \end{gather*}
    for $a, b \in U\mathfrak{h}$. As in Section \ref{fox}, we capitalize on the $f$-induced $U\mathfrak{h}$-bimodule structure on $U\mathfrak{h}'$, and abbreviate
    \begin{equation*}
        xay = x \cdot a \cdot y =f(x)af(y)
    \end{equation*}
    for $x,y \in U\mathfrak{h}$, $a \in U\mathfrak{h}'$. Notice that the notion of exact Fox pairing and exact quasi-derivation we described in previous sections (see Definition \ref{FoxP_exact} and Definition \ref{def_ex_qder}, respectively), matches the notion of exactness we can extract from the total complex (\ref{goodone}). Moreover, the condition $q \in \mathrm{Qder}(-\rho)$ is equivalent to asking that $q \oplus \rho \in Z^2(\mathcal{M}(\mathrm{id}_\mathfrak{h}))$, so we can rephrase $\mathbf{Fox}$ as the category consisting of objects of the form
    \begin{equation*}
        ( \mathfrak{h}, q \oplus \rho \in Z^2(\mathcal{M}(\mathrm{id}_\mathfrak{h})) ),
    \end{equation*}
    and in which the set of morphisms between objects $(\mathfrak{h}, q \oplus \rho )$ and $( \mathfrak{h}', q' \oplus \rho')$ is given by
    \begin{equation*}
        \{ f: \mathfrak{h} \rightarrow \mathfrak{h}', r \in C^1(\mathcal{M}(f)) \, |\, \delta^{\mathcal{M}(f)} r = (f q - q' f) \oplus (f \rho - \rho' (f \otimes f))\}
    \end{equation*}

\begin{proposition}\label{def_E}
    The functor $\mathrm{E}: \mathbf{Fox} \rightarrow \mathbf{RelCo}^{\Delta}$, defined on objects and morphisms, respectively, by
    \begin{equation*}
        (\mathfrak{h}, q \oplus \rho) \mapsto (\Delta:\mathfrak{h} \rightarrow \mathfrak{h} \oplus \mathfrak{h},  U\mathfrak{h} , \omega_q \oplus c_\rho),
    \end{equation*}
    \begin{equation*}
        \left(
        \begin{tikzcd}
            \mathfrak{h} \arrow[d, "f"]\\
            \mathfrak{h}'
        \end{tikzcd},
        r \in C^1(\mathcal{M}(f))
        \right) \mapsto
        \left(
        \begin{tikzcd}
            \mathfrak{h} \arrow[d, "f"]\\
            \mathfrak{h}'
        \end{tikzcd},
        \begin{tikzcd}
            U\mathfrak{h} \arrow[d, "f"]\\
            U \mathfrak{h}'
        \end{tikzcd},
        \mathrm{E}(r) \in C^1(\mathfrak{h} \oplus \mathfrak{h}, \mathfrak{h}; U\mathfrak{h}')
        \right),
    \end{equation*}
    is fully faithful, where 
    \begin{gather*}
        \omega_q(x)= q(x) \hspace{1em}\text{for all }x \in \mathfrak{h};\\
        c_\rho(v,w)=\rho(v_1, w_2) - \rho(w_1, v_2) \hspace{1em} \text{for all }v=v_1 \oplus v_2, w=w_1 \oplus w_2 \in \mathfrak{h} \oplus \mathfrak{h};\\
        \text{for }r= \partial_L \oplus \partial_R,\, E(r)(h_1 \oplus h_2)=\partial_L(h_1) + \partial_R(h_2).
    \end{gather*}
\end{proposition}
\begin{proof}
    Let $\delta$ denote the differential in the Chevalley-Eilenberg total complex (\ref{chevalley_one}). To check that $\mathrm{E}$ is well defined, we should verify that
    \begin{equation*}
        \delta(\omega_q \oplus c_\rho)=(d \omega_q - \Delta^* c_\rho) \oplus d c_\rho=0 \oplus0. 
    \end{equation*}
    We can check by direct computation that $d c_\rho$ is identically zero for all Fox pairings $\rho$. Likewise, let $x,y \in \mathfrak{h}$, then
    \begin{equation*}
        \begin{split}
            d \omega_q(x,y) &=x \cdot \omega_q(y)-y\cdot \omega_q(x)- \omega_q([x,y])\\
            &=xq(y)-q(y)x-yq(x)+q(x)y-q(xy)+q(yx)\\
            &=(xq(y)+q(x)y-q(xy))-(q(y)x+yq(x)-q(yx))\\
            &=\rho(x,y)-\rho(y,x)\\
            &=c_\rho(x\oplus x, y \oplus y)=\Delta^*c_\rho(x,y), 
        \end{split}
    \end{equation*}
    where on the third equality we used that $q \in \mathrm{Qder}(-\rho)$. 
    
    To verify that $\mathrm{E}$ is fully faithful, consider two objects in $\mathbf{Fox}$, $\mathrm{X}\defeq (\mathfrak{h}, q \oplus \rho)$ and $\mathrm{Y} \defeq ( \mathfrak{h}', q' \oplus \rho')$. To check that $\mathbf{Fox}(\mathrm{X,Y})$ and $\mathbf{RelCo}^\Delta(\mathrm{E(X), E(Y)})$ have the same cardinality, we need to construct a bijection between the sets
    \begin{equation*}
        C^1(\mathcal{M}(f)) \hspace{1em} \text{and} \hspace{1em} S \defeq \{s \in C^1(\mathfrak{h} \oplus \mathfrak{h}, \mathfrak{h}; U \mathfrak{h}')\, |\, \delta s=(f \omega_q-\omega_{q'}f) \oplus (fc_\rho - c_{\rho'}(f\otimes f))\},
    \end{equation*}
    for any Lie morphism $f:\mathfrak{h} \rightarrow \mathfrak{h}'$. By the Poincar\'e-Birkhoff-Witt theorem, any left (resp. right) Fox derivative on $\mathrm{LFox}_f(U\mathfrak{h}, U \mathfrak{h}')$ (resp. $\mathrm{RFox}_f(U\mathfrak{h}, U\mathfrak{h}')$) is fully determined by its values on $\mathfrak{h}$. Given $s \in S$, we define $s_L$ (resp. $s_R$) as the unique left (resp. right) Fox derivative such that
    \begin{equation*}
        s_L(x)=s(x \oplus 0) \hspace{0.5em} (\text{resp.}\,\, s_R(x)=s(0 \oplus x)) \hspace{1.5em} \text{for all }x\in \mathfrak{h}.
    \end{equation*}
    To check that $s_L$ and $s_R$ are well defined, notice that, since $s \in S$,
    \begin{equation*}
        ds(x \oplus 0, y \oplus 0)=fc_\rho(x \oplus0, y\oplus 0)-c_{\rho'}(f(x) \oplus0, f(y) \oplus0).
    \end{equation*}
    But $c_\rho(x \oplus0, y\oplus 0)= \rho(x,0)-\rho(y,0)=0$, and similarly for $c_{\rho'}$. It follows that 
    \begin{equation}\label{ds=0}
        ds(x\oplus 0, y \oplus 0)=0 \hspace{1em} \text{for all }x,y\in \mathfrak{h}.
    \end{equation}
    Now we may compute
    \begin{equation*}
        \begin{split}
            s_L([x,y])=s([x,y] \oplus0)&=f(x)s(y\oplus0)-f(y)s(x \oplus 0)\\
            &=f(x)s_L(y)-f(y)s_L(x)\\
            &=s_L(xy)-s_L(yx),
        \end{split}
    \end{equation*}
    where we used equation (\ref{ds=0}) on the second equality, and the condition that $s_L$ be a left Fox derivative on the last one. This shows $s_L$ is well defined on $U\mathfrak{h}$, and an analogous argument works for $s_R$. We can check by direct computation that
    \begin{equation*}
        \begin{split}
            S &\longrightarrow C^1(\mathcal{M}(f))\\
            s & \longmapsto s_L \oplus s_R,
        \end{split}
    \end{equation*}
    is an inverse to $\mathrm{E}:C^1(\mathcal{M}(f)) \rightarrow S$.
\end{proof}

The functor $\mathrm{E}: \mathbf{Fox} \rightarrow \mathbf{RelCo}^\Delta$ is not essentially surjective. This is why we restrict to finer subcategories to obtain an equivalence. 

\begin{definition}
Let $\mathbf{Fox}^\eta$ be the full subcategory of $\mathbf{Fox}$ consisting of objects of the form,
\begin{equation*}
    (\mathfrak{f}, q \oplus \rho \in Z^2(\mathcal{M}(\mathrm{id}_\mathfrak{f})) ),
\end{equation*}
where $\mathfrak{f}$ is isomorphic to a free Lie algebra\footnote{The isomorphism between $\mathfrak{f}$ and the free Lie algebra is not specified, and it need not be \textit{preserved} by the morphisms in $\mathbf{Fox}^\eta$.}. Analogously, let $\mathbf{RelCo}^\eta$ be the full subcategory of $\mathbf{RelCo}^\Delta$ in which objects are of the form
\begin{equation}
        (\Delta:\mathfrak{f} \rightarrow \mathfrak{f}\oplus \mathfrak{f},\, U\mathfrak{f},\, \omega \oplus c \in Z^2(\mathfrak{f} \oplus \mathfrak{f}, \mathfrak{f}; M)),
\end{equation}
where $\mathfrak{f}$ is isomorphic to a free Lie algebra.
\end{definition}

\begin{theorem}\label{quasi-iso}
    Restricted to the subcategory $\mathbf{Fox}^\eta$, the functor $\mathrm{E}: \mathbf{Fox}^\eta \rightarrow \mathbf{RelCo}^\eta$ is an equivalence of categories.
\end{theorem}
\begin{proof}
    All that remains to show is that $\mathrm{E}$ is essentially surjective. To simplify the notation, suppose $\mathfrak{f}$ is a free Lie algebra, as opposed to just isomorphic to one. Since $\mathfrak{f}$ is free, then the following is a $(U\mathfrak{f} \otimes U \mathfrak{f})-$projective resolution of $\mathbb{K}$:
    \begin{equation}\label{free_proj}
        P_*:\, \begin{tikzcd}
            0 \arrow[r] & K \otimes K \arrow[r, "f"] & (K \otimes U\mathfrak{f}) \oplus (U\mathfrak{f} \otimes K) \arrow[r, "g"]& U\mathfrak{f} \otimes U\mathfrak{f} \arrow[r, "\varepsilon \otimes \varepsilon"] & \mathbb{K},
        \end{tikzcd}
    \end{equation}
    where $f(x) = x \oplus x$ for all $x \in K \otimes K$ and $g(x \oplus y)= x-y$ for $x \in K \otimes U\mathfrak{f}$ and $y \in U\mathfrak{f} \otimes K$. We are viewing our spaces as left $(U\mathfrak{f}\otimes U\mathfrak{f})$-modules through the action
    \[
        (a\otimes b) \cdot (x \otimes y)= ax \otimes by.
    \]
    Note also the left $(U\mathfrak{f} \otimes U\mathfrak{f})$-module structure on $U\mathfrak{f}$, $(a\otimes b) \cdot x = axS(b)$.

    We also consider the Chevalley-Eilenberg chain complex as a separate $(U\mathfrak{f} \otimes U \mathfrak{f)}-$projective resolution:
    \begin{equation*}
        \begin{tikzcd}
            Q_*:\, \cdots \arrow[r] & (U\mathfrak{f} \otimes U \mathfrak{f}) \otimes \bigwedge^2(\mathfrak{f} \oplus \mathfrak{f}) \arrow[r, "d_2"] & (U\mathfrak{f} \otimes U \mathfrak{f}) \otimes (\mathfrak{f} \oplus \mathfrak{f}) \arrow[r, "d_1"] & U\mathfrak{f} \otimes U\mathfrak{f} \arrow[r, "\varepsilon \otimes \varepsilon"] & \mathbb{K}
        \end{tikzcd}
    \end{equation*}
    It follows that there must exist chain maps $D_*: P_* \rightarrow Q_*$ and $E_*: Q_* \rightarrow P_*$ inducing the identity on $\mathbb{K}$. In particular, we pinpoint the algebra morphism
    \begin{equation*}
        \begin{split}
            E_2: \hspace{1em} (U\mathfrak{f} \otimes U \mathfrak{f}) \otimes \bigwedge^2 (\mathfrak{f} \oplus \mathfrak{f}) & \rightarrow K \otimes K\\
           (a \otimes b) \otimes (x \wedge y) & \mapsto ay_1 \otimes bx_2-ax_1 \otimes by_2 .
        \end{split}
    \end{equation*}
    There must also exist a uniquely defined chain homotopy $\Psi_*:Q_* \rightarrow Q_{*+1}$ connecting $D_*E_*: Q_* \rightarrow Q_*$ to the identity on $Q_*$.

    We can verify by direct computation that the map 
    \begin{equation*}
        \begin{split}
            \mathrm{Hom}_{U \mathfrak{f}\otimes U \mathfrak{f}}(K \otimes K, U\mathfrak{f}) &\rightarrow \mathrm{Fox}(U\mathfrak{f})\\
            \rho & \mapsto ((a,b) \mapsto \rho(a, S(b))
        \end{split}
    \end{equation*}
    is a bijection. For clarity, we will leave it implicit, so that we can interpret $(D_2)^*$ as a map between the spaces\footnote{We have also made the identification $\mathrm{Hom}_{U\mathfrak{f}\otimes U\mathfrak{f}}((U\mathfrak{f} \otimes U \mathfrak{f})\otimes \bigwedge^2(\mathfrak{f}\oplus \mathfrak{f}), U\mathfrak{f}) \cong \mathrm{Hom}_\mathbb{K}(\bigwedge^2(\mathfrak{f}\oplus \mathfrak{f}), U\mathfrak{f})$.}
    \begin{equation*}
        (D_2)^*: \mathrm{Hom}_\mathbb{K}(\bigwedge^2(\mathfrak{f} \oplus \mathfrak{f}), U\mathfrak{f}) \rightarrow \mathrm{Fox}(U\mathfrak{f}).
    \end{equation*}

    With the previous set-up, let
    \begin{equation} \label{gen_obj}
        (\Delta:\mathfrak{f} \rightarrow \mathfrak{f}\oplus \mathfrak{f},\, U\mathfrak{f},\, \omega \oplus c \in Z^2(\mathfrak{f} \oplus \mathfrak{f}, \mathfrak{f}; U\mathfrak{f})) \in \mathrm{Ob}(\mathbf{RelCo}^\eta).
    \end{equation}
    Consider now $( \mathfrak{f}, q_\omega \oplus \rho_c \in Z^2(\mathcal{M}(\mathrm{id}_{\mathfrak{f}}))) \in \mathrm{Ob}(\mathbf{Fox}^\eta)$, where
    \begin{equation*}
        \rho_c \defeq (D_2)^*c, 
    \end{equation*}
    and $q_\omega$ is uniquely defined by imposing that $q_\omega \in \mathrm{Qder}(-\rho_c)$ and
    \begin{equation*}
    q_\omega(\alpha_i)=\omega(\alpha_i)+\Delta^*(\Psi_1)^*c(\alpha_i)
    \end{equation*}
    for all $i\in I$. 
    
    Since $q_\omega \in \mathrm{Qder}(-\rho_c)$, $q_\omega \oplus \rho_c \in Z^2(\mathcal{M}(\mathrm{id}_\mathfrak{f}))$. Additionally, 
    \begin{equation*}
        \mathrm{E}(q_\omega \oplus \rho_c)= (\omega + \Delta^*(\Psi_1)^*c) \oplus (D_2 \circ E_2)^*c, 
    \end{equation*}
    where we used $\mathrm{E}(\rho)=c_\rho = (E_2)^*\rho$ for any $\rho \in \mathrm{Fox}(U\mathfrak{f})$. It follows that
    \begin{equation*}
        \mathrm{E}(q_\omega \oplus \rho_c) - \omega \oplus c = \Delta^*(\Psi_1)^*c \oplus d(\Psi_1)^*c=\delta(-(\Psi_1)^*c),
    \end{equation*}
    where we used the condition that $\Psi_*$ is a chain homotopy. This implies that
    \begin{equation*}
        (\mathrm{id}_\mathrm{f}, \mathrm{id}_{U\mathfrak{f}}, -(\Psi_1)^*c \in C^1(\mathfrak{f}\oplus \mathfrak{f}, \mathfrak{f}; U \mathfrak{f}))
    \end{equation*}
    is an isomorphism in $\mathbf{RelCo}^\eta$ between the objects $\mathrm{E}( \mathfrak{f}, q_\omega \oplus \rho_c)$ and $( \Delta,\, U\mathfrak{f},\, \omega \oplus c)$.
\end{proof}

For future reference, we also define a category $\mathbf{RelE}^\eta$; it will be (equivalent to) the subcategory of $\mathbf{RelE}$ that corresponds to $\mathbf{RelCo^\eta}$ under the categorical equivalence of Proposition \ref{cat-equiv}. Unlike $\mathbf{RelCo}^\eta$, the objects of $\mathbf{RelE}^\eta$ will be decorated with a module isomorphism.
\begin{definition}
Let $\mathbf{RelE}^\eta$ be the category with objects consisting of pairs
    \begin{equation*}
        \left(
            \begin{tikzcd}
                & \mathfrak{f} \arrow[ld, "g" swap] \arrow[d, "\Delta"]\\
                \mathfrak{e} \arrow[r, "p"]& \mathfrak{f} \oplus \mathfrak{f}
            \end{tikzcd},
            \begin{tikzcd}
                \ker p \arrow[r, "\phi"] & U\mathfrak{f}
            \end{tikzcd}
        \right),
    \end{equation*}
    where $\mathfrak{f}$ is isomorphic to a free Lie algebra, and the commutative diagram in Lie algebras satisfies that
    \begin{itemize}
        \item The map $p$ is surjective with an abelian kernel;
        \item The diagonal map $\Delta: \mathfrak{f} \rightarrow \mathfrak{f} \oplus \mathfrak{f}$ is given by $\Delta(x)=x\oplus x$.
    \end{itemize}
    The map $\phi: \ker p \rightarrow U\mathfrak{f}$ is an isomorphism of $\mathfrak{f} \oplus \mathfrak{f}$-modules, where the action on $U\mathfrak{f}$ is given by
    \begin{equation*}
        x \cdot a \defeq xa -ax
    \end{equation*}
    for $x\in \mathfrak{f}$ and $a\in U\mathfrak{f}$. Recall how a choice of section of $p: \mathfrak{e} \rightarrow \mathfrak{f}\oplus \mathfrak{f}$ defined an $\mathfrak{f\oplus \mathfrak{f}}$-module structure on $\ker p$ (cf. equation (\ref{g-mod_act})). Morphisms in $\mathbf{RelE}^\eta$ consist of pairs of Lie maps
    \begin{equation*}
        (f:\mathfrak{f} \rightarrow \mathfrak{f}', \psi: \mathfrak{e} \rightarrow \mathfrak{e}')
    \end{equation*}
    which slot into commutative diagrams:
    \begin{equation*}
        \begin{tikzcd}
            \mathfrak{f} \arrow[r, "s"] \arrow[d, "f"]  &\mathfrak{e} \arrow[d, "\psi"] \\
            \mathfrak{f}' \arrow[r, "s'"] & \mathfrak{e}' 
        \end{tikzcd} \hspace{2.5em}
        \begin{tikzcd}
            \mathfrak{e} \arrow[r, "p"] \arrow[d, "\psi"] & \mathfrak{f}\oplus \mathfrak{f} \arrow[d, "f \oplus f"] & \\
            \mathfrak{e}' \arrow[r, "p'"] & \mathfrak{f}' \oplus \mathfrak{f}' 
        \end{tikzcd}
        \begin{tikzcd}
            \ker p \arrow[r, "\phi"] \arrow[d, "\psi"] & U\mathfrak{f} \arrow[d, "f"]\\
            \ker p' \arrow[r, "\phi'"] & U\mathfrak{f}'
        \end{tikzcd}
    \end{equation*}
\end{definition}
Notice that the condition that $p' \psi(m)=(f \oplus f)p(m)$ for all $m \in \mathfrak{e}$ ensures that $\psi$ restricts to a map on the kernels of the projections.

The following result follows immediately from the equivalence of categories $\mathbf{RelE}$ and $\mathbf{RelCo}$; we state it for the sake of completeness:
\begin{proposition}\label{co_del=e_del}
    The functor $\mathrm{F}: \mathbf{RelCo}^\eta \rightarrow \mathbf{RelE}^\eta$, defined on objects by
    \begin{equation*}
        ( \Delta:\mathfrak{f} \rightarrow \mathfrak{f} \oplus \mathfrak{f},\,U\mathfrak{f},\, \omega \oplus c \in Z^2(\mathfrak{f} \oplus f, \mathfrak{f}; U\mathfrak{f}) ) \mapsto 
        \left(
        \begin{tikzcd}
            & \mathfrak{f} \arrow[d, "\Delta"] \arrow[dl, "\Delta \times \omega" swap]\\
            (\mathfrak{f} \oplus \mathfrak{f})\times_c U\mathfrak{f} \arrow[r] & \mathfrak{f} \oplus \mathfrak{f}
        \end{tikzcd},\, \mathrm{id}: U\mathfrak{f} \rightarrow U\mathfrak{f}
       \right), 
    \end{equation*}
    and on morphisms by
    \begin{equation*}
        \left(
        \begin{tikzcd}
            \mathfrak{f} \arrow[d, "f"]\\
            f
        \end{tikzcd},\,
        r \in C^1(\mathfrak{f} \oplus \mathfrak{f}, \mathfrak{f}; U\mathfrak{f}')
        \right) \mapsto 
        \left(
            \begin{tikzcd}
                \mathfrak{f} \arrow[d, "f"] \\
            \mathfrak{f}'
            \end{tikzcd}\,
            \begin{tikzcd}
                \mathfrak{e} \times_c U\mathfrak{f} \arrow[d, "\sigma"]\\
                \mathfrak{e}' \times_{e'} U\mathfrak{f}'
            \end{tikzcd}
        \right),
    \end{equation*}
    where $\sigma(x,m) = (\psi(x), \alpha(m) + r(x))$, is an equivalence of categories.
\end{proposition}

\section{Operads}

In this chapter, we construct bracket and cobracket operations out of right operad modules in the category of small $\mathrm{CoAlg}$-enriched small categories. Unfortunately, we cannot define a nontrivial functor $\Lambda: \mathrm{OpR}\, \mathbf{Cat}(\mathbf{CoAlg}) \rightarrow \mathbf{GoTu}$. Rather, the category $\mathrm{OpR}\, \mathbf{Cat}(\mathbf{CoAlg})$ must become sufficiently decorated before we can produce operations out of it. In particular, we must have a notion of \textit{parenthesizations} as objects in our candidate operad modules, and they must come with a marked operation in arity one (cf. Definition \ref{delta_cat}). Our intention is to produce diagrams in $\mathbf{RelE}^\eta$ from operad modules. Given the equivalences of categories $\mathbf{RelE}^\eta \cong \mathbf{RelCo}^\eta \cong \mathbf{Fox}^\eta$, and the functor $\Psi: \mathbf{Fox} \rightarrow \mathbf{GoTu}$, such a diagram would contain the information needed to define bracket and cobracket operations. We close this chapter by reviewing Gonzalez' $\mathbf{PaCD}^f$-operad module of framed paranthesized chord diagrams of genus $g$, $\mathbf{PaCD}^f_g$, and then computing the bracket and cobracket we may extract from them. 

\subsection{Operad modules}\label{operad_sec}

\begin{definition}\cite[\S 1]{gonz}\label{OpR_def}
    Let $\mathcal{C}$ be a symmetric monoidal category. We write $\mathrm{OpR}\, \mathcal{C}$ for the category of pairs $(\mathcal{P}, \mathcal{M})$, where $\mathcal{P}$ is an operad in $\mathcal{C}$ and $\mathcal{M}$ is a right $\mathcal{P}$-module. A morphism $(\mathcal{P}, \mathcal{M}) \rightarrow (\mathcal{Q}, \mathcal{N})$ consists of an operad map $f: \mathcal{P} \rightarrow \mathcal{Q}$ and a $\mathcal{P}$-module map $g: \mathcal{M} \rightarrow \mathcal{N}$, where $\mathcal{N}$ is regarded as a $\mathcal{P}$-module via $f$.
\end{definition}
All our operads and operad modules will be pointed in the sense of \cite{gonz}. For our purposes, this essentially means there exist operations of arity zero.

\begin{definition}
Let $\mathbf{Cat(CoAlg)}$ denote the symmetric monoidal category of small $\mathrm{CoAlg}$-enriched categories. We write $\mathbf{Pa}$ for the operad of parenthesized permutations in $\mathbf{Cat(CoAlg)}$. More explicitly, for each $n  > 0$, $\mathbf{Pa}(n)$ is the category defined by:
    \begin{enumerate}
        \item Objects of $\mathbf{Pa}(n)$ consist of maximal parenthesized permutations of the set $\{1,\dots, n\}$.
        \item For any two parenthesized permutations $\sigma$ and $\sigma'$,
        \begin{equation*}
            \mathrm{mor_{\mathbf{Pa}(n)}}(\sigma, \sigma') =\begin{cases}\mathbb{K} & \text{if}\,\, \sigma=\sigma',\\
            0 & \text{otherwise}                
            \end{cases}
        \end{equation*}
    \end{enumerate}
    We also set $\mathbf{Pa}(0)=\mathbf{Pa}(1)$. The operadic composition in $\mathbf{Pa}(n)$ is defined by substituting parenthesizations in place of numbers. Composing with $\mathbf{Pa}(0)$ on the right deletes the corresponding number. For $1 \leq i \leq n$, we will use the following shorthand notation:
\begin{equation*}
    \begin{split}
        d_i: \hspace{1em} \mathrm{Ob}(\mathbf{Pa}(n)) & \longrightarrow \mathrm{Ob}(\mathbf{Pa}(n+1))\\
        p_n &\longmapsto p_n \circ_i^{n,2} (12);
    \end{split} \hspace{2 em}
    \begin{split}
        s_i: \hspace{1em} \mathrm{Ob}(\mathbf{Pa}(n)) & \longrightarrow \mathrm{Ob}(\mathbf{Pa}(n-1))\\
        p_n & \longmapsto p_n\circ_i^{n,0} 1 .
    \end{split}
\end{equation*}
Furthermore, we also define:
\begin{equation*}
    \begin{split}
        d_0: \hspace{1em} \mathrm{Ob}(\mathbf{Pa}(n)) & \longrightarrow \mathrm{Ob}(\mathbf{Pa}(n+1))\\
        p_n & \longmapsto (12) \circ_2^{2,n} p_n;
    \end{split} \hspace{2em}
    \begin{split}
        d_{n+1}: \hspace{1em} \mathrm{Ob}(\mathbf{Pa}(n)) & \longrightarrow \mathrm{Ob}(\mathbf{Pa}(n+1))\\
        p_n & \longmapsto (12) \circ_1^{2,n} p_n.
    \end{split}
\end{equation*}
Additionally, we abbreviate $d^{i}\defeq(d_0)^i$ for $i \geq 2$. 
\end{definition}

We can regard $\mathbf{Pa}$ as a right operad $\mathbf{Pa}$-module, so that $(\mathbf{Pa}, \mathbf{Pa)} \in \mathrm{Ob}(\mathrm{OpR}\, \mathbf{Cat}(\mathbf{CoAlg}))$. In the future, we will only deal with operads in $\mathrm{OpR}\, (\mathbf{CoAlg})$ that share the same objects as $\mathbf{Pa}$, in a sense we will make precise presently.

\begin{definition}
Recall that the \textit{under category} $(\mathbf{Pa}, \mathbf{Pa})/\mathrm{OpR}\, \mathbf{Cat}(\mathbf{CoAlg})$ is the category in which objects are morphisms $F \in \mathrm{Mor(\mathrm{OpR}\, \mathbf{Cat}(\mathbf{CoAlg}))}$ such that $\mathrm{dom}(F)=(\mathbf{Pa}, \mathbf{Pa})$, and in which
    \small
            \begin{equation*}
                \mathrm{mor}_{(\mathbf{Pa}, \mathbf{Pa})/\mathrm{OpR}\, \mathbf{Cat}(\mathbf{CoAlg})}\left(
                \begin{tikzcd}[row sep = scriptsize]
                    (\mathbf{Pa}, \mathbf{Pa}) \arrow[d, "F"]\\
                    (\mathcal{P}, \mathcal{M})
                \end{tikzcd},
                \begin{tikzcd}[row sep = scriptsize]
                    (\mathbf{Pa}, \mathbf{Pa}) \arrow[d, "F'"]\\
                    (\mathcal{P}',\mathcal{M}')
                \end{tikzcd}
                \right) = \left\{ G \in \mathrm{mor}_{\mathrm{OpR}\, \mathbf{Cat}(\mathbf{CoAlg})}((\mathcal{P}, \mathcal{M}), (\mathcal{P}', \mathcal{M}'))\, |\, G \circ F =F'  \right\}. 
            \end{equation*}
    \normalsize
By abuse of notation, whenever 
\begin{equation*}
    \left(
    \begin{tikzcd}[column sep= large]
        (\mathbf{Pa}, \mathbf{Pa}) \arrow[r, "{(F_\mathcal{P}, F_\mathcal{M})}"] & (\mathcal{P}, \mathcal{M})
    \end{tikzcd}
    \right) \in (\mathbf{Pa}, \mathbf{Pa})/\mathrm{OpR}\, \mathbf{Cat} (\mathbf{CoAlg})
\end{equation*}
we will abbreviate $F_\mathcal{P}(\sigma) = F_\mathcal{M}(\sigma)=\sigma$ for all parenthesizations $\sigma \in \mathrm{Ob}(\mathbf{Pa})$. Furthermore, we will often omit any mention of the morphism $(F_\mathcal{P}, F_\mathcal{M})$, and simply write $(\mathcal{P}, \mathcal{M})\in (\mathbf{Pa}, \mathbf{Pa})/\mathrm{OpR}\, \mathbf{Cat}(\mathbf{CoAlg})$. More plainly, in all operads and operad modules in $(\mathbf{Pa}, \mathbf{Pa})/\mathrm{OpR}\, \mathbf{Cat} (\mathbf{CoAlg})$, we can single out special objects as parenthesizations; morphisms in $(\mathbf{Pa}, \mathbf{Pa})/\mathrm{OpR}\, \mathbf{Cat}(\mathbf{CoAlg})$ fix these parenthesizations.
\end{definition}

Mirroring the notation we introduced in the operad $\mathbf{Pa}$, we define analogous operations in $(\mathbf{Pa}, \mathbf{Pa})/\mathrm{OpR}\, \mathbf{Cat}(\mathbf{CoAlg})$. For $1 \leq i \leq n$, set
\begin{equation} \label{strand_rem}
    \begin{split}
        s_i:\, \mathrm{End}_{\mathcal{O}(n)}(d^{n-1}(1)) &\longrightarrow \mathrm{End}_{\mathcal{O}(n-1)}(s_id^{n-1}(1))\\
         A_n & \longmapsto A_n \circ_i^{n,0} \mathbf{u};
    \end{split}
\end{equation}
\begin{equation}\label{strand_add}
    \begin{split}
        d_i: \hspace{1em} \mathrm{End}_{\mathcal{O}(n)}(d^{n-1}(1)) & \longrightarrow \mathrm{End}_{\mathcal{O}(n+1)}(d_id^{n-1}(1))\\
        A_n & \longmapsto A_n \circ_i^{n,2} \mathbf{m};
    \end{split}
\end{equation}
where $\mathbf{u}$ and $\mathbf{m}$ denote the multiplicative identity in $\mathrm{End}_{\mathcal{P}(0)}(1)$ and $\mathrm{End}_{\mathcal{P}(2)}(12)$, respectively\footnote{The maps $s_i$ and $d_i$ depend on both the operad (module) $\mathcal{O}$ and the arity $n$. The notation does not reflect this, to help with readability.}. The $\mathbb{S}$-module $\mathcal{O}$ can stand for either an operad $\mathcal{P}$ or an operad module $\mathcal{M}$. For the rest of this section, we will abbreviate $\Upsilon_n^\mathcal{M} \defeq \mathrm{End}_{\mathcal{M}(n)}(d^{n-1}(1))$. Notice that any $(f,g) \in \mathrm{Aut}_{(\mathbf{Pa},\mathbf{Pa)/\mathrm{OpR}\, ( \mathbf{CoAlg})}}(\mathcal{P}, \mathcal{M})$ defines Hopf algebra automorphisms $g_n: \Upsilon_n^\mathcal{M} \rightarrow \Upsilon_n^\mathcal{M} $, since morphisms fix parenthesizations.

\begin{lemma}
    Let $n \geq 1$ be fixed. The operations $s_i$ and $d_i$ satisfy the following properties:
    \begin{equation} \label{n_n_id}
            s_n \circ d_n = \mathrm{id}_{\Upsilon_n};
    \end{equation}
    \begin{equation} \label{n+1_n_id}
        s_{n+1} \circ d_n = \mathrm{id}_{\Upsilon_n}.
    \end{equation}
\end{lemma}
\begin{proof}
    To verify equation (\ref{n+1_n_id}), take $a \in \Upsilon_n$, then
    \begin{equation*}
        \begin{split}
            s_{n+1} \circ d_n (a) = (a \circ_n \mathbf{m}) \circ_{n+1} \mathbf{u}
            = a \circ_n (\mathbf{m} \circ_2 \mathbf{u})
             = a \circ_n \mathbf{1}
            =a,
        \end{split}
    \end{equation*}
    where we have denoted the multiplicative identity in $\mathrm{End}_{\mathcal{P}(1)}(1)$ by $\mathbf{1}$. The proof of equality (\ref{n_n_id}) is analogous.
\end{proof}

Our overarching goal is to construct bracket and cobracket operations from special operads in $\mathrm{OpR\, \mathbf{Cat}(\mathbf{CoAlg})}$. The candidates for this procedure still need to be endowed with additional structure.

\begin{definition} \label{delta_cat}
        Let $\mathrm{OpR}^{\Delta}$ be the category consisting of triplets $(\mathcal{P}, \mathcal{M}, t_\mathcal{P})$, where
        \begin{equation*}
            (\mathcal{P}, \mathcal{M})\in (\mathbf{Pa}, \mathbf{Pa})/\mathrm{OpR}\, \mathbf{Cat}(\mathbf{CoAlg})   
        \end{equation*}
        and $t_\mathcal{P}\in \mathrm{End}_{\mathcal{P}(1)}(1)$ satisfies that
        \[
            t_\mathcal{P} \circ_1^{1,0} \mathbf{u}=0.
        \]
        Morphisms in $\mathrm{OpR}^{\Delta}$ are the morphisms in $(\mathbf{Pa}, \mathbf{Pa})/\mathrm{OpR}\, \mathbf{Cat}(\mathbf{CoAlg})$ that preserve the marked element $t_{\mathcal{P}}$.
        In other words, $\mathrm{OpR}^{\Delta}$ is the category of elements of the functor 
        \begin{align*}
            (\mathbf{Pa}, \mathbf{Pa})/\mathrm{OpR}\, \mathbf{Cat}(\mathbf{CoAlg}) & \to \mathbf{Set} \\
            (\mathcal{P}, \mathcal{M}) & \mapsto \ker(\mathrm{End}_{\mathcal{P}(1)}(1) \overset{\circ_1^{1,0} \mathbf{u}}{\longrightarrow} \mathrm{End}_{\mathcal{P}(0)}(\varnothing,\varnothing)).
        \end{align*}
        
\end{definition}
The first condition also imposes requirements on the module-map component of morphisms in $\mathrm{OpR}^\Delta$. For all $n \geq 0$, define operations
        \begin{equation*}
            \begin{split}
                q_i^\mathcal{M}: \hspace{1 em} \mathcal{M}(n) \otimes \mathcal{P}(1) & \longrightarrow \mathcal{M}(n)\\
                m \otimes p &\longmapsto m \circ_i^{n,1} p.
            \end{split}
        \end{equation*}
The commutativity of the diagram 
\begin{equation*}
    \begin{tikzcd}
        \Upsilon_n \otimes \mathrm{End}_{\mathcal{P}(1)}(1)  \arrow[d, "g \otimes f" ] \arrow[r, "\circ_i^{n,1}"] & \Upsilon_n \arrow[d, "g"]\\
        \Upsilon_n \otimes \mathrm{End}_{\mathcal{P}(1)}(1) \arrow[r, "\circ_i^{n,1}"] & \Upsilon_n
    \end{tikzcd}
\end{equation*}
for any morphism $(f: \mathcal{P} \rightarrow \mathcal{Q}, g: \mathcal{M} \rightarrow \mathcal{N}$), implies that
\begin{equation} \label{q_ik_f}
    \begin{tikzcd}
        q^{\mathcal{M}}_i(1 \otimes t_\mathcal{P}) \arrow[r, maps to, "g" ]&  q_i^{\mathcal{N}}(1 \otimes t_\mathcal{P})
    \end{tikzcd}.
\end{equation}
 
\begin{definition}
We define the category of Pre-Relative Extensions $\mathbf{PRelE}$ to be the category whose objects are commutative diagrams between Lie algebras of the sort:
    \begin{equation}\label{PRelE_ob}
        \begin{tikzcd}
            \tilde{\mathfrak{h}} \arrow[r, "\pi"] \arrow[d, "s"] & \mathfrak{h} \arrow[d, "f"] \\
        \mathfrak{e} \arrow[r, "p"] & \mathfrak{g}
        \end{tikzcd}, 
    \end{equation}
with the added requirements that the Lie map $p: \mathfrak{e} \rightarrow \mathfrak{g}$ be surjective and have an abelian kernel. Morphisms in $\mathbf{PRelE}$ consist of Lie maps $E: \mathfrak{e} \rightarrow \mathfrak{e}'$, $G: \mathfrak{g} \rightarrow \mathfrak{g}'$, $H: \mathfrak{h} \rightarrow \mathfrak{h}'$, and $\tilde{H}: \tilde{\mathfrak{h}} \rightarrow \tilde{\mathfrak{h}}'$, which slot into commutative diagrams:
    \begin{equation} \label{PRelE_mor}
        \begin{tikzcd}
            \mathfrak{h} \arrow[d, "H"] \arrow[r, "f"]& \mathfrak{g} \arrow[d, "G"]\\
            \mathfrak{h}' \arrow[r, "f'"]& \mathfrak{g}'
        \end{tikzcd},\hspace{1 em}
        \begin{tikzcd}
            \tilde{\mathfrak{h}} \arrow[r, "s"] \arrow[d, "\tilde{H}"]& \mathfrak{e} \arrow[d, "E"]\\
            \tilde{\mathfrak{h}}' \arrow[r, "s'"]& \mathfrak{e}'
        \end{tikzcd}, \hspace{1 em}
        \begin{tikzcd}
            \mathfrak{e} \arrow[r, "p"] \arrow[d, "E"] & \mathfrak{g} \arrow[d, "G"]\\
            \mathfrak{e}' \arrow[r, "p'"]& \mathfrak{g}'
        \end{tikzcd}, \hspace{1 em}
        \begin{tikzcd}
        \tilde{\mathfrak{h}} \arrow[r, "\pi"] \arrow[d, "\tilde{H}"] & \mathfrak{h} \arrow[d, "H"]\\
        \tilde{\mathfrak{h}}' \arrow[r, "\pi'"] & \mathfrak{h}'
        \end{tikzcd}.
    \end{equation}
\end{definition}

\begin{definition}
We define $\mathbf{PRelE}^\Gamma$ to be the subcategory of $\mathbf{PRelE}$ in which objects like (\ref{PRelE_ob}) come with a marked section $\gamma: \mathfrak{h} \rightarrow \tilde{\mathfrak{h}}$ of $\pi$,
    \begin{equation*}
        \begin{tikzcd}
            \tilde{\mathfrak{h}} \arrow[r, "\pi"] \arrow[d, "s"] & \mathfrak{h} \arrow[d, "f"] \arrow[l, bend right=40, swap, "\gamma"] \\
        \mathfrak{e} \arrow[r, "p"] & \mathfrak{g}
        \end{tikzcd}.
    \end{equation*}
    Morphisms in $\mathbf{PRelE}^\Gamma$ match those in $\mathbf{PRelE}$, but with the added requirement that they preserve the marked sections:
    \begin{equation} \label{PRelE_gamma}
        \begin{tikzcd}
            \mathfrak{h} \arrow[r, "\gamma"] \arrow[d, "H"] & \tilde{\mathfrak{h}} \arrow[d, "\tilde{H}"]\\
            \mathfrak{h} \arrow[r, "\gamma'"] & \tilde{\mathfrak{h}}' 
        \end{tikzcd}.
    \end{equation}
\end{definition}

\begin{proposition}\label{Prele_to_rele}
    There exists a functor $\mathrm{P}: \mathbf{PRelE}^\Gamma \rightarrow \mathbf{RelE}$ defined on objects and morphisms, respectively, by 
    \begin{equation*}
        \begin{tikzcd}
            \tilde{\mathfrak{h}} \arrow[r, "\pi"] \arrow[d, "s"] & \mathfrak{h} \arrow[d, "f"] \arrow[l, bend right=40, swap, "\gamma"] \\
            \mathfrak{e} \arrow[r, "p"] & \mathfrak{g}
        \end{tikzcd} \longmapsto
        \begin{tikzcd}
            & \mathfrak{h} \arrow[dl, swap, "s \circ \gamma"] \arrow[d, "f"] \\
            \mathfrak{e} \arrow[r, "p"] & \mathfrak{g}
        \end{tikzcd}.
    \end{equation*}
    \small
    \begin{equation*}
            \left\{\begin{tikzcd}
            \mathfrak{h} \arrow[d, "H"] \arrow[r, "f"]& \mathfrak{g} \arrow[d, "G"]\\
            \mathfrak{h}' \arrow[r, "f'"]& \mathfrak{g}'
        \end{tikzcd},\,
        \begin{tikzcd}
            \tilde{\mathfrak{h}} \arrow[r, "s"] \arrow[d, "\tilde{H}"]& \mathfrak{e} \arrow[d, "E"]\\
            \tilde{\mathfrak{h}}' \arrow[r, "s'"]& \mathfrak{e}'
        \end{tikzcd}, \,
        \begin{tikzcd}
            \mathfrak{e} \arrow[r, "p"] \arrow[d, "E"] & \mathfrak{g} \arrow[d, "G"]\\
            \mathfrak{e}' \arrow[r, "p'"]& \mathfrak{g}'
        \end{tikzcd}, \,
        \begin{tikzcd}
        \tilde{\mathfrak{h}} \arrow[r, "\pi"] \arrow[d, "\tilde{H}"] & \mathfrak{h} \arrow[d, "H"]\\
        \tilde{\mathfrak{h}}' \arrow[r, "\pi'"] & \mathfrak{h}'
        \end{tikzcd} \right\} \mapsto
        \left\{\begin{tikzcd}
            \mathfrak{h} \arrow[d, "H"] \arrow[r, "f"]& \mathfrak{g} \arrow[d, "G"]\\
            \mathfrak{h}' \arrow[r, "f'"]& \mathfrak{g}'
        \end{tikzcd},\,
        \begin{tikzcd}
            \mathfrak{h} \arrow[r, "s \circ \gamma"] \arrow[d, "H"] & \mathfrak{e} \arrow[d, "E"]\\
            \mathfrak{h}' \arrow[r, "s' \circ \gamma'"] & \mathfrak{e}'
        \end{tikzcd},\,
        \begin{tikzcd}
            \mathfrak{e} \arrow[r, "p"] \arrow[d, "E"] & \mathfrak{g} \arrow[d, "G"]\\
            \mathfrak{e}' \arrow[r, "p'"]& \mathfrak{g}'
        \end{tikzcd}\right\}
    \end{equation*}
    \normalsize
\end{proposition}
\begin{proof}
    To check that $\mathrm{P}$ is well-defined, we only need to verify the commutativity of the diagram
    \begin{equation*}
        \begin{tikzcd}
            \mathfrak{h} \arrow[r, "s \circ \gamma"] \arrow[d, "H"] & \mathfrak{e} \arrow[d, "E"]\\
            \mathfrak{h}' \arrow[r, "s' \circ \gamma'"] & \mathfrak{e}'
        \end{tikzcd},
    \end{equation*}
    but this follows from the imposed condition (\ref{PRelE_gamma}).
\end{proof}

We now mean to define a functor from $\mathrm{OpR}^\Delta$ into $\mathbf{PRelE}$. Throughout the following proposition, we will implicitly identify Hopf algebra maps with their restriction to primitive elements. The bracket between two primitive elements $x$ and $y$ is defined as
    \begin{equation*}
        [x, y] = x \cdot y-y\cdot x,
    \end{equation*}
where $\cdot$ stands for the algebraic product in the underlying Hopf algebra. Whenever $X \subset \mathfrak{h}$ is a subset of a Lie algebra, we denote the ideal generated by $X$ by $\langle X \rangle_\mathfrak{h}$. Unless it can lead to confusion, we usually drop the subscript $_\mathfrak{h}$. 
\begin{proposition} \label{gen_squares}
    For each $n \geq 2$, there exist functors $\Pi_n: \mathrm{OpR}^\Delta \rightarrow \mathbf{PRelE}$. On objects, $\Pi_n$ is defined by 
    \begin{equation*}
        (\mathcal{P}, \mathcal{M}, t_\mathcal{P}) \longmapsto
        \begin{tikzcd} [column sep= large]
            K_n \arrow[r, "\pi"] \arrow[d, "d_n"] & K_n/\langle q_n(1 \otimes t_\mathcal{P}) \rangle \arrow[d, "\Delta"]\\
           \mathfrak{g}_{n+1} \arrow[r, "s_{n+1} \oplus s_{n}"] & (K_n/\langle q_n(1 \otimes t_\mathcal{P})\rangle)^{\oplus 2}
        \end{tikzcd},
    \end{equation*}
    where
    \begin{equation*}
        \mathfrak{g}_{n+1} \defeq (H_{n+1}/ \bigoplus_{i\in\{n,n+1\}} \langle q_i(1 \otimes t_{\mathcal{P}}) \rangle)/\langle[\ker (s_{n+1} \oplus s_n), \ker (s_{n+1} \oplus s_n)] \rangle.
    \end{equation*}
    We have also set:
    \begin{itemize}
        \item $K_n \defeq \ker (s_n: \Upsilon_n\rightarrow \Upsilon_{n-1})$;\\
        \item $H_n \defeq \ker (s_{n-1} \circ s_{n}^: \Upsilon_{n}\rightarrow \Upsilon_{n-2} )$;\\
        \item The map $\Delta$ is the diagonal map, so that $\Delta(m)=m \oplus m$ for all $m\in K_n/\langle q_n(1 \otimes t_\mathcal{P}) \rangle$.
    \end{itemize}
    On morphisms, $\Pi_n$ is defined in the natural way. Explicitly,
    \begin{equation*}
        \begin{tikzcd}
            (\mathcal{P}, \mathcal{M}) \arrow[d, "g", xshift=1.5ex] \arrow[d, "f", xshift=-1.5ex, swap]\\
            (\mathcal{P}', \mathcal{M}') 
        \end{tikzcd} \longmapsto
        \left(
        \begin{tikzcd}
             K_n^\mathcal{M} \arrow[d, "g_n"] \\
             K_n^\mathcal{N}
        \end{tikzcd},
        \begin{tikzcd}
            K_n^\mathcal{M}/ \langle q_n(1 \otimes t_\mathcal{P}) \rangle \arrow[d, "g_n"]\\
            K_n^\mathcal{N}/ \langle q_n(1 \otimes t_\mathcal{P}) \rangle
        \end{tikzcd},
        \begin{tikzcd}
            (K_n^\mathcal{M}/ \langle q_n(1 \otimes t_\mathcal{P}) \rangle)^{\oplus^2} \arrow[d, "g_n \oplus g_n"]\\
            (K_n^\mathcal{N}/ \langle q_n(1 \otimes t_\mathcal{P}) \rangle)^{\oplus^2}
        \end{tikzcd},
        \begin{tikzcd}
            \mathfrak{g}_{n+1}^\mathcal{M} \arrow[d, "g_{n+1}"]\\
            \mathfrak{g}_{n+1}^\mathcal{N}
        \end{tikzcd}
        \right)
    \end{equation*}
\end{proposition}
\begin{proof}
    The functorial nature of the map $\Pi_n$ is apparent, since we essentially only used the operadic composition to construct the corresponding object in $\mathbf{PRelE}$. We only need to check that the functor is well-defined. We will not distinguish between Lie algebra maps and the maps they induce on various quotients.

    Take $a \in K_n$. In accordance with equations (\ref{n_n_id}) and (\ref{n+1_n_id}), 
    \begin{equation*}
        s_{n+1} \oplus s_n(d_n(a))=a \oplus a = \Delta(a),
    \end{equation*}
    which implies the diagram $\Pi_n(\mathcal{P},\mathcal{M}, t_\mathcal{P})$ is commutative. 
    
    Now take instead $a \in K_n$, then
    \begin{equation*}
    s_{n} \circ s_{n+1}(d_n(a))= s_n(a)=0
    \end{equation*}
    by assumption, implying that the map $d_n:K_n \rightarrow \mathfrak{g}_{n+1}$ is well-defined.

    Next, we would like to prove the map
    \begin{equation*}
        s_{n+1} \oplus s_n : \mathfrak{g}_{n+1} \rightarrow (K_n/\langle q_n (1 \otimes t_\mathcal{P})\rangle)^{\oplus2}
    \end{equation*}
    is well-defined. Firstly, notice the equations
    \begin{equation*}
        s_n \circ s_{n+1}(a)=(a \circ_{n+1} \mathbf{u}) \circ_n \mathbf{u}=(a \circ_{n} \mathbf{u}) \circ_n \mathbf{u}=s_n \circ s_n (a).
    \end{equation*}
    These imply that if $a \in H_{n+1}$, then $s_{n+1} \oplus s_n (a) \in K_n \oplus K_n$. We also need to prove that
    \begin{equation*}
        s_{n+1} \oplus s_n (q_i(1 \otimes t_\mathcal{P})) \in \langle q_n(1 \otimes t_\mathcal{P})\rangle
    \end{equation*}
    for $i=n, n+1$. But notice that
    \begin{equation*}
        \begin{split}
            s_n \circ q_n(1 \otimes t_\mathcal{P}) &= (1 \circ_n t_\mathcal{P}) \circ_n \mathbf{u}\\
            &=(1 \circ_n \mathbf{u}) \circ_n t_\mathcal{P}\\
            &= 1 \circ_n t_\mathcal{P}\\
            &=q_n(1 \otimes t_\mathcal{P});
        \end{split}
    \end{equation*}
    \begin{equation*}
        \begin{split}
            s_{n+1} \circ q_n(1 \otimes t_\mathcal{P}) &= (1 \circ_n t_\mathcal{P}) \circ_{n+1} \mathbf{u}=1 \circ_n (t_\mathcal{P} \circ_1 \mathbf{u})=0;
        \end{split}
    \end{equation*}
    where for the last equality we used the second requirement from Definition \ref{delta_cat}. Similarly, $s_{n} \circ q_{n+1}(1  \otimes t_\mathcal{P})=q_n(1 \otimes t_\mathcal{P})$ and $s_{n+1} \circ q_{n+1}(1 \otimes t_\mathcal{P})=0$.
    
    The previous arguments prove that the assignment $\Pi_n$ is well-defined on objects. With regards to morphisms, the only possible point of contention comes from quotienting by subspaces of the sort $\bigoplus_{i\in I} \langle q_i^\mathcal{M} (1 \otimes t_p) \rangle$, but notice how equation (\ref{q_ik_f}) already shows that these subalgebras are equivariant with respect to operadic module maps.
\end{proof}

For a fixed $n \geq 2$ and $(\mathcal{P}, \mathcal{M},t_\mathcal{P}) \in \mathrm{OpR}^\Delta$, let us introduce some shorthand notation to do with the components of $\Pi_n(\mathcal{P}, \mathcal{M}, t_\mathcal{P})$:
\begin{itemize}
    \item $\overline{K}_n \defeq K_n/\langle q_n(1 \otimes t_\mathcal{P}) \rangle$;
    \item $M \defeq \ker (s_{n+1} \oplus s_n:\mathfrak{g}_{n+1} \rightarrow \overline{K}_n \oplus \overline{K}_n)$.
\end{itemize}
Given Proposition \ref{gen_squares}, it is clear we still need to enhance objects in $\mathrm{OpR}^\Delta$ with both a choice of section $\gamma: \overline{K}_n \rightarrow K_n$ of the canonical projection, and a $\overline{K}_n \oplus \overline{K}_n$-module isomorphism between $M$ and $U\overline{K}_n$. To reiterate, $M$ is a $\overline{K}_n \oplus \overline{K}_n$-module with respect to the action
\begin{equation*}
    (x\oplus y) \cdot m \defeq [m, u(x \oplus y)],
\end{equation*}
where $u: \overline{K}_n \oplus \overline{K}_n \rightarrow \mathfrak{g}_{n+1}$ is a section\footnote{Since $M$ is abelian by definition, this action is independent from the choice of section} of the projection $s_{n+1} \oplus s_n:\mathfrak{g}_{n+1} \rightarrow \overline{K}_n$. On the other hand, $U\overline{K}_n$ is a $\overline{K}_n \oplus \overline{K}_n$-module with respect to the action
\begin{equation*}
    (x\oplus y) \cdot a= xa-ay.
\end{equation*}
We now introduce a \textit{framed} counterpart to $\mathrm{OpR}^\Delta$:
\begin{definition}\label{op_delta}
    For a fixed $n \geq 2$, let $\mathrm{OpR}^\Delta_f(n)$ be the category consisting of triplets
    \begin{equation*}
        \left( (\mathcal{P}, \mathcal{M}, t_\mathcal{P}) \in \mathrm{OpR^\Delta},\, \gamma: \overline{K}_n \rightarrow K_n,\, \phi: M \rightarrow U\overline{K}_n \right),
    \end{equation*}
    where $\gamma$ is a section of the natural projection $\pi: K_n \rightarrow \overline{K}_n$, and $\phi$ is an isomorphism of $\overline{K}_n \oplus \overline{K}_n$-modules. We also add the requirement that $\overline{K}_n$ must be isomorphic to a free Lie algebra.
    Morphisms in $\mathrm{OpR}^\Delta_f(n)$ consist of morphisms in $\mathrm{OpR}^\Delta$ that preserve the marked sections and module isomorphisms. More precisely,
\begin{equation*}
    (F,G) \in \mathrm{mor}_{\mathrm{OpR}^\Delta_f(n)}
    \left( 
        \left(
            (\mathcal{P}, \mathcal{M}, t_\mathcal{P}),\,
            \begin{tikzcd} [row sep =small]
                \overline{K}_n \arrow[d, "\gamma"]\\
                K_n
            \end{tikzcd},\,
            \begin{tikzcd}[row sep=small]
                M \arrow[d, "\phi"]\\
                U\overline{K}_n
            \end{tikzcd}
        \right),
        \left(
            (\mathcal{P}', \mathcal{M}', t_\mathcal{P'}),\,
            \begin{tikzcd} [row sep =small]
                \overline{K}_n' \arrow[d, "\gamma'"]\\
                K_n'
            \end{tikzcd},\,
            \begin{tikzcd}[row sep=small]
                M' \arrow[d, "\phi'"]\\
                U\overline{K}_n'
            \end{tikzcd}
        \right)
    \right)
\end{equation*}
if and only if $(F,G)$ is a morphism between the corresponding objects in $\mathrm{OpR}^\Delta$ and the following diagrams are commutative:
\begin{equation*}
    \begin{tikzcd}
        \overline{K}_n \arrow[r, "\gamma"] \arrow[d, "G"] & K_n \arrow[d, "G"]\\
        \overline{K}_n' \arrow[r, "\gamma'"] & K_n'
    \end{tikzcd}, \hspace{1em}
    \begin{tikzcd}
        M \arrow[r, "\phi"] \arrow[d, "G"] & U \overline{K}_n \arrow[d, "G"] \\
        M' \arrow[r, "\phi'"] & U\overline{K}_n'
    \end{tikzcd}.
\end{equation*}
\end{definition}

\begin{proposition}
    Let $n \geq 2$. There exist a functors $\mathrm{D}_n: \mathrm{OpR}^\Delta_f(n) \rightarrow \mathbf{RelE}^\eta$, defined on objects by
    \begin{equation}
        \left((\mathcal{P}, \mathcal{M}, t_\mathcal{P}),\,
        \begin{tikzcd}\label{D_of_op}
            \overline{K}_n \arrow[d, "\gamma"] \\
            K_n
        \end{tikzcd},\,
        \begin{tikzcd}
            M \arrow[d, "\phi"]\\
            U\overline{K}_n
        \end{tikzcd}\right) \longmapsto
        \left( \mathcal{A}^\gamma \defeq\begin{tikzcd} [column sep=large]
            & \overline{K}_n \arrow[dl, "d_n \circ \gamma" swap] \arrow[d, "\Delta"]\\
            \mathfrak{g}_{n+1} \arrow[r, "s_{n+1} \oplus s_n" swap] & \overline{K}_n \oplus \overline{K}_n
        \end{tikzcd},\,
        \begin{tikzcd}
            M \arrow[d, "\phi"]\\
            U\overline{K}_n
        \end{tikzcd}\right),
    \end{equation}
    and in the natural way on morphisms.
\end{proposition}
\begin{proof}
Let $((\mathcal{P}, \mathcal{M}), \gamma, \phi)\in \mathrm{OpR}^\Delta_f(n)$. Firstly, notice that $\mathrm{D}_n$ is well-defined, as $M = \ker s_{n+1} \oplus s_n \subset \mathfrak{g}_{n+1}$ is expressly abelian. By Propositions \ref{Prele_to_rele} and \ref{gen_squares}, the assignment
\begin{equation*}
    ((\mathcal{P}, \mathcal{M},t_\mathcal{P}), \gamma, \phi) \longmapsto
    \begin{tikzcd} [column sep=large]
            K_n \arrow[r, "\pi"] \arrow[d, "d_n"] & \overline{K}_n \arrow[d, "\Delta"] \arrow[l, bend right=40, swap, "\gamma"] \\
            \mathfrak{g}_{n+1} \arrow[r, "s_{n+1} \oplus s_n"] & \overline{K}_n \oplus \overline{K_n}
        \end{tikzcd} \longmapsto
         \mathcal{A}^\gamma
\end{equation*}
is functorial. Lastly, morphisms in $\mathrm{OpR}^\Delta_f(n)$ preserve the marked map $\phi$ by definition.
\end{proof}

\begin{definition}\label{def_opr_del_f}
    Let $n \geq 2$. We define the functor $\Lambda_n: \mathrm{OpR}^\Delta_f(n) \rightarrow \mathbf{GoTu}$ as the composition
    \begin{equation}
        \begin{tikzcd}
            \Lambda_n: \hspace{0.5 em}\mathrm{OpR}^\Delta_f(n) \arrow[r, "\mathrm{D}_n"] & \mathbf{RelE}^\eta \arrow[r, "\mathrm{F}^{-1}"] &\mathbf{RelCo}^\eta \arrow[r, "\mathrm{E}^{-1}"] & \mathbf{Fox}^\eta \arrow[r, "\Psi"] & \mathbf{GoTu}
        \end{tikzcd}.
    \end{equation}
\end{definition}
\begin{remark}
    The composition $(\mathrm{E}^{-1})(\mathrm{F}^{-1})\mathrm{D}_n: \mathrm{OpR}^\Delta_f(n) \rightarrow \mathbf{Fox}$ is what we called the functor $\mathrm{IC}_n$ in the Introduction.
\end{remark}
Rather than giving a convoluted formula for the bracket and cobracket operations extractable from an arbitrary object of $\mathrm{OpR}^\Delta_f(n)$, we will focus on a specific example, and compute the operations associated to Gonzalez' operad module of \textit{framed} genus $g$ chord diagrams $\mathbf{PaCD}^f_g$. Nevertheless, although the bracket and cobracket may be difficult to discern, we note that their underlying Hopf algebra is just the space $U\overline{K}_n$.

\subsection{Infinitesimal braid algebras}\label{sec_braid_alg}

The usual operad of parenthesized chord diagrams is defined in terms of the Drinfeld-Kohno Lie algebras $\mathfrak{t}_n$. The operad $\mathbf{PaCD^f}$ and the operad module $\mathbf{PaCD}^f_g$ are instead concerned with framed and higher genus analogues of $\mathfrak{t}_n$, respectively. We spend this subsection studying these infinitesimal braid algebras, denoted by $\mathfrak{t}^f_n$ and $\mathfrak{t}_{g,n}^f$. Recall how we need to have some non-trivial information regarding the endomorphism spaces of operad modules before we can extract bracket and cobracket operations from them (cf. Definition \ref{def_opr_del_f}). This section is mostly technical, in the sense that our goal is to gain this information by studying several quotients of the algebras $\mathfrak{t}_{g,n}^f$.

\begin{definition}
    Let $\mathfrak{t}_n^f$ denote the degree completion of the graded Lie algebra over $\mathbb{K}$ with generators $t_{ij}$ for $1 \leq i,j\leq n$, subject to relations
    \begin{equation} \label{FS} \tag{FS}
        t_{ij}=t_{ji}, \hspace{1em} 1\leq i,j \leq n;
    \end{equation}
    \begin{equation} \label{FL} \tag{FL}
        [t_{ij}, t_{kl}]=0 \hspace{1 em} \text{if}\, \, \{i,j\}\cap \{k,l\}=\varnothing;
    \end{equation}
    \begin{equation} \label{F4T} \tag{F4T}
        [t_{ij}, t_{ik}+t_{jk}]=0\hspace{1em } \text{if}\,\,\{i,j\}\cap \{k\}=\varnothing.
    \end{equation}
    We will call the $\mathfrak{t}_n^f$ the \textit{framed Drinfeld-Kohno Lie algebras}.
\end{definition}

\begin{remark}
    The generators $t_{ii}$ are central for all $i=1, \dots n$, so that $\mathfrak{t}_n^f$ decomposes as the direct sum of Lie algebras
    \begin{equation}
        \mathfrak{t}_n^f= \bigoplus_{i=1} ^n \mathbb{K}t_{ii} \oplus\mathfrak{t}_n,
    \end{equation}
    where $\mathfrak{t}_n$ stands for the usual Drinfeld-Kohno Lie algebra in $n$ generators.
\end{remark}

\begin{proposition}[\cite{gonz}, \S 2.2] 
    The Lie algebras $\mathfrak{t}_n^f$ furnish an operad $\mathfrak{t}^f$ in the category $\mathbf{Lie}$ of Lie algebras over $\mathbb{K}$. We write
    \begin{equation*}
        \mathfrak{t}^f \defeq \{\mathfrak{t}_n^f\}_{n \geq 0},
    \end{equation*}
    and define the partial compositions by
    \begin{equation}
        \begin{split}
            \circ_k: \hspace{1em} \mathfrak{t}_{I}^f \oplus \mathfrak{t}_{J}^f & \rightarrow \mathfrak{t}_{J \sqcup (I -\{k\})}^f\\
            (0, t_{\alpha \beta}) &\mapsto t_{\alpha \beta}\\
            (t_{ij},0) &\mapsto \begin{cases}
                t_{ij} & \text{if}\,\, k\notin\{i,j\}\\
                \sum_{p \in J} t_{pj} & \text{if}\,\, k=i, k\neq j\\
                \sum_{\{p,q\}\subset J} t_{pq} & \text{if}\,\, k=i=j
            \end{cases}
        \end{split}
    \end{equation}
    whenever $I$ and $J$ are finite sets with $k\in I$.
\end{proposition}

\begin{definition}
    Let $\mathfrak{t}_{g,n}^f$ denote the degree completion of the graded Lie algebra over $\mathbb{K}$ with generators $t_{ij}$, for $1 \leq i,j \leq n$; $x^a_i$, $y^a_i$ for $1 \leq i \leq n$ and $1 \leq a \leq g$, where
    \begin{equation*}
        \deg(t_{ij})=2, \hspace{1em} \deg(x^a_i)=\deg(y^a_i)=1,
    \end{equation*}
    subject to relations (\ref{FS}), (\ref{FL}), (\ref{F4T}), as well as
    \begin{equation}\label{Sg} \tag{$\mathrm{S}_g$}
        [x_i^a, y_j^b]=\delta_{ab}t_{ij} \hspace{1em} \text{for all}\,\, i \neq j,    
    \end{equation}
    \begin{equation} \label{Ng} \tag{$\mathrm{N}_g$}
        [x_i^a, x_j^b]=[y_i^a, y_j^b]=0 \hspace{1em} \text{for all}\,\, i \neq j,
    \end{equation}
    \begin{equation} \label{FTg} \tag{$\mathrm{FT}_g$}
        \sum_{a=1}^g [x_i^a, y_i^a]=-\sum_{j:j\neq i} t_{ij}-2(g-1)t_{ii},
    \end{equation}
    \begin{equation} \label{FLg} \tag{$\mathrm{FL}_g$}
        [x_k^a, t_{ij}]=[y_k^a, t_{ij}]=0 \hspace{1em} \text{if}\,\, \{i,j\}\cap \{k\}=\varnothing,
    \end{equation}
    \begin{equation} \label{F4Tg} \tag{$\mathrm{F4T_g}$}
        [x_i^a+x_j^a, t_{ij}]=[y_i^a+y_j^a, t_{ij}]=0 \hspace{1em} \text{for all}\,\, i,j.
    \end{equation}
    We call the $\mathfrak{t}_{g,n}^f$ \textit{framed Drinfeld-Kohno Lie algebras of genus $g$}.
\end{definition}

\begin{proposition} \cite[\S 3.4]{gonz}
    The Lie algebras $\mathfrak{t}_{g,n}^f$ furnish a right module over the operad of framed Drinfeld-Kohno Lie algebras, $\mathfrak{t}^f$. We write
    \begin{equation*}
        \mathfrak{t}_{g}^f \defeq \{\mathfrak{t}_{g,n}^f\}_{n \geq 0},
    \end{equation*}
    and define the partial compositions by
    \begin{equation}\label{op_tn}
        \begin{split}
            \circ_k: \hspace{1em} \mathfrak{t}_{g, I}^f \oplus \mathfrak{t}_J^f & \rightarrow \mathfrak{t}_{q, J \sqcup (I - \{k\})}^f\\
            (0, t_{\alpha \beta}) & \mapsto t_{\alpha \beta}\\
            (t_{ij},0) &\mapsto \begin{cases}
                t_{ij} & \text{if}\,\, k\notin\{i,j\}\\
                \sum_{p \in J} t_{pj} & \text{if}\,\, k=i, k\neq j\\
                \sum_{\{p,q\}\subset J} t_{pq} & \text{if}\,\, k=i=j
            \end{cases}\\
            (x_i^a,0) & \mapsto \begin{cases}
                x_i^a & \text{if}\,\, k \neq i\\
                \sum_{p \in J} x_{p}^a & \text{if}\,\, k=i
            \end{cases}\\
            (y_i^a,0) & \mapsto \begin{cases}
                y_i^a & \text{if}\,\, k \neq i\\
                \sum_{p \in J} y_{p}^a & \text{if}\,\, k=i
            \end{cases}
        \end{split}
    \end{equation}
\end{proposition}

    Recall the symmetric monoidal functor
        \begin{equation*}
            \hat{\mathrm{U}}: \mathbf{grLie} \rightarrow \mathbf{Cat(CoAlg)},
        \end{equation*}
    which to a positively graded Lie algebra $\mathfrak{g}$ assigns the category with a single object, and with morphisms given by the degree completion of the universal enveloping algebra of $\mathfrak{g}$. We thus define:
    \begin{itemize}
        \item The \textit{operad of framed chord diagrams}, $\mathbf{CD}^f \defeq \hat{\mathrm{U}}(\mathfrak{t}^f)$;
        \item The $\mathbf{CD}^f$-\textit{operad module of framed genus $g$ chord diagrams}, $\mathbf{CD}^f_g \defeq \hat{\mathrm{U}}(\mathfrak{t}^f_g)$.
    \end{itemize}

\begin{definition}
    Consider the operadic composition 
    \begin{equation*}
        \circ_k: \mathfrak{t}_{g,n}^f \oplus \mathfrak{t}_{\{k,k+1\}}^f \rightarrow \mathfrak{t}_{g,n+1}^f.
    \end{equation*}
    We define \textit{string splitting} maps $d_k: \mathfrak{t}_{g,n}^f \rightarrow \mathfrak{t}_{g,n+1}^f$ for $k=1, \dots, n$ by
    \begin{equation}
        d_k(a)\defeq \circ_k (a,0) \hspace{1em} \text{for all} \,\, a\in \mathfrak{t}_{g,n}^f.
    \end{equation}
\end{definition}

\begin{definition}
    Consider the operadic composition 
    \begin{equation*}
        \circ_k: \mathfrak{t}_{g,n}^f \oplus \mathfrak{t}_{g,0}^f \rightarrow \mathfrak{t_{g,n-1}^f}.
    \end{equation*}
    We define \textit{string deletion} maps $s_k: \mathfrak{t}_{g,n}^f \rightarrow \mathfrak{t}_{g,n-1}^f$ for $k=1,\dots, n$ by
    \begin{equation}
        s_k(a)\defeq \circ_k(a,0) \hspace{1em} \text{for all}\,\, a\in \mathfrak{t}_{g,n}^f.
    \end{equation}
\end{definition}

\begin{proposition}
    Define $K_{g,n}= \ker (s_n: \mathfrak{t}_{g,n}^f \rightarrow \mathfrak{t}^f_{g,n-1})$ and $H_{g,n}=\ker(s_{n-1}\circ s_n: \mathfrak{t}_{g,n}^f \rightarrow \mathfrak{t}_{g,n-2}^f)$. The following are split exact sequences:
    \begin{equation*}
            \begin{tikzcd}[column sep= large, row sep= tiny]
            K_{g, l} \arrow[r] & \mathfrak{t}_{g,l} \arrow[r, "s_{l}"] & \mathfrak{t}_{g,l}^f,\\
            H_{g, l} \arrow[r] & \mathfrak{t}_{g,l} \arrow[r, "s_{l-1}\circ s_{l}"] & \mathfrak{t}_{g,l-2}^f
            \end{tikzcd},
        \end{equation*}
\end{proposition}
\begin{proof}
    Clearly, the map $s_n$ is surjective, and we can check by direct computation that $d_{n-1}: \mathfrak{t}_{g,n-1}^f \rightarrow \mathfrak{t}_{g,n}^f$ is a section of $s_n$.
\end{proof}

It will be useful to describe $K_{g,n}$ and $H_{g,n}$ in terms of generators and relations in $\mathfrak{t}_{g,n}^f$. The proofs of the following two propositions will be relegated to Appendix \ref{F-well-def}:
\begin{proposition}\label{def_kg,n}
    Let $\mathfrak{k}_{g,n}$ denote the degree completion of the free Lie algebra in $2g +n$ generators, 
    \begin{equation*}
        L(x_n^a, y_n^a, t_{in}, \dots ,t_{nn})_{1 \leq a \leq g},    
    \end{equation*}
    subject to the relation
    \begin{equation*}
         \sum_{a=1}^g [x_n^a, y_n^a]=-\sum_{j:j\neq n} t_{nj}-2(g-1)t_{nn}
    \end{equation*}
    and the requirement that $t_{nn}$ be central. Then $\mathfrak{k}_{g,n}$ and $K_{g,n}$ are isomorphic as Lie algebras.
\end{proposition}
\begin{proposition}\label{def_hg,n}
    Let $\mathfrak{h}_{g,n}$ denote the degree completion of the free Lie algebra in $2(2g+n)-1$ generators,
    \begin{equation*}
        L(x_{n-1}^a,y_{n-1}^a,x_n^a, y_n^a,t_{1(n-1)}, \dots t_{n(n-1)},t_{1n},\dots,t_{(n-2)n},t_{nn}),
    \end{equation*}
    subject to the applicable relations in $\mathfrak{t}_{g,n}^f$. Then $\mathfrak{h}_{g,n}$ and $H_{g,n}$ are isomorphic as Lie algebras.
\end{proposition}

\begin{definition}
    We define the Lie algebras
    \begin{equation*}
        \overline{\mathfrak{k}}_{g,n} \defeq \mathfrak{k}_{g,n}/\langle t_{nn} \rangle,
    \end{equation*}
    \begin{equation*}
        \overline{\mathfrak{h}}_{g,n} \defeq \mathfrak{h}_{g,n}/ \langle t_{(n-1)(n-1)} ,t_{nn}\rangle,
    \end{equation*}
    the quotients of $\mathfrak{k}_{g,n}$ and $\mathfrak{h}_{g.n}$ by their central elements, respectively.
\end{definition}

As we will see, an element of the genus $g$ Grothendieck-Teichm\"uler group $\mathrm{GRT}_g$ describes automorphisms  in a variety of Lie quotients of $\mathfrak{t}_{g,n}^f$. For the remainder of this section, our goal is to relate these braid algebras with the Lie algebras that make up the diagram $\mathcal{A}^\gamma$ from equation (\ref{D_of_op}).

\begin{definition}
    We define additional quotient algebras
    \begin{equation*}
        \mathfrak{g}\defeq \overline{\mathfrak{h}}_{g,n+3}/[\langle t_{(n+2)(n+3)}\rangle, \langle t_{(n+2)(n+3)}\rangle]_{\overline{\mathfrak{h}}_{g,n+3}},
    \end{equation*}
    \begin{equation*}
        M \defeq \langle t_{(n+2)(n+3)} \rangle_{\overline{\mathfrak{h}}_{g,n+3}} / [\langle t_{(n+2)(n+3)} \rangle, \langle t_{(n+2)(n+3)} \rangle]_{\overline{\mathfrak{h}}_{g,n+3}}.
    \end{equation*}
\end{definition}

From now on, we will denote
\begin{equation*}
    \begin{split}
        L^{n+2} & \defeq \widehat{\mathrm{Lie}}(x_{n+2}^a, y_{n+2}^a, t_{2(n+2)}, \dots t_{(n+1)(n+2)}),
    \end{split}
\end{equation*}
where $1 \leq a \leq g$. We are interested in the relative cohomology spaces $H^*(L^{n+2} \oplus L^{n+2}, L^{n+2}; M)$, computed with respect to the diagonal map $\Delta: L^{n+2} \rightarrow L^{n+2} \oplus L^{n+2}$. The $L^{n+2} \oplus L ^{n+2}$-module structure we choose for $M$ is given by
\begin{equation*}
    (v \oplus w) \cdot m = [m, \alpha'(v \oplus w)],
\end{equation*} 
where the map $\alpha': L^{n+2} \oplus L^{n+2} \rightarrow \mathfrak{g}$ is fixed by the choices
\begin{equation*}
    \begin{tikzcd} [row sep=tiny]
        x_{n+2}^a \oplus 0 \arrow[r, maps to, "\alpha '"] & x_{n+2}^a,\\
        y_{n+2}^a \oplus 0 \arrow[r, maps to] & y_{n+2}^a,\\
        t_{i(n+2)} \oplus 0 \arrow[r, maps to] & t_{i(n+2)},\\
        0 \oplus x_{n+2}^a \arrow[r, maps to] & x_{n+3}^a,\\
        0 \oplus y_{n+2}^a \arrow[r, maps to] & y_{n+3}^a,\\
        0 \oplus t_{i(n+2)} \arrow[r, maps to] & t_{i(n+3)}.\\
    \end{tikzcd}
\end{equation*}

\begin{lemma}\label{ker_sis}
    The sequence
    \begin{equation*}
        \begin{tikzcd}
            0 \arrow[r]& M \arrow[r, "\iota"]& \mathfrak{g} \arrow[r, "P"] & L^{n+2} \oplus L^{n+2}
        \end{tikzcd},
    \end{equation*}
    where $P=s_{n+3} \oplus s_{n+2}$, is short-exact.
\end{lemma}
\begin{proof}
    The injectivity of $\iota$ is straightforward. Similarly, the surjectivity of $P$ follows from the fact that there are no braid relations between the elements $x_{n+2}^a$, $y_{n+2}^a$, and $t_{(j+1)(n+2)}$ in $\mathfrak{g}$, for any $1 \leq a \leq g$ and $1 \leq j \leq n$. The same is true if we substitute $n+2$ for $n+3$ in the preceding statement. The heart of the proof is then showing that $\ima \iota = \ker P$. Given that $t_{(n+2)(n+3)} \in \ker P$, we already know that $\ima \iota \subset \ker P$. For the reverse inclusion, notice the equality of vector spaces
    \begin{equation*}
        \mathfrak{g}= \langle t_{(n+2)(n+3)} \rangle_{\mathfrak{g}} \oplus ( x_{n+2}^a, y_{n+2}^a, t_{(j+1)(n+2)} )_{\mathfrak{g}} \oplus ( x_{n+3}^a, y_{n+3}^a, t_{(j+1)(n+3)} )_{\mathfrak{g}},
    \end{equation*}
    where the parentheses $( - )_{\mathfrak{g}}$ stand for the subalgebra generated in $\mathfrak{g}$ and $1 \leq a \leq g$, $1 \leq j \leq n$. It is clear no element in $\mathfrak{g}-\langle t_{(n+2)(n+3)} \rangle_{\mathfrak{g}}$ maps to zero under $P$.\footnote{Notice that this argument also hinges on the fact that $\langle t_{(n+2)(n+3)} \rangle_{\mathfrak{g}} \cap ( x_{n+2}^a, y_{n+2}^a, t_{(j+1)(n+2)} )_{\mathfrak{g}} = \varnothing$, for instance.}.
\end{proof}

Let $r: L^{n+2} \rightarrow \mathbb{K}$ be a linear map. We define a Lie morphism $d^r_{n+2}:L^{n+2} \rightarrow \mathfrak{g}$ by
\begin{equation*}
    d^r_{n+2}(x)=d_{n+2}(x) + r(x)t_{(n+1)(n+3)} \hspace{1em} \text{for all }x\in L^{n+2}.
\end{equation*}
We can check by direct computation that the following diagram is commutative:
\begin{equation*}
     \mathcal{D}^r \defeq 
    \begin{tikzcd}
         & L^{n+2} \arrow[d, "\Delta"] \arrow[dl, "d_{n+2}^r" swap]\\
         \mathfrak{g} \arrow[r, "P" swap] & L^{n+2} \oplus L^{n+2}
    \end{tikzcd}.
\end{equation*}
Since
\begin{equation*}
    M = \langle t_{(n+2)(n+3)} \rangle_{\overline{\mathfrak{h}}_{g,n+3}} / [\langle t_{(n+2)(n+3)} \rangle, \langle t_{(n+2)(n+3)} \rangle]_{\overline{\mathfrak{h}}_{g,n+3}}
\end{equation*}
is explicitly an abelian Lie algebra, it follows that $\mathcal{D}^r \in \mathrm{Ob}(\mathbf{RelE})$. 

Let $\tilde{\rho}_G \in \mathrm{Fox}(UL^{n+2})$ be the Fox pairing with non-vanishing entries given by
\begin{equation}\label{rho_G_tilde}
    \tilde{\rho}_G(x_{n+2}^i, y_{n+2}^i)=1, \hspace{1em} \tilde{\rho}_G(y_{n+2}^i, x_{n+2}^i)=-1, \hspace{1em} \tilde{\rho}_G(t_{j(n+2)},t_{j(n+2)})=-t_{j(n+2)},
\end{equation}
for $1 \leq i \leq g$ and $2 \leq j \leq n+1$. Likewise, let $q^r \in \mathrm{Qder}(-\tilde{\rho}_G)$ be the quasi-derivation uniquely specified by setting
\begin{equation}\label{q_r}
    q^r(a)=r(a) \hspace{1em}\text{for all generators }a\text{ of }L^{n+2}.
\end{equation}
Finally, let $T^r \oplus G \in Z^2(L^{n+2} \oplus L^{n+2}, L^{n+2};UL^{n+2})$ be the 2-cocycle corresponding to $q^r \oplus \tilde{\rho}_G \in Z^2(\mathcal{M}(\mathrm{id}_{L^{n+2}})) $ with respect to the equivalence of Theorem \ref{quasi-iso}. Explicitly,
\begin{equation}\label{qf_1}
    T^f\in \mathrm{Hom}_{\mathbb{K}}(L^{n+2}, UL^{n+2})\text{ and } T^r(x)=q^r(x)\text{ for all }x\in L^{n+2};
\end{equation}
\begin{equation}\label{rho_G_1}
     G \in \mathrm{Hom}_{\mathbb{K}}(\bigwedge^2 L^{n+2} \oplus L^{n+2}, UL^{n+2})\text{ and }G(a_1 \oplus a_2, b_1 \oplus b_2)=\tilde{\rho}_G(a_1, b_2)-\tilde{\rho}_G(b_1, a_2).
\end{equation}
The following proposition will be proved in Appendix \ref{iso_lemma}:
\begin{proposition} \label{Cgamma_is_C}
    The two commutative diagrams
    \begin{equation*}
        \mathcal{D}^r = 
        \begin{tikzcd}
         & L^{n+2} \arrow[d, "\Delta"] \arrow[dl, "d_{n+2}^r" swap]\\
         \mathfrak{g} \arrow[r, "P" swap] & L^{n+2} \oplus L^{n+2}
    \end{tikzcd} \hspace{1em} \text{and} \hspace{1em}
    \mathcal{C}^r=
    \begin{tikzcd}
        & L^{n+2} \arrow[d, "\Delta"] \arrow[dl, "\Delta \times T^r" swap]\\
         L^{n+2} \oplus L^{n+2} \times_G UL^{n+2} \arrow[r, "P_1" swap] & L^{n+2} \oplus L^{n+2}
    \end{tikzcd}
    \end{equation*}
    are isomorphic as objects in $\mathbf{RelE}$.
\end{proposition}

\subsection{Parenthesized chord diagrams}\label{cat_int_2}

In this section we compute the brackets and cobrackets associated to the $\mathbf{PaCD}^f$-operad module $\mathbf{PaCD}^f_g$. We also define the higher genus Grothendieck-Teichm\"uller groups $\mathrm{GRT}_g$. We will only deal with $\mathrm{GRT}$ from an operadic point of view, following \cite{gonz} and \cite{as-op}. We refer the reader to these sources for more details.

\begin{definition}[\cite{gonz}, \S 1]
     Let $\mathcal{P}$ and $\mathcal{Q}$ be two operads in groupoids. Given an operad morphism $f:\mathrm{Ob}(\mathcal{P}) \rightarrow \mathrm{Ob}(Q)$, we define the fake pullback of $\mathcal{Q}$ along $f$, $f^*\mathcal{Q}$, by:
    \begin{itemize}
        \item $\mathrm{Ob}(f^*\mathcal{Q}(n))=\mathrm{Ob}(\mathcal{P}(n))$.
        \item $\mathrm{Mor}_{f^*\mathcal{Q}(n)}(p,q)= \mathrm{Mor}_{\mathcal{Q}(n)}(p,q)$. 
    \end{itemize}
    The pullback construction carries through to the category of operad modules in groupoids.
\end{definition}

\begin{definition}[\cite{gonz}, \S 2.2.2]
\label{def_PaCD}
    \textit{Operad of parenthesized framed chord diagrams.} Let $\mathbf{PaCD}^f$ be the operad in $\mathbf{Cat(CoAlg)}$
        \begin{equation*}
            \mathbf{PaCD}^f = \omega_1^* (\mathbf{CD}^f),
        \end{equation*}
        where $\omega_1: \mathrm{Ob}(\mathbf{Pa)}\rightarrow \mathrm{Ob}(\mathbf{CD}^f)$ is the obvious operad morphism.
\end{definition}

\begin{definition}[\cite{gonz}, \S 3.4.2]  
\label{PaCDfg}
    \textit{$\mathbf{PaCD}^f$-module of parenthesized framed genus $g$ chord diagrams.} Let $\mathbf{PaCD}^f_g$ be the $\mathbf{PaCD}^f$-operad module in $\mathbf{Cat(CoAlg)}$
        \begin{equation*}
            \mathbf{PaCD}^f_g = \omega_2^* (\mathbf{CD}^f_2),
        \end{equation*}
        where $\omega_2: \mathrm{Ob}(\mathbf{Pa)}\rightarrow \mathrm{Ob}(\mathbf{CD}^f_g)$ is the obvious operad module morphism.
\end{definition}

Recall the notation
\begin{equation*}
    \Upsilon_n \defeq \mathrm{End}_{\mathbf{PaCD}^f(n)}(d^{n-1}(1)), \hspace{2em} \Upsilon_n^g \defeq \mathrm{End}_{\mathbf{PaCD}^f_g(n)}(d^{n-1}(1)).
\end{equation*}
Let $1\leq i \leq n$. Last section, we defined maps $s_i : \Upsilon_n^g \rightarrow \Upsilon_{n-1}^g$, $d_i: \Upsilon_n^g \rightarrow \mathrm{End}_{\mathbf{PaCD}^f_g(n)}(d_id^{n-1}(1))$ (see equations (\ref{strand_rem}) and (\ref{strand_add}), respectively). When we are dealing with an operad instead of an operad module, we can define additional operations, $d_0$ and $d_{n+1}$:
\begin{equation*}
    \begin{split}
        d_0: \hspace{1 em} \Upsilon_n & \longrightarrow \mathrm{End}_{\mathbf{PaCD}^f(n)}(d^{n}(1))\\
        A_n & \longmapsto \mathbf{m} \circ_2 A_n;
    \end{split} \hspace{2em}
    \begin{split}
        d_{n+1}: \hspace{1 em} \Upsilon_n & \longrightarrow \mathrm{End}_{\mathbf{PaCD}^f(n)}(d_{n+1}d^{n-1}(1))\\
        A_n & \longmapsto \mathbf{m} \circ_1 A_n;
    \end{split} \hspace{2em}
\end{equation*}
where as before, $\mathbf{m} \in \mathrm{End}_{\mathbf{PaCD}^f(2)}(12)$ stands for the algebraic unit.

As is common in the literature, we can represent morphisms in $\mathbf{PaCD}^f$ and $\mathbf{PaCD}^f_g$ graphically. A few distinguished ones are:
\begin{equation*}
    T \defeq t_{11} \cdot
    \begin{tikzcd}[row sep=large]
        1 \\
        1 \arrow[u]
    \end{tikzcd} \in \mathrm{End}_{\mathbf{PaCD}^f(1)}(1); 
    \hspace{2 em}
    X \defeq 1 \cdot
    \begin{tikzcd} [row sep= large, column sep=tiny]
        \hspace{0.5 em}2 & 1 \hspace{0.5 em} \\
         \hspace{0.5 em}1 \arrow[ur] & 2 \hspace{0.5 em}\arrow[ul]
    \end{tikzcd} \in \mathrm{Mor}_{\mathbf{PaCD}^f(2)}(12, 21);
\end{equation*}
\begin{equation*}
    a^{1,2,3} \defeq 1 \cdot 
    \begin{tikzcd} [row sep=large, column sep=small]
        1\,\, \,\,\,\, & (2 \,\,\,\, 3) \\
        (1 \,\,\,\, 2) \arrow[u, xshift=-1.1 ex] \arrow[ur, xshift=1 ex, shorten <= -.3em] &\,\,\,\,\,\,\,\, 3 \arrow[u, xshift=1.5 ex]
    \end{tikzcd} \in \mathrm{Mor}_{\mathbf{PaCD}^f(3)}((12)3, 1(23));
    \hspace{2 em}
    H \defeq t_{12} \cdot 
    \begin{tikzcd} [row sep=large]
        1\,\, 2\\
        1\,\, 2 \arrow[u, xshift= 1ex] \arrow[u, xshift=-1ex]
    \end{tikzcd} \in \mathrm{End}_{\mathbf{PaCD}^f(2)}(12).
\end{equation*}
\begin{proposition}[\cite{pavol}, \S 2.2] 
\label{pav_prop}
    The elements $T$, $X$, and $a^{1,2,3}$ are generators of the $\mathbf{PaCD}^f$-module $\mathbf{PaCD}^f_g$.
\end{proposition}
Indeed, for future reference, notice that we can write
\begin{equation} \label{H_from_T}
    H=d_1 T- d_0T -d_2T.
\end{equation}

\begin{definition}\label{GRT_def}
    The genus $g$ graded Grothendieck-Teichm\"uller group $\mathrm{GRT}_g$ is the automorphism group
    \begin{equation*}
        \mathrm{GRT}_g=\mathrm{Aut}_{\mathrm{OpR}\, \mathbf{Cat(CoAlg)}}^+(\mathbf{PaCD}^f,\mathbf{PaCD}^f_g),
    \end{equation*}
    where the superscript $+$ indicates that elements of $\mathrm{GRT}_g$ are the identity on objects. 
\end{definition}

Consider the operad $\mathbf{PaCD}^f$. Notice that
\begin{equation*}
    \mathrm{End}_{\mathbf{PaCD}^f(1)}(1)=U\mathfrak{t}_1^f=\mathbb{K}t_{11}.
\end{equation*}
We can observe that $(\mathrm{PaCD}^f, \mathrm{PaCD}^f_g, t_{11}) \in \mathrm{Ob}(\mathrm{OpR}^\Delta)$. Indeed, for all $n\in \mathbb{N}$,
\begin{equation*}
    \mathrm{Ob}(\mathbf{PaCD}^f(n))=\mathrm{Ob}(\mathbf{PaCD}^f_g(n))=\mathrm{Ob}(\mathbf{Pa}(n)).
\end{equation*}
Also, in accordance with the operadic composition of (\ref{op_tn}),
\begin{equation*}
    t_{ii} \circ_1 \mathbf{u}= \sum_{\{p,q\} \subset \varnothing} t_{pq}=0.
\end{equation*}

Now take $(f,g)\in \mathrm{GRT}_g$. Since elements of $\mathrm{GRT}_g$ are the identity on objects, then necessarily
\begin{equation*}
    f(t_{11})= k_f t_{11} \hspace{1em} \text{for some}\, k_f\in \mathbb{K}. 
\end{equation*}
\begin{definition}
    We define $\mathrm{GRT}^g_1$ to be the subgroup of $\mathrm{GRT_g}$ characterized by the condition that
    \begin{equation*}
        f(t_{11})=t_{11} \hspace{1em} \text{for all }\,(f,g) \in \mathrm{GRT}_1^g.
    \end{equation*}
\end{definition}

The previous discussion implies that $\mathrm{GRT}_1^g \hookrightarrow \mathrm{OpR}\, \mathbf{Cat(CoAlg)}^\Delta$. By Proposition \ref{gen_squares}, this inclusion guarantees the existence of a map
\begin{equation*}
    \begin{tikzcd}
        \mathrm{Aut}_{\mathrm{OpR}^\Delta}(\mathrm{PaCD}^f,\mathrm{PaCD}^f_g, t_{11}) \supset \mathrm{GRT}_1^g \arrow[r, "\Pi_{n+2}"] & \mathrm{Aut}_{\mathbf{PRelE}}(\Pi_{n+2}(\mathrm{PaCD}^f, \mathrm{PaCD}^f_g))
    \end{tikzcd}.
\end{equation*}
\begin{lemma}\label{C_in_prele}
    For $n \geq 0$ and all $g$, up to isomorphism in $\mathbf{PRelE}$,
    \begin{equation*}
        \Pi_{n+2}(\mathbf{PaCD}^f, \mathbf{PaCD}^f_g,t_{11})=
        \begin{tikzcd}
            \mathfrak{k}_{g,n+2} \arrow[r, " \pi"] \arrow[d, "d_{n+2}"] & L^{n+2} \arrow[d, "\Delta"]\\
            \mathfrak{g} \arrow[r, swap, "P"] & L^{n+2} \oplus L^{n+2}
        \end{tikzcd}.
    \end{equation*}
\end{lemma}
\begin{proof}
    This proof will amount to compiling some of our previous notation. Notice that
    \begin{equation*}
        q_i ^{\mathbf{PaCD}^f}(1 \otimes t_{11})=1 \circ_{i}^{n,1} t_{11}=t_{ii}.
    \end{equation*}
    Then, by definition,
        \begin{equation*}
            \Pi_{n+2}(\mathbf{PaCD}^f, \mathbf{PaCD}^f_g, t_{11})=
        \begin{tikzcd} [column sep= large]
            K_{g,n+2} \arrow[r, "\pi"] \arrow[d, swap, "d_{n+2}"] & K_{g,n+2}/\langle t_{(n+2)(n+2)} \rangle \arrow[d, "\Delta"]\\
            \mathfrak{e}_{g,n+3} \arrow[r, swap, "s_{n+2} \circ s_{n+3}"] & (K_{g,n+2}/\langle t_{(n+2)(n+2)} \rangle)^{\oplus 2}
        \end{tikzcd},
        \end{equation*}
    where 
        \begin{equation}
            \mathfrak{e}_{g,l}=(H_{g,l}/\bigoplus_{i=l-1,l}\langle t_{ii}\rangle)/\langle [\ker s_{l} \oplus s_{l-1}, \ker s_{l} \oplus s_{l-1}] \rangle, 
        \end{equation}
        and $K_{g,l}$ and $H_{g,l}$ slot into short exact sequences:
        \begin{equation*}
            \begin{tikzcd}[column sep= large, row sep= tiny]
            K_{g, l} \arrow[r] & \mathfrak{t}_{g,l}^f \arrow[r, "s_{l}"] & \mathfrak{t}_{g,l-1}^f,\\
            H_{g, l} \arrow[r] & \mathfrak{t}_{g,l}^f \arrow[r, "s_{l-1}\circ s_{l}"] & \mathfrak{t}_{g,l-2}^f.
            \end{tikzcd}
        \end{equation*}
    Recall, in accordance with Proposition \ref{def_kg,n}, how
    \begin{equation*}
        K_{g,l} \cong \mathfrak{k}_{g,l},
    \end{equation*}
    where the former is the degree completion of the free Lie algebra in $2g+l$ generators $\mathrm{Lie}(x_n^a, y_n^a, t_{il}, \dots t_{ll})_{1 \leq a \leq g}$ subject to the relation
    \begin{equation*}
         \sum_{a=1}^g [x_l^a, y_l^a]=-\sum_{j:j\neq l} t_{lj}-2(g-1)t_{ll}.
    \end{equation*}
    It clearly follows that 
    \begin{equation}\label{k_is_free}
        K_{g,n+2}/\mathbb{K}t_{(n+2)(n+2)} \cong \widehat{\mathrm{Lie}}(x_{n+2}^a, y_{n+2}^a, t_{2(n+2)}, \dots t_{(n+1)(n+2)}) = L^{n+2}.
    \end{equation}
    Likewise, recall how Proposition (\ref{def_hg,n}) sets up the isomorphism
    \begin{equation*}
        H_{g,l} \cong \mathfrak{h}_{g,l},
    \end{equation*}
    where the latter algebra is defined in terms of generators and relations, analogously to $\mathfrak{k}_{g,l}$. Lemma \ref{ker_sis} also established that $\ker s_{n+3}\oplus s_{n+2}=\langle t_{(n-2)(n-3)}\rangle$. Comparing the definition of $\mathfrak{g}$ and $\mathfrak{e}_{g,n+3}$, we can conclude the proof:
    \begin{equation*}
        \mathfrak{g}= (\mathfrak{h}_{g,n+3}/ \bigoplus_{i=n+2,n+3}\langle t_{ii}\rangle)/[\langle t_{(n-3)(n-2)}\rangle, \langle t_{(n-3)(n-2)}\rangle].
    \end{equation*}  
\end{proof}

Given any linear map $r:L^{n+2}\rightarrow \mathbb{K}$, we can define a section $\gamma^r: \overline{K}_{g,n+2} \rightarrow K_{g,n+2}$ by
\begin{equation*}
    \gamma^r(x)=x +r(x)t_{(n+2)(n+3)}
\end{equation*}
for all $x \in L^{n+2} \cong \overline{K}_{g,n+2}$.
\begin{lemma}
    Let $n \geq 0$ be fixed. Then given any linear map $r: L^{n+2} \rightarrow \mathbb{K}$, 
    \begin{equation*}
        ( (\mathrm{PaCD}^f, \mathrm{PaCD}^f_g, t_{11}),\, \gamma^r:\overline{K}_{g,n+2} \rightarrow K_{g,n+2},\, \phi: M \rightarrow UL^{n+2}) \in \mathrm{OpR}_f^\Delta(n+2),
    \end{equation*}
    where $\phi$ is the unique $L\oplus L$-module isomorphism uniquely defined by the condition that
    \begin{equation}\label{phi_def}
            \phi(t_{(n+2)(n+3)})=1.
    \end{equation}
\end{lemma}
\begin{proof}
    We already know that $\overline{K}_{g,n+2}$ is isomorphic to a free Lie algebra. For clarity, we suppress the isomorphism $K_{g,n+2}/\mathbb{K}t_{(n+2)(n+2)} \cong L^{n+2}$. Recall that
    \begin{equation*}
        M \defeq \langle t_{(n+2)(n+3)} \rangle_{\overline{\mathfrak{h}}_{g,n+3}} / [\langle t_{(n+2)(n+3)} \rangle, \langle t_{(n+2)(n+3)} \rangle]_{\overline{\mathfrak{h}}_{g,n+3}}.
    \end{equation*}
    The space $M$ is freely generated by $t_{(n+2)(n+3)}$ as an $L^{n+2} \oplus L^{n+2}$-module (cf. Corollary \ref{M_is_free}), so that condition (\ref{phi_def}) is indeed enough to define an isomorphism $\phi$.
\end{proof}

\begin{proposition}\label{PaCD--GoTU}
    Let $n\geq 0$ be fixed. Then the functor $\Lambda_{n+2}: \mathrm{OpR}^\Delta_f (n+2)\rightarrow \mathbf{GoTu}$ satisfies
        \begin{equation*}
            \begin{tikzcd}
                ( (\mathrm{PaCD}^f, \mathrm{PaCD}^f_g, t_{11}),\, \gamma^r:\overline{K}_{g,n+2} \rightarrow K_{g,n+2},\, \phi: M \rightarrow UL^{n+2}) \arrow[r, "\Lambda_{n+2}", mapsto] & (UL^{n+2}, [-,-]^{\tilde{\rho}_G}, \delta_{q^r})
            \end{tikzcd},
        \end{equation*}
        where the Fox pairing $\tilde{\rho}_G$ and quasi-derivation $q^r$ are defined in equations (\ref{rho_G_tilde}) and (\ref{q_r}), respectively.
\end{proposition}
\begin{proof}
    Recall that the functor $\Lambda_{n+2}$ consists of the composition
    \begin{equation*}
        \begin{tikzcd}[row sep=small]
            \mathrm{OpR}^\Delta_f(n+2) \arrow[r, "\mathrm{D}_{n+2}"]   & \mathbf{RelE}^\eta \arrow[r, "\mathrm{F}^{-1}", "\cong"'] & \mathbf{RelCo}^\eta \arrow[r, "\mathrm{E}^{-1}", "\cong"'] & \mathbf{Fox}^\eta \arrow[r, "\Psi"] & \mathbf{GoTu},
        \end{tikzcd}
    \end{equation*}
    so we can break down the proof along this chain. The argument will only seemingly be complicated by some intervening isomorphisms.
    
    Consider the object of $\mathbf{PRelE}^\Gamma$
    \begin{equation*}
        \mathcal{C}^\Gamma \defeq
        \begin{tikzcd}
            \mathfrak{k}_{g,n+2} \arrow[r, "\pi"] \arrow[d, "d_{n+2}"] & L^{n+2} \arrow[d, "\Delta"] \\
            \mathfrak{g} \arrow[r, swap, "P"] & L^{n+2} \oplus L^{n+2}
        \end{tikzcd}.
    \end{equation*}
    By Lemma \ref{C_in_prele}, we know $\Pi_{n+2}(\mathrm{PaCD}^f, \mathrm{PaCD}^f_g, t_{11})=\mathcal{C}^\Gamma$. Now we notice that $d^r_{n+2}=d_{n+2} \circ \gamma^r$, which implies
    \begin{equation*}
        \mathcal{D}^r= \mathrm{P} \left( 
        \begin{tikzcd}
            \mathfrak{k}_{g,n+2} \arrow[r, "\pi"] \arrow[d, "d_{n+2}"] & L^{n+2} \arrow[d, "\Delta"] \arrow[l, bend right=40, swap, "\gamma^r"]\\
            \mathfrak{g} \arrow[r, swap, "P"] & L^{n+2} \oplus L^{n+2}
        \end{tikzcd} \right).
    \end{equation*}
    This in turn means $\mathrm{D}_{n+2}((\mathcal{P}, \mathcal{M}, t_{11}),\gamma^r, \phi)=(\mathcal{D}^r, \phi)$. According to Proposition \ref{Cgamma_is_C}, $\mathcal{D}^r$ and the commutative diagram
    \begin{equation*}
        \mathcal{C}^r=
    \begin{tikzcd}
        & L^{n+2} \arrow[d, "\Delta"] \arrow[dl, "\Delta \times T^r" swap]\\
         L^{n+2} \oplus L^{n+2} \times_G UL^{n+2} \arrow[r, "P_1" swap] & L^{n+2} \oplus L^{n+2}
    \end{tikzcd}
    \end{equation*}
    are isomorphic in $\mathbf{RelE}$. The object in $\mathbf{RelCo}^\eta$ corresponding to $(\mathcal{C}^r, \phi)$ under the equivalence of categories of Proposition \ref{cat-equiv} is
    \begin{equation*}
        \mathcal{C}^r=(\Delta: L^{n+2} \rightarrow L^{n+2} \oplus L^{n+2}, UL^{n+2},  T^f \oplus G \in Z^2(L^{n+2}\oplus L^{n+2}, L^{n+2}; UL^{n+2})) \in \mathrm{Ob}(\mathbf{RelCo}^\eta),
    \end{equation*}
    which we abusively denote by the same symbol as its diagram counterpart. Now notice that
    \begin{equation*}
        \mathrm{E}(L^{n+2}, q^r \oplus \tilde{\rho}_G)=(\Delta:L^{n+2} \oplus L^{n+2} \rightarrow L^{n+2}, UL^{n+2}, T^r \oplus G).
    \end{equation*}
    This implies the desired result:
    \begin{equation*}
        \Lambda_{n+2}((\mathrm{PaCD}^r, \mathrm{PaCD^f}_g),\gamma^r, \phi)=(UL^{n+2}, [-,-]^{\tilde{\rho}_G}, \delta_{q^{r}}).
    \end{equation*}
\end{proof}

\section{Higher genus Kashiwara-Vergne groups} \label{krv_chap}

The higher genus Kashiwara-Vergne ($\mathrm{KV}$) problems were introduced in \cite{Flor1}. The higher genus graded Kashiwara-Vergne groups $\mathrm{KRV}_g$ act freely and transitively on the set of solutions of the corresponding KV problem (cf. \cite[Theorem 8.4]{Flor1}). The $\mathrm{KRV}_g$ groups are topological in nature. Given a framing 
\begin{equation*}
    \begin{tikzcd}
    f:\, T\Sigma_{g,n+1} \arrow[r, "\cong"] & \Sigma_{g,n+1} \times \mathbb{R}^2
\end{tikzcd}
\end{equation*}
on a compact oriented surface of genus $g$ with $n+1$ boundary components, the group $\mathrm{KRV^f_{g,n+1}}$ can be recovered from the automorphism group of the Goldman-Turaev Lie bialgebra of $\Sigma_{g,n+1}$ (cf. Theorem \ref{special}). This topological description is the only treatment we give the groups $\mathrm{KRV}_g$ in this text; we refer the reader to \cite{Flor1} for more details. Our main goal for this chapter is to prove that $(UL^{n+2}, [-,-]^{\rho_G}, \delta_{q^r})$ is isomorphic to the associated graded of the Goldman-Turaev Lie bialgebra on the surface $\Sigma_{g,n+1}$ (with respect to a framing dependent on the linear map $r: L^{n+2} \rightarrow \mathbb{K}$). This will indirectly yield a morphism between (a subgroup of) $\mathrm{GRT}_g$ and $\mathrm{KRV}_{g,n+1}$.

\subsection{The Goldman-Turaev bialgebra} \label{GT_bialg}

As previously mentioned, $\mathrm{KRV}_{g,n+1}^f$ is intimately tied to the automorphism group of the Goldman-Turaev Lie bialgebra on $\Sigma_{g,n+1}$. In this subsection, we review this in some detail. We also pinpoint the Fox pairing and quasi-derivation that induce the relevant Goldman bracket and Turaev cobracket. We closely follow \cite{Flor1}.

We fix a labeling of the boundary components of $\Sigma_{g,n+1}$, 
\begin{equation*}
    \partial \Sigma_{g,n+1}= \bigsqcup_{j=0}^n \partial_j \Sigma_{g,n+1},
\end{equation*}
and mark a point $*\in \partial_0 \Sigma_{g,n+1}$. We call the collection $\{\alpha_i, \beta_i, \gamma_j\}$ for $i=1, \dots g$, $j=1, \dots ,n$ of elements of $\pi \defeq \pi_1(\Sigma_{g,n+1}, *)$ a \textit{generating system} if each loop $\gamma_j$ is freely homotopic to the $j$th boundary component, and also
\begin{equation*}
    \gamma_0=\prod_{i=1}^g \alpha_i \beta_i \alpha_i^{-1} \beta_i^{-1}\prod_{j=1}^n\gamma_j.
\end{equation*}

The Goldman bracket $[-,-]$ is a Lie bracket defined on the space of cyclic words on the group algebra $\mathbb{K}\pi$. It is induced by a double bracket
\begin{equation*}
    \kappa: \mathbb{K}\pi \otimes \mathbb{K}\pi \rightarrow \mathbb{K}\pi \otimes \mathbb{K}\pi
\end{equation*}
in the sense of Proposition \ref{brac_from_double}. The Turaev cobracket $\delta^f$ is a Lie cobracket on $|\mathbb{K}\pi|$. Its definition depends on a choice of framing $f$ on $\Sigma_{g,n+1}$; that is, it depends on a choice of trivialization of the tangent bundle of $\Sigma_{g,n+1}$,
\begin{equation*}
    \begin{tikzcd}
        f:\, T\Sigma_{g,n+1} \arrow[r, maps to, "\cong"] & \Sigma_{g,n+1} \times \mathbb{R}^2
    \end{tikzcd}.
\end{equation*}
The Goldman bracket and Turaev cobracket endow the space $(|\mathbb{K}\pi|, [-,-], \delta^f$) with the structure of a Lie bialgebra.

It is preferable to deal with the double bracket $\kappa$ as opposed to $[-,-]$, since the former satisfies multiplicative relations on $\mathbb{K}\pi$. Likewise, we will mostly handle the operations
\begin{equation*}
    \mu_r^f: \mathbb{K}\pi \rightarrow |\mathbb{K}\pi| \otimes \mathbb{K}\pi, \hspace{1em} \mu_l^f:\mathbb{K}\pi \rightarrow \mathbb{K}\pi \otimes |\mathbb{K}\pi|,
\end{equation*}
which are specified entirely by their values on the generating system, but which nevertheless characterize $\delta^f$ by
\begin{equation*}
    \delta^f=(\mathrm{id}\otimes|\cdot|)\mu_r^f+(|\cdot| \otimes \mathrm{id})\mu_l^f.
\end{equation*}

Given a generating system $\alpha_i,\beta_i$ for $i=1,\dots g$, and $\gamma_j$ for $j=1, \dots n$, we denote the corresponding homology classes in $H \defeq H_1(\pi, \mathbb{K})$ by
\begin{equation*}
    [\alpha_i]=x_i, \hspace{1em} [\beta_i]=y_i, \hspace{1em} [\gamma_j]=z_j.
\end{equation*}

\begin{definition}[\cite{Flor1}, Definition 3.6]
\label{wght_pi} 
We define a weight filtration on $\mathbb{K}\pi$ by the assignment
\begin{equation*}
    \mathrm{wt}(\alpha_i-1)=\mathrm{wt}(\beta_i-1)=1, \hspace{ 1em} \mathrm{wt}(\gamma_j-1)=2,
\end{equation*}
for $i=1,\dots, g$ and $j=1,\dots n$.    
\end{definition}
These choices also induce a weight filtration on $H$, where now
\begin{equation*}
    \mathrm{deg}(x_i)=\mathrm{deg}(y_i)=1, \hspace{1em} \mathrm{deg}(z_j)=2.
\end{equation*}
Despite the way it is defined, the weight filtration on $\mathbb{K}\pi$ is independent of the choice of generating system of $\pi$ (cf. \cite[Proposition 3.9]{Flor1}).

The tensor algebra $T(\mathrm{gr^{wt}}H)$ is a graded Hopf algebra over $\mathbb{K}$ with respect to the operations specified by
\begin{equation*}
    \Delta(a)=a \otimes 1 + 1 \otimes a, \hspace{1em} \varepsilon(a)=0, \hspace{1em} S(a)=-a
\end{equation*}
for all $a \in H$. Let $L=L_{2g +n}$ denote the degree completion of the free Lie algebra with generators $x_i$, $y_i$ for $i=1, \dots g$, and $z_j$ for $j=1, \dots n$. The completion
\begin{equation*}
    \hat{T}(\mathrm{gr^wt} H)\defeq \prod_{m \geq 0}(\mathrm{gr^wt}H)^{\otimes m}
\end{equation*}
is naturally isomorphic to the the algebra of power series $A \defeq \mathbb{K}\langle \langle x_i, y_i, z_j \rangle \rangle \cong UL.$

\begin{proposition}[\cite{Flor1}, Proposition 3.5] \label{A_is_grad}
    There is a unique isomorphism of complete graded Hopf algebras
    \begin{equation*}
        \mathrm{gr^{wt}}(\mathbb{K}\pi) \cong\ \hat{T}(\mathrm{gr^{wt}} H) \cong A,
    \end{equation*}
    which maps $\mathrm{gr}(\gamma_j)$ to $z_j$.
\end{proposition}

By exploiting this isomorphism, we may view the associated graded of the Goldman bracket and Turaev cobracket as maps
\begin{equation*}
    [-,-]_{\mathrm{gr}}: |A| \otimes |A| \rightarrow |A|, \hspace{1em} \delta_{\mathrm{gr}}^f:|A| \rightarrow |A| \otimes|A|.
\end{equation*}
As in the non-graded case, these maps are uniquely characterized by $\kappa_{\mathrm{gr}}$ and $\mu_{r,\mathrm{gr}}^f$, $\mu_{l,\mathrm{gr}}^f$.
\begin{lemma}[\cite{Flor1}, Lemma 3.14]
    The non-vanishing generator entries of the graded double bracket are
        \begin{equation*}
            \kappa_{\mathrm{gr}}(x_i,y_i)=-\kappa_{\mathrm{gr}}(y_i, x_i)=1 \otimes 1, \hspace{1em} \kappa_{\mathrm{gr}}(z_j, z_j)=z_j \otimes 1 -1 \otimes z_j.
        \end{equation*}
\end{lemma}
Given a framing $f$ of $\Sigma$ we define a linear map $r^f: L \rightarrow \mathbb{K}$ by setting
\begin{equation*}
    r^f(x_i)=r^f(y_i)=0, \hspace{1em} r^f(z_j)=\mathrm{rot}^f(\gamma_j)+1
\end{equation*}
for the generators of $L$. Following \cite{Flor1}, we define the \textit{adapted framing} $f^{\mathrm{adp}}$ of $\Sigma$ as the framing with respect to which 
\begin{equation*}
    \mathrm{rot}^{f^{\mathrm{adp}}}(\alpha_i)=\mathrm{rot}^{f^{\mathrm{adp}}}(\beta_i)=0, \quad \mathrm{rot}^{f^{\mathrm{adp}}}(\gamma_j)=-1,
\end{equation*}
for all $i=1, \dots, a$ and $j=1, \dots n$. Notice that $r^{f^{\mathrm{adp}}}$ is identically zero.
\begin{lemma}[\cite{Flor1}, Proposition 3.22] \label{lem_mu_gr}
    Let $a=a_i \cdots a_m \in A$, with each $a_i \in L$. Then
    \begin{equation*}
        \begin{split}
            \mu_{r,\mathrm{gr}}^f(a)=& (|\cdot| \otimes \mathrm{id})\Biggl(\sum_{i=1}^m r^f(a_i)1 \otimes a_1 \cdots a_{i-1}a_{i+1} \cdots a_m\\
            &+\sum_{k=1}^m \kappa_{\mathrm{gr}}(a_1 \cdots a_{k-1}, a_k)(1 \otimes a_{k+1} \cdots a_m) \Biggr),
        \end{split}
    \end{equation*}
    and $\mu_{l,\mathrm{gr}}^f$ is specified by the condition
    \begin{equation}\label{mu_skew}
        \mu_{l,\mathrm{gr}}^f = - P(\mu_{l, \mathrm{gr}}^{f}),
    \end{equation}
    where $P(a \otimes b)=b\otimes a$ for all $a,b \in A$.
\end{lemma}

\begin{definition}\label{special}
    We denote by $\mathrm{tAut^+(L)}$ the set of \textit{tangential Lie automorphisms} of L with positive grading. That is, for each $F \in \mathrm{tAut^+(L)}$, there exist $f_j\in \exp (L)$, $j=1, \dots n$, such that
    \begin{equation}\label{tangential}
        F(z_j)= f_j^{-1}z_jf_j\,\, \text{for all}\,\,j.
    \end{equation}
\end{definition}

\begin{theorem}[\cite{Flor1}, Theorem 8.21] \label{AutGT=KRV}
    Let $\mathrm{tAut}^+(L, \kappa_{\mathrm{gr}}, \delta_{\mathrm{gr}}^f)$ denote the subgroup of $\mathrm{tAut}^+(L)$ which preserves the operations $\kappa_{\mathrm{gr}}$ and $\delta_{\mathrm{gr}}^f$. There exists a natural isomorphism
    \begin{equation*}
        \mathrm{tAut}^+(L, \kappa_{\mathrm{gr}}, \delta_{\mathrm{gr}}^f) \cong \mathrm{KRV}^{f}_{(g,n+1)}.
    \end{equation*}
\end{theorem}

\begin{proposition}\label{rhoG_gives_brac}
    Let $\rho_G:A \times A \rightarrow A$ be the Fox pairing with non-vanishing generator entries given by:
    \begin{equation*}
        \rho_G(x_i,y_i)= 1, \hspace{1 em} \rho_G(y_i, x_i)=-1, \hspace{1em} \rho_G(z_j, z_j)=-z_j.
    \end{equation*}
    Then $\dbl -, - \dbr^{\rho_G}$ is the associated graded of the Goldman double bracket on the surface $\Sigma_{g, n+1}$, $\kappa_{\mathrm{gr}}$. 
\end{proposition}
\begin{proof}
    We are implicitly regarding $\{x_i, y_i, z_j\}$ as standard generators for $\hat{T}(\mathrm{gr^{wt}}\, H)$. By direct computation, we can verify that the non-zero entries of $\dbl - , - \dbr^{\rho_G}$ match those of $\kappa_\mathrm{gr}$:
    \begin{equation*}
        \dbl x_i, y_i \dbr^{\rho_G}=1 \otimes 1, \hspace{1em} \dbl y_i, x_i \dbr^{\rho_G}=-1 \otimes 1, \hspace{1em} \dbl z_j, z_j \dbr^{\rho_G}=z_j \otimes1-1\otimes z_j.
    \end{equation*}
\end{proof}
\begin{corollary} \label{rho_G=gold}
    The bracket $[-,-]^{\rho_G}: |A| \otimes|A| \rightarrow |A|$, induced from the Fox pairing $\rho_G$ in the sense of Proposition \ref{brac_from_fox}, is the graded Goldman bracket on the surface $\Sigma_{g,n+1}$. In symbols,
    \begin{equation*}
        [-,-]^{\rho_{G}}=[-,-]_{\mathrm{gr}}.
    \end{equation*}
\end{corollary}

The Fox pairing $\rho_G$ is skew-symmetric in the sense of Definition \ref{fox-skew-def}. It follows that any quasi-derivation $q \in \mathrm{Qder}(\rho_G)$ that satisfies $q(a_k)=-q^t(a_k)$ for all generators $a_k$ of $A$ will be skew-symmetric in the sense of Definition \ref{qder-skew-def}.
\begin{proposition}
    Let $q^f:A \rightarrow A$ be the quasi-derivation associated to the Fox pairing $-\rho_G$, and whose values on the generators of $A$ are specified by setting $q^f(a)=r^f(a)$ for all $a \in L$. Then
    \begin{equation*}
        (|\cdot| \times \mathrm{id})d_{q^f}= \mu_{r, \mathrm{gr}}^{f} \hspace{1em} \text{and} \hspace{1em} (\mathrm{id} \times |\cdot|)Pd_{\overline{q}^f}=\mu_{l,\mathrm{gr}}^f.
    \end{equation*}
\end{proposition}
\begin{proof}
    Given any quasi derivation $q$ in $A$, recall we defined $d_q: A \rightarrow A \otimes A$ by
    \begin{equation*}
        d_q(a)=a'S(q(a'')')\otimes q(a'')''.
    \end{equation*}
    We specify two additional quasi-derivations of $A$:
    \begin{gather*}
        q_1\in \mathrm{Qder(0)}, \, q_1(a)=r^f(a)\,\, \text{for all }a\in L;\\
        q_2\in \mathrm{Qder}(-\rho_g), \, q_2(a)=0\,\, \text{for all }a\in L.
    \end{gather*}
    We do this with the intention of splitting up $d_{q^f}$ into the sum $d_{q^f}=d_{q_1}+d_{q_2}$, to deal with each term separately.
    
    Let $a=a_1 \cdots a_m\in A$ be a product of standard generators. Notice we can write
    \begin{equation*}
        d_{q_2}(a)=\sum_{k=1}^{m} a_1' \cdots a_m'S(\rho_G(a_1'' \cdots a_{k-1}'', a_k'')'a_{k+1}'' \cdots a_m'')\otimes \rho_{G}(a_1'' \cdots a_{k-1}'', a_k'')''a_{k+1}''' \cdots a_m''',
    \end{equation*}
    where we have only used the fact that $\rho_G$ is a right Fox derivative with respect to the second variable. It follows that
    \begin{equation*} 
        \begin{split}
            d_{q_2}(a)&=\sum_{i=1}^m a_1' \dots a_{k-1}''\varepsilon(a_{k+1}' \cdots a_m') S(\rho_G(a_1'' \cdots a_{k-1}'', a_k'')')\otimes \rho_G(a_1'' \cdots a_{k-1}'', a_k'')''a_{k+1}'''\cdots a_m'''\\
            &=\sum_{i=1}^m a_1' \dots a_{k-1}''\ S(\rho_G(a_1'' \cdots a_{k-1}'', a_k'')')\otimes \rho_G(a_1'' \cdots a_{k-1}'', a_k'')''a_{k+1}\cdots a_m.\\
        \end{split}
    \end{equation*}
    On the other hand, 
    \begin{equation*}
        \begin{split}
            d_{q_1}(a)&=\sum_{i=1}^m a_1' \cdots a_m'S(r^f(a_i'')a_1'' \cdots a_{i-1}''a_{i+1}''\cdots a_m'') \otimes a_1'' \cdots a_{i-1}'' a_{i+1}'' \cdots a_m''\\
            &=\sum_{i=1}^m r^f(a_i) \otimes a_1 \cdots a_{i-1}a_{i+1} \cdots a_m.
        \end{split}
    \end{equation*}
    In accordance with Lemma \ref{lem_mu_gr}, 
    \begin{equation*}
        \begin{split}
            \mu_{r, \mathrm{gr}}^{f}(a)
            &=\sum_{i=1}^m(|\cdot| \otimes \mathrm{id}) r^f(a_i) \otimes a_1 \cdots a_{i-1}a_{i+1} \cdots a_m\\&+\sum_{k=1}^m (|\cdot| \otimes \mathrm{id})a_k' S(\rho_G(a_1'' \cdots a_{k-1}'',a_k'')')a_1' \cdots a_{k-1}' \otimes \rho_G(a_1' \cdots a_{k-1}'', a_k'')a_{k+1} \cdots a_n'',
        \end{split}
    \end{equation*}
    where we have used the fact that $\kappa_\mathrm{gr}= \dbl -,- \dbr^{\rho_G}$. Comparing each summand, we can see how the first and second terms match the contributions of $d_{q_1}$ and $d_{q_2}$, respectively. It follows that $(|\cdot| \times \mathrm{id})d_{q^f}= \mu_{r, \mathrm{gr}}^{f}$. To prove the other half of our claim, recall how
    \begin{equation}
        \mu_{l,\mathrm{gr}}^f = - P(\mu_{l, \mathrm{gr}}^{f}),
    \end{equation}
    where $P(a \otimes b)=b\otimes a$ for all $a,b \in A$. It suffices to show that $d_{q^f}$ and $Pd_{(q^f)^t}$ satisfy an analogous relation to (\ref{mu_skew}), and this follows directly from the skew-symmetry of $q^f$.
\end{proof}
\begin{corollary}\label{qf=Tur}
    The cobracket $\delta_{q^f}: |A| \rightarrow |A| \otimes |A|$ induced by the quasi derivation $q^f$ in the sense of Proposition \ref{cobracyc}, coincides with the graded Turaev cobracket of the surface $\Sigma_{g,n+1}$ associated to the framing $f$. In symbols,
    \begin{equation*}
        \delta_{q^f}=\delta_{\mathrm{gr}}^f.
    \end{equation*}
\end{corollary}

\subsection{$\mathbf{GRT}$ and $\mathbf{KRV}$}

In this section, we finally recover the Goldman-Turaev Lie bialgebras on higher genus surfaces from the $\mathbf{PaCD}^f$-module $\mathbf{PaCD}^f_g$. As a corollary, we obtain a group map between the corresponding automorphism groups of these objects, the higher genus analogues of $\mathrm{GRT}$ and $\mathrm{KRV}$.

Let $f:T\Sigma_{g,n+1} \rightarrow \Sigma_{g,n+1} \times \mathbb{R}^2$ be a framing on $\Sigma_{g,n+1}$. In all that follows, we will mainly be concerned with a slight extension of the group $\mathrm{KRV}_{(g, n+1)}^f$.
\begin{definition}\label{KRV_def}
    We set
    \begin{equation*}
        \overline{\mathrm{KRV}}_{(g,n+1)}^f \defeq \mathrm{Aut}^+(L, [-,-]_{\mathrm{gr}}, \delta_{\mathrm{gr}}^f).
    \end{equation*}
    Notice that Lie morphisms in $\overline{\mathrm{KRV}}_{(g,n+1)}^f$ only need to preserve the bracket $[-,-]_{\mathrm{gr}}$, a slightly weaker condition than fixing $\kappa_{\mathrm{gr}}$.
\end{definition}
The group $\overline{\mathrm{KRV}}^f_{g,n+1}$ is a semi-direct product of the original $\mathrm{KRV}^f_{g,n+1}$. To see this, we take a brief foray into \textit{special} automorphisms.

\begin{definition}
    Recall that a tangential automorphism is characterized by the equation (\ref{tangential}). A tangential automorphism $f$ is \textit{special} if also $f(\omega)=\omega$, where 
    \begin{equation*}
        \omega \defeq \sum_{i=1}^g[x_i, y_i]+\sum_{j=1}^nz_j.
    \end{equation*}
\end{definition}

In \cite[Theorem 6.11]{Flor1}, the authors prove that any special automorphism $g$ defines a commutative diagram
\begin{equation}\label{spec_rho_phi}
    \begin{tikzcd}
        A \otimes A \arrow[r, "\rho_G + \rho_\phi"] \arrow[d, "g \otimes g" swap] & A \arrow[d, "g"]\\
        A \otimes A \arrow[r, "\rho_g + \rho_\phi" swap] & A
    \end{tikzcd},
\end{equation}
where $\rho_\phi$ is the inner Fox pairing in $A$ in given by $\rho_\phi(a,b)=D(a)\phi D(b)$, and
\begin{equation*}
    \phi= \phi(\omega) \defeq \frac{1}{e^{\omega}-1}-\frac{1}{\omega}.
\end{equation*}
The element $\phi \in A$ is also expressible in terms of Bernoulli numbers as
\begin{equation*}
    \phi=\sum_{k=1}^{\infty} \frac{B_k}{K}\omega^{k-1}.
\end{equation*}
By virtue of the condition $g(\omega)=\omega$, we immediately deduce that
\begin{equation*}
    g^{-1}\rho_\phi(g \otimes g)=\rho_\phi
\end{equation*}
for any special automorphism $g$. This lets us rephrase diagram (\ref{spec_rho_phi}) into a statement more useful for our purposes:

\begin{proposition}[\cite{Flor1, MasTu}] \label{spec_fix_rho}
Any special automorphism $g$ preserves the Fox pairing $\rho_G$, in the sense that the following diagram is commutative:
\begin{equation} \label{spec_rho_G}
    \begin{tikzcd}
        A \otimes A \arrow[r, "\rho_G "] \arrow[d, "g \otimes g" swap] & A \arrow[d, "g"]\\
        A \otimes A \arrow[r, "\rho_G"] & A
    \end{tikzcd}.
\end{equation}
\end{proposition}

Per \cite[Theorem 2.7]{Flor3}, any $g\in \mathrm{Aut}(L, [-,-]_\mathrm{gr})$ is conjugate to a special automorphism. That is, there exist a group-like $x\in A$ and some special automorphism $g_0$ such that
\begin{equation*}
    g(a)=x^{-1}g_0(a)x\quad \text{for all }a\in A.
\end{equation*}

\begin{proposition}\label{acc_ext}
    Let $\langle \exp{\omega} \rangle_L$ denote the normal subgroup generated by $\exp{\omega}$ in $\exp(L)$, the set of group-like elements in $A$. Then
    \begin{equation*}
        \overline{\mathrm{KRV}}_{(g,n+1)}^f = \left( \exp(L)/\mathbb{K}\langle \exp{\omega}\rangle_L \right) \rtimes \mathrm{KRV}_{(g,n+1)}^f.
    \end{equation*}
\end{proposition}
\begin{proof}
    We will check the claim directly. We write a map
    \begin{equation*}
        \begin{split}
            \xi: \left( \exp(L)/\mathbb{K}\langle \exp{\omega}\rangle_L \right) \rtimes \mathrm{KRV}_{(g,n+1)}^f & \longrightarrow \overline{\mathrm{KRV}}_{(g,n+1)}^f\\
            (x, g) & \longmapsto \mathrm{Ad}(x) \circ g,
        \end{split}
    \end{equation*}
    where $\mathrm{Ad}(x): L \rightarrow L$ denotes conjugation by $x$. Preserving the double bracket $\kappa_{\mathrm{gr}}$ implies the preservation of the bracket $[-,-]_{\mathrm{gr}}$. Note also that conjugation maps induce the identity map on $|A|$. This means that $\xi$ is well defined.

    Let $g \in \overline{\mathrm{KRV}}_{(g,n+1)}^f$. To prove surjectivity, recall that $g$ is expressible as $g = \mathrm{Ad}(\exp{x}) \circ g_0$, where $g_0$ is a special automorphism and $x\in L$. Since all special automorphisms fix the Fox pairing $\rho_G$ exactly, it follows that $g_0 \in \mathrm{KRV}_{(g,n+1)}^f$, and that $\xi(\mathrm{Ad}(\exp x),g_0)=g$.

    Now let us suppose that $\mathrm{Ad}((\exp{x})^{-1}) \circ g=\mathrm{id}$. Without loss of generality, we may assume that $g$ is a special automorphism. Our assumption then translates into the statement that the conjugation $\mathrm{Ad}(\exp x)$ is special. In particular,
    \begin{equation*}
        (\exp{x})^{-1}\omega\exp{x}=\omega,
    \end{equation*}
    implying that $\exp{x} \in \mathbb{K}\langle \! \langle \omega \rangle \! \rangle$, so necessarily $\exp{x}\in \mathbb{K}\langle \exp{\omega}\rangle_L$
\end{proof}

We are all but ready to state the main theorem in this paper. The last prerequisite we need to address is the dependence of the Goldman-Turaev bialgebra on a choice of framing. 
\begin{lemma}\label{sec_from_frames}
    There is a one-to-one correspondence between sections $\gamma: \overline{K}_{g,n+2} \rightarrow K_{g,n+2}$ of the projection $ \pi: K_{g,n+2} \rightarrow \overline{K}_{g,n+2}$ and framings of the surface $\Sigma_{g,n+1}$.
\end{lemma}
\begin{proof}
    Let $f: T\Sigma_{g,n+1} \rightarrow \Sigma_{g,n+1} \times \mathbb{R}^2$ be a framing on $\Sigma_{g,n+1}$. Recall the definition of the $\mathbb{K}$-linear map $r^f: \mathrm{gr}^\mathrm{wt}H \rightarrow \mathbb{K}$ associated to $f$; it is induced by the choices
    \begin{equation*}
        r^f(x_i)=r^f(y_1)=0, \hspace{1em} r^f(z_j)=\mathrm{rot}^f(\gamma_j)+1.
    \end{equation*}
    Now, to any framing $f$, we may associate the section $\gamma^f: \overline{K}_{g,n+2} \rightarrow K_{g,n+2}$,
    \begin{equation*}
        \gamma^f(a)=a + r^f(a)t_{(n+2)(n+2)},
    \end{equation*}
    where we are implicitly identifying the Lie algebras $L^{n+2} \cong \overline{K}_{g,n+2}$ with $L$ via the isomorphism
    \begin{equation}\label{exp_iso}
    \begin{tikzcd}[row sep=tiny]
        L \arrow[r, "\cong"] & L^{n+2}\\
        x_i \arrow[r, maps to] & x_{n+2}^i\\
        y_i \arrow[r, maps to] & y^i_{n+2}\\
        z_j \arrow[r, maps to] & t_{(j+1)(n+2)}
    \end{tikzcd},
\end{equation}
where $1 \leq i \leq g$, $1 \leq j \leq n$.
\end{proof}

\begin{theorem}\label{PaCD_is_GT}
    Let $n \geq 0$ be fixed. Given any framing $f$ on $\Sigma_{g,n+1}$, let us abbreviate
    \begin{equation*}
        \mathcal{N}^f \defeq( (\mathrm{PaCD}^f, \mathrm{PaCD}^f_g, t_{11}),\, \gamma^f:\overline{K}_{g,n+2} \rightarrow K_{g,n+2},\, \phi: M \rightarrow UL^{n+2}) \in \mathrm{OpR}^\Delta_f(n+2).
    \end{equation*}
    Then $\Lambda_{n+2}(\mathcal{N}^f)$ is isomorphic to $(A, [-,-]_{\mathrm{gr}}, \delta_{\mathrm{gr}}^f)$, where the latter is the associated graded of the Goldman-Turaev bialgebra in $\Sigma_{g,n+1}$ with respect to $f$.
\end{theorem}
\begin{proof}
    Per Proposition \ref{PaCD--GoTU}, we know that $\Lambda_{n+2}(\mathcal{N}^f)= (UL^{n+2}, [-,-]^{\tilde{\rho}_G}, \delta_{q^{r^{f}}})$, where $\tilde{\rho}_G$ and $q^{r^f}$ are the Fox pairing and quasi-derivation defined in equations (\ref{qf_1}) and (\ref{rho_G_1}), respectively. But notice that $\tilde{\rho}_G \in \mathrm{Fox}(UL^{n+1})$ and $\rho_G \in \mathrm{Fox}(UL)$ correspond to one another under the isomorphism of equation (\ref{exp_iso}), and similarly for $q^{r^f}\in \mathrm{Qder}(UL^{n+2})$ and $q^f\in \mathrm{Qder}(UL)$. We finish this proof by referencing Corollaries \ref{rho_G=gold} and \ref{qf=Tur}, which state, respectively, that
\begin{equation*}
    [-,-]^{\rho_G}=[-,-]_{\mathrm{gr}} \hspace{1em} \text{and} \hspace{1em} \delta_{q^f}=\delta^f_{\mathrm{gr}}.
\end{equation*}
\end{proof}
The reader could reasonably be annoyed at all the extra data we had to tack onto operad modules to derive brackets and cobrackets out of them. To an extent, Theorem \ref{PaCD_is_GT} justifies these additions. The marking of a section $\gamma^f$ is necessary, as it corresponds to the dependence of the Turaev cobracket on a framing $f$. Although not topologically motivated, the choice of an isomorphism $\phi: M \rightarrow UL$ is very limited. The image of the generator $t_{(n+2)(n+3)}$ can only possibly be a scalar. The choice amounts to rescaling the (co)bracket operation by the corresponding number.

\begin{lemma}\label{phi_is_fix}
    Let $(f,g) \in \mathrm{GRT}^g_1$. For $n \geq 0$, the morphisms induced by $g$, $g_{n+3}:H_{g,n+3} \rightarrow H_{g,n+3}$, satisfy that $g_{n+3}(t_{(n+2)(n+3)})=t_{(n+2)(n+3)}$.
\end{lemma}
\begin{proof}
    We abbreviate
    \begin{equation*}
    \sigma \defeq 1 \cdot
        \begin{tikzcd}[column sep=tiny, line width=0.6 pt]
            (\cdots(1 \hspace{7em} (2 \hspace{4em} (3 & \cdots & (n+2 \hspace{0.2 em} n+3) \cdots)\\
            & & \\
            (\cdots(1 \hspace{7em} (2 \hspace{4em} (3  \arrow[uu, xshift=-5.3 em] \arrow[uu, xshift=2.55 em] \arrow[uu, xshift= 7.45 em] &  \cdots & (n+2 \hspace{0.2 em} n+3) \cdots) \arrow[uu, xshift=-2.1 em] \arrow[uu, xshift=0.4 em]
        \end{tikzcd}
        \in \mathrm{End}_{\mathbf{PaCD}^f}(d^{n+2}(1)).
    \end{equation*}
    Notice that we can write $t_{(n+2)(n+3)} \cdot \sigma = d^{n+1}H$. But equality (\ref{H_from_T}) implies that $g$ fixes the morphism $H$, so that
    \begin{equation*}
        g(t_{(n+2)(n+3)}\cdot \sigma)= g_{n+3}(t_{(n+2)(n+3)}) \cdot \sigma=t_{(n+2)(n+3)} \cdot \sigma.
    \end{equation*}
\end{proof}

\begin{theorem}\label{GRT_in_KRV}
    Given any framing $f$ of $\Sigma_{g,n+1}$, let $\mathrm{GRT}_1^{g,f}$ denote the subgroup of $\mathrm{GRT}^g_1$ that preserves the section $\gamma^{f}$, in the sense that
    \begin{equation*}
        G \circ \gamma^{f}= \gamma^{f} \circ G \hspace{1em} \text{for all}\,\ (F,G)\in \mathrm{GRT}_1^{g,{f}}.
    \end{equation*}
    The map
    \begin{equation}\label{main_mor}
        \begin{split}
            \mathrm{GRT}_1^{g,f} &\longrightarrow \overline{\mathrm{KRV}}_{(g,n+1)}^f\\
            (F,G) & \longmapsto \{G_{n+2}: L \rightarrow L \}
        \end{split}
    \end{equation}
    is a group morphism for all $n \geq 0$.
\end{theorem}
\begin{proof}
    We just need to note that $\mathrm{GRT}^{g,f}_1 \subset \mathrm{Aut}_{\mathrm{OpR}^\Delta_f(n)}(\mathcal{N}^f)$.  Given $(F,G) \in \mathrm{GRT}^{g,f}_1$, $G$ will preserve the section $\gamma^f$ by definition. Lastly, by Lemma \ref{phi_is_fix}, $G$ defines a commutative diagram
    \begin{equation*}
        \begin{tikzcd}
            M \arrow[r, "\phi"] \arrow[d, "G"] & UL^{n+2} \arrow[d, "G"]\\
            M \arrow[r, "\phi"] & UL^{n+2}
        \end{tikzcd},
    \end{equation*}
    so that indeed $(F,G) \in \mathrm{Aut}_{\mathrm{OpR}^\Delta_f(n)}(\mathcal{N}^f)$.
\end{proof}

We have succeeded in our goal of extracting the Goldman-Turaev Lie bialgebra in higher genus surfaces from the $\mathbf{PaCD}^f$-module $\mathbf{PaCD}^f_g$. To an extent, this implies an analogous result for genu zero surfaces, but this involves interpreting $\mathbf{PaCD}^f$ as the limiting case $\mathbf{PaCD}^f_{g=0}$, and is somewhat nebulous. For clarity, we devote Appendix \ref{app_KRV} to the genus zero case.

Lastly, we have neglected the category $\mathbf{Fox}$ in our statement of Theorem \ref{PaCD_is_GT}, even though the morphism (\ref{main_mor}) factors through it. We can correct this with the following theorem:
\begin{theorem}
    Given a framing $f$ of $\Sigma_{g,n+1}$, let us denote $\mathcal{C}^f_{g,n+1} \defeq (L, q^{f} \oplus \rho_G)\in \mathrm{Ob}(\mathbf{Fox})$. Then
    \begin{equation*}
        \mathrm{Aut}_{\mathrm{Fox}}(\mathcal{C}^f_{g,n+1})/ \mathbb{K} \cong \overline{\mathrm{KRV}}^f_{g,n+1}, 
    \end{equation*}
    where $\mathbb{K}$ stands for scalar multiplication.
\end{theorem}
The proof of this theorem is slightly technical and involves a closer look at special automorphisms, so we address it in Appendix \ref{aut_in_relco}.

\section*{Closing remarks}

In \cite{AlekToro}, Alekseev and Torossian give a proof of the Kashiwara-Vergne conjecture by producing a solution out of a Drinfeld associator. Their proof describes an inclusion of bitorsors
\begin{equation*}
    \begin{tikzcd}[cells={nodes={}}]
        \arrow[loop left, distance=3em, start anchor={[yshift=-1ex]west}, end anchor={[yshift=1ex]west}]{}{\mathrm{GT}}  \mathrm{Assoc} \arrow[loop right, distance=3em, start anchor={[yshift=1ex]east}, end anchor={[yshift=-1ex]east}]{}{\mathrm{GRT}} \arrow[d] \\ 
        \arrow[loop left, distance=3em, start anchor={[yshift=-1ex]west}, end anchor={[yshift=1ex]west}]{}{\mathrm{KV}} \mathrm{Sol(KV)} \arrow[loop right, distance=3em, start anchor={[yshift=1ex]east}, end anchor={[yshift=-1ex]east}]{}{\mathrm{KRV}} 
    \end{tikzcd}
\end{equation*}
from the set of Drinfeld associators into the set of solutions to the KV conjecture. The inclusion of groups
\begin{equation*}
    \begin{tikzcd}
        \mathrm{GT} \arrow[r, hook] & \mathrm{KV}
    \end{tikzcd}
\end{equation*}
can be regarded as a topological counterpart to that of the groups $\mathrm{GRT}$ and $\mathrm{KRV}$. On this final section, we will outline how the machinery we have set up could be applied to the topological setting as well.

The group $\mathrm{KV}$ acts freely and transitively on $\mathrm{Sol(KV)}$ from the left. Analogously to $\mathrm{KRV}$, its higher genus generalizations admit a description as automorphism groups. Recall the notation $\pi \defeq \pi_1(\Sigma_{g,n+1},*)$ for the fundamental group of a surface of genus zero $g$ with $n+1$ boundary components. Let $\widehat{\mathbb{K}\pi}$ denote the completion of the group algebra $\mathbb{K}\pi$ with respect to the weight filtration of Definition \ref{wght_pi}.

\begin{theorem}[\cite{Flor1}, Theorem 8.21] \label{KV_def}
    Given any framing $f$ on $\Sigma_{g,n+1}$, let $\mathrm{tAut^+}(\widehat{\mathbb{K}\pi}, \kappa, \delta^f)$ denote the set of tangential Hopf algebra automorphisms of $\widehat{\mathbb{K}\pi}$ that preserve the Turaev cobracket
    \begin{equation*}
        \delta^f:|\widehat{\mathbb{K}\pi}| \rightarrow |\widehat{\mathbb{K}\pi}| \otimes|\widehat{\mathbb{K}\pi}|
    \end{equation*}
    and the double bracket
    \begin{equation*}
        \kappa:\widehat{\mathbb{K}\pi} \otimes\widehat{\mathbb{K}\pi} \rightarrow\widehat{\mathbb{K}\pi}\otimes \widehat{\mathbb{K}\pi}
    \end{equation*}
    inducing the Goldman bracket. There exists a natural isomorphism
    \begin{equation*}
        \mathrm{tAut^+}(\widehat{\mathbb{K}\pi}, \kappa, \delta^f) \cong \mathrm{KV}^f_{g,n+1}.
    \end{equation*}
\end{theorem}
Just like before, we will treat Theorem \ref{KV_def} as a working definition of the group $\mathrm{KV}^f_{g,n+1}$. On the other hand, in \cite{gonz}, Gonzalez defined higher genus analogues of the group $\mathrm{KV}$; we will use the same notation as them.

\begin{definition}
    Let $\Sigma$ be an oriented surface. Denote by $\pi_\Sigma:F\Sigma \rightarrow \Sigma$ the $\mathrm{GL}(2)$-principal bundle of frames on $\Sigma$. The \textit{framed configuration space of $n$ points on $\Sigma$} is the space
    \begin{equation*}
        \mathrm{Conf}^f(\Sigma, n) \defeq \{(x, f_1, f_1) \in \Sigma\times\mathrm{SO}(2)^{\times2}\, |\, \pi_\Sigma(f_i)=x \}. 
    \end{equation*}
     When $\Sigma=\mathbb{C}$, we define the \textit{reduced framed configuration space of $n$ points on $\mathbb{C}$} by 
    \begin{equation*}
        C^f(\mathbb{C}, n) \defeq \mathrm{Conf}^f(\mathbb{C}, n)/(\mathbb{C} \ltimes \mathbb{R}_{>0}).
    \end{equation*}
\end{definition}

\begin{definition}[\cite{gonz}, \S 2.1]
The \textit{operad of parenthesized framed braids} $\mathbf{PaB}^f$ is the operad in groupoids having $\mathbf{Pa}$ as an operad of objects
\begin{equation*}
    \mathbf{PaB}^f(n)\defeq \pi_1(\overline{C}^f(\mathbb{C},n),\mathbf{Pa}),
\end{equation*}
where $\overline{C}^f(\mathbb{C},n)$ denotes the Axelrod-Singer-Fulton-MacPherson (ASFM) compactification of $C^f(\mathbb{C},n)$. 
\end{definition}

\begin{definition}[\cite{gonz}, \S 3.1]
    The \textit{$\mathbf{PaB}^f$-operad module of parenthesized braids of genus $g$} is the operad module in groupoids having $\mathbf{Pa}$ as an operad of objects
    \begin{equation*}
        \mathbf{PaB}^f_g(n) \defeq \pi_1(\overline{\mathrm{Conf}}^f(\Sigma_g, n),\mathbf{Pa}),
    \end{equation*}
    where $\overline{\mathrm{Conf}}^f(\Sigma_g,n)$ is the ASFM compactification of $\mathrm{Conf}^f(\Sigma_g,n)$, with $\Sigma_g$ a compact oriented surface of genus $g$.
\end{definition}

We use the notation
\begin{equation*}
    \widehat{\square}:\mathbf{Grpds} \rightarrow\mathbf{Cat}(\mathbf{CoAlg)}
\end{equation*}
for the prounipotent completion functor. Concretely, given a group $\Gamma$, it endows the group algebra $\mathbb{K}\Gamma$ with the Hopf algebra structure
\begin{equation*}
    \Delta(g)=g \otimes g, \hspace{1em} S(g)=g^{-1}, \hspace{1em} \varepsilon(g)=1
\end{equation*}
for all $g\in \Gamma$, and completes $\mathbb{K}\Gamma$ with respect to the filtration defined by powers of the augmentation ideal.

Let $F_1 \in \mathrm{End}_{\widehat{\mathbf{PaB}}(1)}(1)$ denote the rotation of the framing by $360^\circ$. In the notation of Section \ref{operad_sec}, we propose that
\begin{conjecture}\label{conj}
    Let $n\geq 0$, and let $f$ be a framing on the surface $\Sigma_{g,n+1}$. Then there exists a section $\gamma^f:\overline{K}_{n+2} \rightarrow K_{n+2}$, and an isomorphism of $\overline{K}_n \oplus\overline{K}_n$-modules $\varphi: M \rightarrow U\overline{K}_{n+2}$ such that
    \begin{equation}
        ((\widehat{\mathbf{PaB}}^f, \widehat{\mathbf{PaB}}^f_g, F_1), \gamma^f, \varphi) \in \mathrm{OpR}^\Delta_f(n+2)
    \end{equation}
    and 
    \begin{equation*}
        \begin{tikzcd}
            \{(\widehat{\mathbf{PaB}}^f, \widehat{\mathbf{PaB}}^f_g, F_1), \gamma^f, \varphi\} \arrow[r, "\Lambda_{n+2}", mapsto]& (|\widehat{\mathbb{K}\pi}|, [-,-], \delta^f),
        \end{tikzcd}
    \end{equation*}
    where the right-hand side stands stands for the Goldman-Turaev bialgebra on $\Sigma_{g,n+1}$ associated to $f$.
\end{conjecture}
Note that this conjecture is reminiscent of the work of Massuyeau in \cite{Mass}, where the author derives Fox pairings and quasi-derivations directly from the fundamental groups of punctured discs. Conjecture \ref{conj} is backed by \cite{tani}, where Taniguchi constructs solutions of the higher genus KV problems starting from Gonzalez' higher genus associators.
\begin{definition}[\cite{gonz}, Definition 3.16]
The set of genus $g$ associators is the set of isomorphisms in $\mathrm{OpR}\, \mathbf{Cat}(\mathbf{CoAlg})$
\begin{equation*}
    \mathrm{Assoc}_g \defeq \mathrm{Iso}_{\mathrm{OpR}\, \mathbf{Cat}(\mathbf{CoAlg})}^+((\widehat{\mathbf{PaB}}^f, \widehat{\mathbf{PaB}}^f_g),(\mathbf{PaCD}^f, \mathbf{PaCD}^f_g))
\end{equation*}
which are the identity on objects.
\end{definition}
If Conjecture \ref{conj} were true, then starting from a genus $g$ associator $\mu$, one could produce an isomorphism of Lie bialgebras
\begin{equation*}
    (|\widehat{\mathbb{K}\pi}|, [-.-], \delta^f) \rightarrow (|A|, [-,-]_{\mathrm{gr}}, \delta^f_{\mathrm{gr}})
\end{equation*}
between the Goldman-Turaev bialgebra and its associated graded. But these isomorphisms are essentially understood as solutions to a higher genus Kashiwara-Vergne problem. We refer the reader to \cite[\S 6.4]{Flor1} for more details.

\appendix
\addcontentsline{toc}{section}{Appendices}
\section*{Appendices}
\section{Auxiliary results}

\subsection{On quotients of free Lie algebras}\label{F-well-def}

Throughout this section, if $\mathfrak{a}$ is a Lie algebra and $A \subset \mathfrak{a}$, we write $\langle A \rangle_{\mathfrak{a}}$ for the ideal of $\mathfrak{a}$ generated by $A$. We will omit the subscript $\mathfrak{a}$ whenever it can be inferred from context.

Let $V$ denote the vector space over $\mathbb{K}$ with basis $\{x_i\}_{i\in I}$, $V= \mathrm{Vec(\{x_i\}_{i\in I}})$. Similarly, we write $W=\mathrm{Vec}(\{y_j\}_{j\in J})$. Lastly, let $L(V)$ and $L(W)$ denote the degree completion of the free Lie algebras over the corresponding vector spaces.

\begin{lemma}
    Suppose we are given subsets
    \begin{equation*}
        R=\{r_i\}_{i\in K} \subset L(V), \hspace{1em} Q=\{q_j\}_{j\in L} \subset L(W), \hspace{1em} F=\{[x_i,y_j]-f_{ij}\, | \, i\in I, j\in J\} \subset L(V \oplus W),
    \end{equation*}
    where $f_{ij}\in L(w)$ for all $i$ and $j$. We abbreviate $\mathfrak{g}=L(V\oplus W)/\langle F \rangle_{L(V \oplus W)}$. If the inclusion\footnote{We are implicitly identifying $W$ with is image under the obvious injection $L(W) \rightarrow L(V \oplus W)/\langle F \rangle_{L(V \oplus W)}$. We do this repeatedly throughout this appendix.}
    \begin{equation}\label{assum}
        [R, L(W)]_{\mathfrak{g}}+[L(V), Q]_\mathfrak{g} \subset \langle Q \rangle_{L(W)}
    \end{equation}
    holds, then the map $\begin{tikzcd}
        L(V) \arrow[r, "\rho"] & \mathrm{L(W)}
    \end{tikzcd}$ induced by the choices
    \begin{equation}
        \rho(x_i)y_j=f_{ij},\, \text{for all}\,\, i\in I, j\in J, 
    \end{equation}
    descends to define a semi-direct product $L(V)/\langle R \rangle_{L(V)} \ltimes L(W)/\langle Q \rangle_{L(W)}$.
    \begin{proof}
        Let $i\in I$. To show that $\rho(x_i)$ preserves the ideal $\langle Q \rangle_{L(W)}$, notice that, for all $j \in L$,
        \begin{equation*}
            \rho(x_i)q_j=[x_i,q_j]_{\mathfrak{g}} \in [R, L(W)]_\mathfrak{g} \subset \langle Q \rangle_{L(W)},
        \end{equation*}
        where we used our assumption (\ref{assum}) for the last inclusion.
        To show that $\rho$ descends to the quotient $L(V)/\langle R \rangle_{L(V)}$, notice now
        \begin{equation*}
            \rho(r_i)y_j= [r_i, y_j]_{\mathfrak{g}}\in [R, L(W)]_{\mathfrak{g}} \subset \langle Q \rangle_{L(W)}
        \end{equation*}
        for all $i\in K$ and $j\in J$.
    \end{proof}

\begin{proposition} \label{main_app}
    The map
    \begin{equation*}
        \begin{split}
            \psi: \hspace{1em} L(V)/\langle R \rangle_{L(V)} \ltimes L(W)/\langle Q \rangle_{L(W)} & \rightarrow L(V \oplus W)/\langle R,Q,F \rangle_{L(V\oplus W)}\\
            x_i & \mapsto x_i\\
            y_j &\mapsto y_j
        \end{split}
    \end{equation*}
    is a Lie algebra isomorphism.
\end{proposition}
\begin{proof}
    Let $x_a, x_b \in V$, $y_a, y_b \in W$. The map $\psi$ is a Lie homomorphism by definition of the semi-direct product:
    \begin{equation*}
        \begin{split}
            \psi([x_a + y_a, x_b + y_b])&= \psi([x_a, x_b]_{L(V)/\langle R \rangle}+\rho(x_a)y_b - \rho(x_b)y_a + [y_a, y_b]_{L(W)/\langle Q \rangle})\\
            &=[x_a, x_b]_{L(V \oplus W)/\langle R, Q, F \rangle} + f_{ab}-f_{ba}+ [y_a, y_b]_{L(V \oplus W)/\langle R, Q, F \rangle}\\
            &=[\psi(x_a + y_a), \psi(x_b + y_b)].
        \end{split}
    \end{equation*}
    It is clear that $\psi$ is surjective. To prove injectivity, we propose a left inverse for $\psi$:
    \begin{equation*}
        \begin{split}
            \phi: \hspace{1em} L(V \oplus W)/\langle R,Q,F \rangle_{L(V\oplus W)} & \rightarrow  L(V)/\langle R \rangle_{L(V)} \ltimes L(W)/\langle Q \rangle_{L(W)}\\
            x_i & \mapsto x_i\\
            y_j & \mapsto y_j
        \end{split}
    \end{equation*}
    To show $\phi$ is well-defined, write $\langle R, Q, F \rangle_{L(V \oplus W)}=\langle R \rangle_{L(V \oplus W)}+ \langle Q \rangle_{L(V \oplus W)}+\langle F \rangle_{L(V \oplus W)}$. The ideals $\langle R \rangle_{L(V \oplus W)}$ and $ \langle Q \rangle_{L(V \oplus W)}$ are clearly nullified under $\phi$. As for $\langle F \rangle_{L(V \oplus W}$, notice
    \begin{equation*}
        \phi([x_i, y_j]-f_{ij})=\rho(x_i)y_j-f_{ij}=0
    \end{equation*}
    for all $i \in I$ and $j\in J$.
\end{proof}
\end{lemma}

\begin{proposition}
    Recall how we set $K_{g,n}=\ker (s_n: \mathfrak{t}_{g,n}^f \rightarrow t_{g,n-1}^f)$. The Lie algebra $\mathfrak{k}_{g,n}$, defined in Proposition \ref{def_kg,n}, is isomorphic to $K_{g,n}$.
\end{proposition}
\begin{proof}
    The map $s_n$ has a section, $d_{n-1}: \mathfrak{t}_{g,n-1}^f \rightarrow \mathfrak{t}_{g,n}^f$, so we already know
    \begin{equation} \label{sn_split}
        \mathfrak{t}_{g,n}^f=\ima d_{n-1} \ltimes K_{g,n}.
    \end{equation}
    Let $V$ be the vector space over $\mathbb{K}$ with basis the sets
    \begin{equation*}
        \{ x_i^a, y_i^a,t_{jk}\}_{1 \leq i,j,k \leq n-2;1\leq a \leq g},
    \end{equation*}
    and 
    \begin{equation*}
        \{x_{n-1}^a+x_{n}^a, y_{n-1}^a+y_n^a, t_{i(n-1)}+t_{in}, t_{(n-1)(n-1)}+ t_{nn}+t_{(n-1)n} \}_{1\leq i \leq n-2; 1\leq a \leq g}.
    \end{equation*}
    Notice that we may identify $\ima d_{n-1}$ with the Lie algebra $L(V)/\langle R\rangle_{L(V)}$, where $R$ is the set of braid relations on $\mathfrak{t}_{g,n}^f$.

    Let $W$ be the vector space over $\mathbb{K}$ with basis
    \begin{equation*}
        \{x_n^a, y_n^a, t_{in} \}_{1\leq i \leq n; 1 \leq a \leq g},
    \end{equation*}
    and $Q \subset L(W)$ the set of relations
    \begin{equation*}
        \sum_{a=1}^g [x_n^a, y_n^a]=-\sum_{j: j \neq n}t_{jn}-2(g-1)t_{nn},
    \end{equation*}
    \begin{equation*}
        [t_{nn}, v]=0 \hspace{1 em} \text{for all}\,\,v\in V.
    \end{equation*}
    By definition, $\mathfrak{k}_{g,n}=L(W)/\langle Q \rangle_{L(W)}$.
    
    Let $F \subset L(V \oplus W)$ be the set of relations
    \begin{gather*}
        [x_i^a, y_n^a]=\delta_{ab}t_{in},\,\, [y_i, x_n^a]=-\delta_{ab}t_{in} \hspace{1em} \text{for}\,\, {1\leq i\leq n-1},\\
        [x_i^a, x_n^b]=[y_i^a, y_n^b]=0 \hspace{1em} \text{for all}\,\, 1 \leq i \leq n-1,\\
        [x_k^a, t_{in}]=[y_k^a, t_{in}]=0 \hspace{1 em} \text{for}\,\, k \neq i,\\
        [x_i^a, t_{in}]=-[x_n^a, t_{in}],\,\, [y_i^a, t_{in}]=-[y_n^a, t_{in}] \hspace{1em} \text{for}\,\, 1\leq i \leq n-1,\\
        [t_{ij},t_{kn}]=0\hspace{1em} \text{if}\,\, \{i,j\}\cap \{k,n\}=\varnothing,\\
        [t_{ik}, t_{in}]=-[t_{kn},t_{in}]\ \hspace{1 em} \text{if}\,\, \{i,n\}\cap\{k\}=\varnothing.
    \end{gather*}
    Notice that every relation in $F$ can be written in the form
    \begin{equation*}
        [v_i, w_i]=f_{ij}
    \end{equation*}
    for basis elements $v_i \in L(V)$ and $w_i \in L(W)$, and some $f_{ij}\in L(W)$. Clearly $ \mathfrak{t}_{g,n}^f \cong L(V \oplus W)/\langle R, Q, F \rangle_{L(V \oplus W)}$. At the same time, Proposition \ref{main_app} gives the isomorphism
    \begin{equation*}
        \mathfrak{t}_{g,n}^f \cong \ima d_{n-1} \oplus \mathfrak{k}_{g,n}.
    \end{equation*}
   This last equation in conjunction with equation (\ref{sn_split}) implies the desired result.
\end{proof}

\begin{proposition}
    Recall how we set $H_{g,n}= \ker (s_{n-1}\circ s_n : \mathfrak{t}_{g,n}^f \rightarrow t_{g,n-2}^f)$. The Lie algebra $\mathfrak{h}_{g,n}$, defined in Proposition \ref{def_hg,n}, is isomorphic to $H_{g,n}$.
\end{proposition}
\begin{proof}
    The proof unfolds analogously to the preceding proposition. The map $s_{n-1} \circ s_{n}$ has the section $d_{n-1} \circ d_{n-2}: \mathfrak{t}_{g,n-2}^f \rightarrow \mathfrak{t}_{g,n}^f$, implying that
    \begin{equation} \label{sn-1_split}
        \mathfrak{t}_{g,n}^f \cong \ima d_{n-1} \circ d_{n-2} \ltimes H_{g,n}.
    \end{equation}
    Now we define $V$ to be the vector space over $\mathbb{K}$ with basis the sets
    \begin{equation*}
        \{ x_i^a, y_i^a,t_{jk}\}_{1 \leq i,j,k \leq n-3;1\leq a \leq g},
    \end{equation*}
    and
    \begin{multline*}
        \{x_{n-2}^a +x_{n-1}^a+x_{n}^a, y_{n-2}^a+ y_{n-1}^a+y_n^a, t_{i(n-2)}+ t_{i(n-1)}+t_{in}, \\ t_{(n-2)(n-2)}+ t_{(n-1)(n-1)}+ t_{nn}+t_{(n-2)(n-1)}+t_{(n-2)n} +t_{(n-1)n} \}_{1\leq i \leq n-3; 1\leq a \leq g}.
    \end{multline*}
    Notice that we may identify $\ima d_{n-1} \circ d_{n-2}$ with the Lie algebra $L(V)/\langle R \rangle_{L(V)}$, where $R$ is the set of braid relations on $\mathfrak{t}_{g,n}^f$.

    Let $W$ now be the vector space over $\mathbb{K}$ with basis
    \begin{equation*}
        \{x_{n-1}^a, y_{n-1}^a, t_{i(n-1)}, x_{n}^a, y_{n}^a, t_{in} \}_{1 \leq i \leq n; 1 \leq q \leq g},
    \end{equation*}
    and $Q \subset L(W)$ the set of relations
    \begin{gather*}
        \sum_{a=1}^g [x_i^a, y_i^a]=-\sum_{j: j \neq i}t_{ji}-2(g-1)t_{ii}, \hspace{1em} i=n,n-1;\\
        [x_n^a, y_{n-1}^b]=[x_{n-1}^a, y_{n}^b]=\delta_{ab}t_{(n-1)n};\\
        [x_n^a, x_{n-1}^b]=[y_n^a, y_{n-1}^b]=0;\\
        [x_n^a+x_{n-1}^a, t_{(n-1)n}]=[y_n^a+y_{n-1}^a, t_{(n-1)n}]=0;\\
        [t_{ij}, t_{kl}]=0 \hspace{1em} \text{if}\,\, \{i,j\} \cap \{k,l\}= \varnothing;\\
        [t_{(n-1)n}, t_{(n-1)k}+t_{nk}]=0 \hspace{1em} \text{if}\,\, k\neq n-1,n;\\
        [t_{(n-1)(n-1)},v]=[t_{nn},v]=0 \hspace{1em} \text{for all}\,\, v\in V.
    \end{gather*}
    By definition, $\mathfrak{h}_{g,n}=L(w)/\langle Q \rangle_{L(W)}$.

    Let $F \subset L(V \oplus W)$ be the set of relations
    \begin{gather*}
        [x_i^a, y_j^a]=\delta_{ab}t_{ij},\,\, [y_i, x_j^a]=-\delta_{ab}t_{ij} \hspace{1em} \text{for}\,\, {1\leq i\leq n-2}, j=n-1,n;\\
        [x_i^a, x_j^b]=[y_i^a, y_j^b]=0 \hspace{1em} \text{for all}\,\, 1 \leq i \leq n-2, j=n-1,n;\\
        [x_k^a, t_{ij}]=[y_k^a, t_{ij}]=0 \hspace{1 em} \text{for}\,\, k \neq i, j=n-1,n;\\
        [x_i^a, t_{ij}]=-[x_j^a, t_{ij}],\,\, [y_i^a, t_{ij}]=-[y_j^a, t_{ij}] \hspace{1em} \text{for}\,\, 1\leq i \leq n-2, j=n-1,n;\\
        [t_{ij},t_{k(n-1)}]=0\hspace{1em} \text{if}\,\, \{i,j\}\cap \{k,n-1\}=\varnothing;\\
        [t_{ik}, t_{i(n-1)}]=-[t_{k(n-1)},t_{i(n-1)}]\ \hspace{1 em} \text{if}\,\, \{i,n-1\}\cap\{k\}=\varnothing;\\
        [t_{ij},t_{kn}]=0\hspace{1em} \text{if}\,\, \{i,j\}\cap \{k,n\}=\varnothing;\\
        [t_{ik}, t_{in}]=-[t_{kn},t_{in}]\ \hspace{1 em} \text{if}\,\, \{i,n\}\cap\{k\}=\varnothing.
    \end{gather*}
    Notice that every relation in $F$ can be written in the form
    \begin{equation*}
        [v_i, w_i]=f_{ij}
    \end{equation*}
    for basis elements $v_i \in L(V)$ and $w_i \in L(W)$, and some $f_{ij}\in L(W)$. Clearly $ \mathfrak{t}_{g,n}^f \cong L(V \oplus W)/\langle R, Q, F \rangle_{L(V \oplus W)}$. At the same time, Proposition \ref{main_app} gives the isomorphism
    \begin{equation*}
        \mathfrak{t}_{g,n}^f \cong \ima d_{n-1} \circ d_{n-2} \oplus \mathfrak{h}_{g,n}.
    \end{equation*}
   This last equation in conjunction with (\ref{sn-1_split}) implies the desired result.
\end{proof}

\subsection{On braid Lie algebras} \label{iso_lemma}

We will assume familiarity with the notation of Section \ref{sec_braid_alg}. Our goal is to prove the result:
\begin{proposition}
    The two commutative diagrams
    \begin{equation*}
        \begin{tikzcd}
         & L^{n+2} \arrow[d, "\Delta"] \arrow[dl, "d_{n+2}^r" swap]\\
         \mathfrak{g} \arrow[r, "P" swap] & L^{n+2} \oplus L^{n+2}
    \end{tikzcd} \hspace{1em} \text{and} \hspace{1em}
    \begin{tikzcd}
        & L^{n+2} \arrow[d, "\Delta"] \arrow[dl, "\Delta \times T^r" swap]\\
         L^{n+2} \oplus L^{n+2} \times_G A \arrow[r, "P_1" swap] & L^{n+2} \oplus L^{n+2}
    \end{tikzcd}
    \end{equation*}
    are isomorphic as objects in $\mathbf{RelE}$.
\end{proposition}

We begin by showing that $\mathfrak{g}$ and $(L^{n+2} \oplus L^{n+2}) \times_G UL^{n+2}$ are isomorphic as Lie algebras. Explicitly, the map $G: (L^{n+2} \oplus L^{n+2})^{\oplus2} \rightarrow UL^{n+2}$ is defined by
\begin{equation*}
    \begin{split}
        G(v_1 \oplus v_2, w_1 \oplus w_2)= \tilde{\rho}_G(w_1, v_2)-\tilde{\rho}_G(v_1, w_2),
    \end{split}
\end{equation*}
for $v=v_1 \oplus v_2, \, w=w_1 \oplus w_2 \in L^{n+2} \oplus L^{n+2}$. One can check that $G$ does indeed satisfy the cocycle condition.

\begin{proposition}\label{iso1}
    The map $\varphi: \mathfrak{g} \rightarrow (L^{n+2} \oplus L^{n+2}) \times_G UL^{n+2}$ induced by
    \begin{equation}
        \begin{split}
            x_{n+2}^i & \mapsto (x^i_{n+2} \oplus 0,0) \hspace{1em} \text{for}\,\, 1\leq i \leq g;\\
            y_{n+2}^i &\mapsto (y^i_{n+2} \oplus 0, 0) \hspace{1em} \text{for}\,\, 1 \leq i \leq g;\\
            t_{1(n+2)} & \mapsto - \sum_{i=1}^{g} ([x^i_{n+2}, y^i_{n+2}]\oplus 0, 0)-\sum_{j: j \neq 1,n+3} (t_{j(n+2)} \oplus 0, 0)-(0 \oplus 0,1);\\
            t_{j(n+2)}& \mapsto (t_{j(n+2)} \oplus0,0) \hspace{1em} \text{for}\,\, 1 < j <n+3;\\
            t_{(n+2)(n+3)} & \mapsto (0 \oplus 0, 1);\\
            x_{n+3}^i & \mapsto (0 \oplus x^i_{n+2},0) \hspace{1em} \text{for}\,\, 1\leq i \leq g;\\
            y_{n+3}^i &\mapsto (0 \oplus y^i_{n+2}, 0) \hspace{1em} \text{for}\,\, 1 \leq i \leq g;\\
            t_{1(n+3)} & \mapsto - \sum_{i=1}^{g} (0\oplus [x^i_{n+2}, y^i_{n+2}], 0)-\sum_{j: j \neq 1,n+3} (0 \oplus t_{j(n+2)}, 0)-(0 \oplus 0,1);\\
            t_{j(n+3)}& \mapsto (0 \oplus t_{j(n+2)},0) \hspace{1em} \text{for}\,\, 1 < j <n+2;\\
        \end{split}
    \end{equation}
    is a well defined Lie morphism.
\end{proposition}
We can check by direct computation that $\varphi$ preserves the braid relations on $\mathfrak{g}$. The non-trivial part of the proof is to ensure $\varphi$ is well defined on the quotient. Before we proceed, we must observe:
\begin{lemma}
    Denote by $L_n = \mathrm{Lie}(x_1, \dots, x_n)$, the free Lie algebra in generators $x_i$, $1 \leq i \leq n$. As vectors spaces, $\langle x_n \rangle_{L_n} = [L_n, x_n] \oplus \mathbb{K}x_n$.
\end{lemma}
\begin{proof}
    The subspace $L_{n-1} \oplus [L_n, x_n] \oplus \mathbb{K}x_n$ is a subalgebra of $L_n$ that contains all the generators. Clearly, $\langle x_n \rangle_{L_n} \cap L_{n-1}=\varnothing$, so it must be $\langle x_n \rangle_{L_n} \subset [L_n, x_n] \oplus \mathbb{K}x_n$. The reverse inclusion is straightforward.
\end{proof}

\begin{proof}[Proof (of Proposition \ref{iso1}).]
    Note that $\langle t_{(n+2)(n+3)} \rangle_{\mathfrak{g}} \subset \tilde{L} \defeq \widehat{\mathrm{Lie}}(x_{n+2}^a, y_{n+2}^a, t_{2(n+2)}, \dots t_{(n-1)(n+2)}, t_{(n+2)(n+3)})$, since the latter is an ideal in $\mathfrak{g}$. Therefore,
    \begin{equation*}
        \langle t_{(n+2)(n+3)} \rangle_{\mathfrak{g}}=[\tilde{L}, t_{(n+2)(n++3)}] \oplus \mathbb{K}t_{(n+2)(n+3)}. 
    \end{equation*}
    We will be done if we can show $[[v, t_{(n+2)(n+3)}], t_{(n+2)(n+3)}] \in \ker \varphi$ for all $v\in \tilde{L}$. For this purpose, we can check that the set
    \[
        \{v \in \tilde{L} \,| \, \varphi([[v, t_{(n+2)(n+3)}], t_{(n+2)(n+3)}])=0\}
    \]
    is a subalgebra of the free Lie algebra that contains all its generators.
\end{proof}

Next we will show that $\varphi: \mathfrak{g} \rightarrow (L^{n+2} \oplus L^{n+2}) \times_G UL^{n+2}$ is in fact an isomorphism by exhibiting its inverse. First, notice that we have a short exact sequence
\begin{equation*}
    \begin{tikzcd}
        0 \arrow{r}& \langle t_{(n+2)(n+3)} \rangle_{\mathfrak{g}} / [\langle t_{(n+2)(n+3)} \rangle, \langle t_{(n+2)(n+3)} \rangle]_{\mathfrak{g}} \arrow[r] & \mathfrak{g} \arrow [r, "P_1 \circ \varphi"] & L^{n+2} \oplus L^{n+2} \arrow[r] & 0,
    \end{tikzcd}
\end{equation*}
where $P_1$ is the projection $P_1:(L^{n+2} \oplus L^{n+2}) \times_G UL^{n+2} \rightarrow L^{n+2} \oplus L^{n+2}$. The composition $P_1 \circ \varphi$ has a section $\alpha: L \oplus L \rightarrow \mathfrak{g}$, induced by the choices
\begin{equation}
        \begin{tikzcd} [row sep=tiny]
            x^i_{n+2} \oplus 0 \arrow[r, maps to, "\alpha"] & x_{n+2}^i,\\
            y^i_{n+2} \oplus 0 \arrow[r, maps to]& y_{n+2}^i,\\
            t_{j(n+2)} \oplus 0 \arrow[r, maps to]& t_{j(n+2)},\\
            0 \oplus x^i_{n+2} \arrow[r, maps to] & x_{n+3}^i,\\
            0 \oplus y^i_{n+2} \arrow[r, maps to]& y_{n+3}^i,\\
            0 \oplus t_{j(n+2)} \arrow[r, maps to]& t_{j(n+3)},\\
        \end{tikzcd}
    \end{equation}
    where $1\leq i \leq g$ and $2 \leq j \leq n+1$. The algebra $\mathfrak{g}$ is an $(L^{n+2} \oplus L^{n+2})$-module with respect to the action
\begin{equation}\label{l2action}
    (v_1 \oplus v_2) \cdot g= [g, \alpha(v_1 \oplus v_2)].
\end{equation}

\begin{proposition} \label{psiinverse}
    Define a map $\psi: (L^{n+2} \oplus L^{n+2}) \times_G UL^{n+2} \rightarrow \mathfrak{g}$ by, $\psi(v, a)= \alpha(v) + \beta(a)$, where $\beta: UL^{n+2} \rightarrow \mathfrak{g}$ is the unique $(L^{n+2} \oplus L^{n+2})$-module map satisfying $\beta(1)=t_{(n+2)(n+3)}$. Then $\psi$ is a Lie algebra homomorphism.
\end{proposition}
We may abbreviate $\psi(x)=\alpha(x) + \beta(x)$, and understand that we are extending the maps $\alpha$ and $\beta$ by zero on the corresponding components of $(L^{n+2} \oplus L^{n+2}) \times_G UL^{n+2}$. Apriori, the action (\ref{l2action}) depends on the choice of section, but it restricts to an $(L^{n+2} \oplus L^{n+2})$-action on
\begin{equation*}
    M = \langle t_{(n+2)(n+3)} \rangle_{\mathfrak{g}} / [\langle t_{(n+2)(n+3)} \rangle, \langle t_{(n+2)(n+3)} \rangle]_{\mathfrak{g}}
\end{equation*}
which is independent from that choice. We will need the fact:
\begin{lemma} \label{betaint}
    For all $a\in UL^{n+2}$, $\beta(a) \in M$.
\end{lemma}
\begin{proof}
    It will be enough to show the statement holds for all monomials in $UL^{n+2}$. Let $a^{i_1}\dots a^{i_n}$ be a product of generators. Notice that
    \begin{equation*}
        \begin{split}
            \beta(a^{i_1} \dots a^{i_n})&= (-1)^n\beta((0 \oplus a^{i_n})\cdot( \dots (0 \oplus a^{i_1}) \cdot 1) \dots )\\
            &=(-1)^n[\dots [t_{(n+2)(n+3)}, \alpha(0 \oplus a^{i_1}) ] \dots , \alpha(0 \oplus a^{i_n})].
        \end{split}
    \end{equation*}
    Since this representative of $\beta(a^{i_1} \dots a^{i_n})$ belongs to $\mathrm{Lie}(x_{n+3}^a, y_{n+3}^a, t_{2(n+3)}, \dots t_{(n-1)(n+3)}, t_{(n+2)(n+3)})$, there is no way to suppress the $t_{(n+2)(n+3)}$ term.
\end{proof}

\begin{proof}[Proof (of Proposition \ref{psiinverse})]
    Let $(v, a), (w, b) \in (L^{n+2} \oplus L^{n+2}) \times_G UL^{n+2}$. Then
    \begin{equation*}
        \begin{split}
            \psi( [(v,a), (w, b)]) &= \psi (([v,w], w\cdot a - v \cdot b + c(v,w))\\
            &=\alpha([v,w]) + w\cdot \beta(a) - v \cdot \beta(b) + \beta(c(v,w)),
        \end{split}
    \end{equation*}
    while
    \begin{equation*}
        \begin{split}
            [\psi (v,a), \psi (w,b) ] &=[\alpha(v) + \beta(a), \alpha(\omega) + \beta(b)]\\
            &=[\alpha(v), \alpha(\omega)] + [\alpha(v), \beta(b)]+ [\beta(a), \alpha(w)],
        \end{split}
    \end{equation*}
    where on the last line we used Lemma (\ref{betaint}). Comparing with the definition of the action (\ref{l2action}), all we need to show is that
    \begin{equation}
        \beta(c(v, w))=[\alpha(v), \alpha(w)]-\alpha([v,w]).
    \end{equation}
    Substituting $v=(v_1 \oplus v_2)$, $w=(w_1 \oplus w_2)$, the preceding equation becomes
    \begin{equation*}
        \beta(\rho_G(w_1, v_2)-\rho_G(v_1, w_2))=[\alpha(v_1 \oplus 0), \alpha(w_1 \oplus 0)] + [\alpha(v_1 \oplus 0), \alpha(0 \oplus w_2)].
    \end{equation*}
    We will be done if we can show that
    \begin{equation*}
        \beta(\rho_G(y,z))=[\alpha(y \oplus 0), \alpha(0 \oplus z)].\\
    \end{equation*}
    To show the latter, we prove that the set $B \defeq \{y \in L \, | \, \beta(\tilde{\rho}_G(y, a^i))= [\alpha(y \oplus 0), \alpha(0 \oplus a^i)] \}$, defined for a fixed generator $a^i$, is a subalgebra of $L^{n+2}$ that contains all generators. Indeed, if $y, z \in UL^{n+2}$, then
    \begin{equation*}
        \begin{split}
            \beta(\tilde{\rho}_G([y,z]),a^i)=& (y \oplus 0) \cdot \beta(\tilde{\rho}_G(z, a^i))- (z \oplus 0) \cdot \beta(\tilde{\rho}_G(y, a^i))\\
            &=[[\alpha(z\oplus 0), \alpha(0 \oplus a^i)],\alpha(y \oplus 0)]-[[\alpha(y \oplus 0), \alpha(0 \oplus a^i)], \alpha(z \oplus 0)]\\
            &=-[[\alpha(y \oplus 0), \alpha(z \oplus 0)], \alpha(0 \oplus a^i)]\\
            &=-[\alpha([y, z] \oplus 0), \alpha(0 \oplus a^i)],
        \end{split}
    \end{equation*}
    so that $[y, z] \in B$. We can also check by direct computation that all the generators of $L^{n+2}$ are indeed in $B$.
\end{proof}

We can directly verify that
\begin{corollary} \label{Comp_coro}
    The map $\varphi: \mathfrak{g} \rightarrow (L^{n+2} \oplus L^{n+2}) \times_G UL^{n+2}$ is an isomorphism of Lie algebras, with inverse $\psi: (L^{n+2} \oplus L^{n+2}) \times_G UL^{n+2} \rightarrow \mathfrak{g}$.
\end{corollary}
More importantly, the commutativity of the diagram 
\begin{equation} \label{g-iso}
    \begin{tikzcd}
        0 \arrow[r]& UL^{n+2} \arrow[r] \arrow[d, "g"] & (L^{n+2} \oplus L^{n+2}) \times_G UL^{n+2} \arrow[r, "P_1"] \arrow[d, "\psi"] & L^{n+2} \oplus L^{n+2} \arrow[r] \arrow[d, "\mathrm{id}"] & 0\\
        0 \arrow[r] & M \arrow[r] & \mathfrak{g} \arrow[r, "P_1 \circ \varphi"] & L^{n+2} \oplus L^{n+2} \arrow[r] & 0,
    \end{tikzcd}
\end{equation}
defines an invertible linear map between abelian Lie algebras, $g: UL^{n+2} \rightarrow M$.
\begin{corollary}\label{M_is_free}
    The spaces $UL^{n+2}$ and $M$ are isomorphic as $(L^{n+2} \oplus L^{n+2})$-modules.
\end{corollary}
\begin{proof}
    Recall that if $\gamma: M_1 \rightarrow M_2$ is an invertible module map, then $\gamma^{-1}: M_2 \rightarrow M_1$ is automatically a module map. Therefore, all we need to do is verify that $g: L^{n+2} \rightarrow M$ is a module map, but this follows from the commutativity of (\ref{g-iso}). To check this, notice that, for any $v, w \in L^{n+2}$, $(0, (v \oplus w)\cdot 1)=[(0,1), (v \oplus w,0)]$, then
    \begin{equation*}
        \begin{tikzcd}
            (v \oplus w) \cdot 1 \arrow[r, mapsto] \arrow[d, mapsto, "g"] & {[(0,1), (v \oplus w, 0)]} \arrow[d, mapsto, "\psi"]\\
            (v \oplus w) \cdot t_{(n+2)(n+3)} \arrow[r, mapsto] & {[t_{(n+2)(n+3)}, \alpha(v \oplus w)]}
        \end{tikzcd}
    \end{equation*}
\end{proof}

We are finally ready to prove our main result:
\begin{proposition}
    The two commutative diagrams
    \begin{equation*}
        \begin{tikzcd}
         & L^{n+2} \arrow[d, "\Delta"] \arrow[dl, "d_{n+2}^r" swap]\\
         \mathfrak{g} \arrow[r, "P" swap] & L^{n+2} \oplus L^{n+2}
    \end{tikzcd} \hspace{1em} \text{and} \hspace{1em}
    \begin{tikzcd}
        & L^{n+2} \arrow[d, "\Delta"] \arrow[dl, "\Delta \times T^r" swap]\\
         L^{n+2} \oplus L^{n+2} \times_G UL^{n+2} \arrow[r, "P_1 \circ \varphi" swap] & L^{n+2} \oplus L^{n+2}
    \end{tikzcd}
    \end{equation*}
    are isomorphic as objects in $\mathbf{RelE}$.
\end{proposition}
\begin{proof}
    Our statement is equivalent to the commutativity of the diagrams
    \small
    \begin{equation*}
        \begin{tikzcd}
            L^{n+2} \arrow[r, "\Delta"] \arrow[d, "\mathrm{id}"] & L^{n+2} \oplus L^{n+2} \arrow[d, "\mathrm{id}"]\\
            L^{n+2} \arrow[r, "\Delta"] & L^{n+2} \oplus L^{n+2}
        \end{tikzcd} \hspace{1em}
        \begin{tikzcd}
            (L^{n+2} \oplus L^{n+2}) \times_G UL^{n+2} \arrow[r, "P_1"] \arrow[d, "\psi"] & L^{n+2} \oplus L^{n+2} \arrow[d, "\mathrm{id}"]\\
            \mathfrak{g} \arrow[r, "P"] & L^{n+2} \oplus L^{n+2} 
        \end{tikzcd} \hspace{1em}
        \begin{tikzcd}
            L^{n+2} \arrow[r, "\Delta \times T^r"] \arrow[d, "\mathrm{id}"] & (L^{n+2} \oplus L^{n+2}) \times_G UL^{n+2} \arrow[d, "\psi" ]\\
        L^{n+2} \arrow[r, "d_{n+2}^r"]& \mathfrak{g}
        \end{tikzcd}
    \end{equation*}
    \normalsize
    since the vertical arrows that comprise them are invertible. The commutativity of the first two follows from straightforward computations. For the third diagram, note that, by definition, $T^r:L^{n+2} \rightarrow UL^{n+2}$ satisfies
    \begin{equation*}
        T^r(x)=q^r(x) \hspace{1em} \text{for all }x\in L^{n+2}.
    \end{equation*}
    Let $a_i$ be a generator of $L^{n+2}$, then
    \begin{equation*}
        \begin{split}
            \psi(\Delta \times T^r)(a_i) &=\alpha\Delta(a_i)+\beta(r(a_i))\\
            &=\alpha\Delta(a_i) + r(a_i)\beta(1)\\
            &=\alpha\Delta(a_i) + r(a_i)t_{(n+2)(n+3)}.
        \end{split}
    \end{equation*}
    On the other hand,
    \begin{equation*}
        d_{n+2}^r(a_i)=d_{n+2}(a_i)+r(a_i)t_{(n+2)(n+3)}.
    \end{equation*}
    We can check how each generator $a_i$ of $L^{n+2}$ satisfies the equality $\alpha\Delta(a_i)=d_{n+2}(a_i)$.
\end{proof}

\subsection{On the equivalence of $\mathbf{RelCo}$ and $\mathbf{RelE}$}

We aim to complete the proof of the following proposition:
\begin{proposition} \label{cat_appendix}
    The functor $\mathrm{F}: \mathbf{RelCo} \rightarrow \mathbf{RelE}$, defined on objects by
    \begin{equation*}
        ( \begin{tikzcd}
            \mathfrak{h} \arrow[r,"f"]& \mathfrak{g}
        \end{tikzcd},\, M,\, \omega \oplus c \in Z^2(\mathfrak{g}, \mathfrak{h}; M) ) \mapsto \begin{tikzcd}
            & \mathfrak{h} \arrow[d, "f"] \arrow[dl, "f \times \omega" swap]\\
            \mathfrak{g}\times_c M \arrow[r] & \mathfrak{g}
        \end{tikzcd},
    \end{equation*}
    and on morphisms by
    \begin{equation*}
        \left( \begin{tikzcd}
            \mathfrak{h} \arrow[r, "f"] \arrow[d, "H"] & \mathfrak{g} \arrow[d, "G"]\\
            \mathfrak{h}' \arrow[r, "f'"]& \mathfrak{g}'
        \end{tikzcd},\, 
        \begin{tikzcd}
            M \arrow[d, "\alpha"]\\
            M'
        \end{tikzcd},\,
        r \in C^1(\mathfrak{g}, \mathfrak{h}; M)
        \right) \mapsto
        \left( \begin{tikzcd}
            \mathfrak{h} \arrow[r, "f"] \arrow[d, "H"] & \mathfrak{g} \arrow[d, "G"]\\
            \mathfrak{h}' \arrow[r, "f'"]& \mathfrak{g}'
        \end{tikzcd},\, 
        \begin{tikzcd}
            \mathfrak{h} \arrow[r, "f \times \omega"] \arrow[d, "H"] & \mathfrak{g}\times_c M \arrow[d, "\sigma"]\\
            \mathfrak{h}' \arrow[r, "f' \times \omega'"]  &\mathfrak{g}' \times_{c'} M'
        \end{tikzcd},\,
        \begin{tikzcd}
            \mathfrak{g} \times_c M \arrow[r] \arrow[d, "\sigma"]& \mathfrak{g} \arrow[d, "G"] \\
            \mathfrak{g}' \times_{c'} M' \arrow[r] & \mathfrak{g}'
        \end{tikzcd}
        \right),
    \end{equation*}
    where $\sigma(x,m)= (G(x), \alpha (m) +r(x))$ for all $(x,m)\in \mathfrak{g}\times_c M$, is an equivalence of categories.
\end{proposition}
In the body of the article we proved that $\mathrm{F}$ is faithful and essentially surjective. We will now prove that $\mathrm{F}$ is well-defined.

Keeping with the notation of Proposition \ref{cat_appendix}, we need to check that:
\begin{enumerate}
    \item \textbf{The map $f \times \omega$ is a Lie morphism}: Let $h_1, h_2 \in \mathfrak{h}$, then
    \begin{equation*}
        \begin{split}
            [(f(h_1),\omega(h_1)), (f(h_2),\omega(h_2))]&=( [f(h_1), f(h_2)],f(h_2) \cdot \omega(h_1) -f(h_1) \cdot \omega(h_2)+c(f(h_1), f(h_2)))\\
            &=(f[h_1, h_2], \omega[h_1, h_2]),
        \end{split}
    \end{equation*}
    where for the second equality we used that $\omega \oplus c$ is a closed cocycle, so that $d\omega=-f^*c$.
    \item \textbf{The map $\sigma$ is a Lie morphism:} Recall how we imposed the requirement
    \begin{equation} \label{delta_r}
        \delta r=(\alpha \omega - \omega' H) \oplus(\alpha c- c' (G \otimes G)), 
    \end{equation}
    so in particular 
    \begin{equation} \label{d_r}
        -dr =\alpha c-c'(G \otimes G).
    \end{equation}
    Let $(g_1,m_1), (g_2, m_2)\in \mathfrak{g} \times_c M$, then
    \begin{equation*}
        \begin{split}
            \sigma [(g_1, m_1), (g_2, m_2)]&= \sigma([g_1, g_2], g_2 \cdot m_1- g_1 \cdot m_2+c(g_1, g_2))\\
            &=( G[g_1, g_2], \alpha(g_2 \cdot m_1 - g_1 \cdot m_2+ c(g_1, g_2))+r[g_1, g_2])\\
            &= ( G[g_1, g_2], G(h_2) \cdot \alpha(m_1)- G(g_1) \cdot \alpha(m_2)+ \alpha c(g_1, g_2)+r[g_1, g_2])\\
            &= ( G[g_1, g_2], G(g_2) \cdot \alpha(m_1)\\
            &- G(g_1) \cdot \alpha(m_2)+ G(h_2) \cdot r(g_1)-G(g_1)\cdot r(g_2)+c'(G(g_1), G(g_2))\\
            &=[(G(g_1), \alpha(m_1) + r(g_1)), (G(g_1), \alpha(m_1) + r(g_1))]\\
            &=[\sigma(g_1, m_1), \sigma(g_2, m_2)],
        \end{split}
    \end{equation*}
    where on the third equality we used the requirement that $\alpha(g \cdot m)=G(g) \cdot \alpha(m)$ for all $g\in \mathfrak{g}$ and $m \in M$, and we used equation (\ref{d_r}) for the fourth equality.
    \item \textbf{The diagrams }
    \begin{equation*}
            \mathcal{A} \defeq
            \begin{tikzcd}
                \mathfrak{h} \arrow[r, "f \times \omega"] \arrow[d, "H"] & \mathfrak{g}\times_c M \arrow[d, "\sigma"]\\
                \mathfrak{h}' \arrow[r, "f' \times \omega'"]  &\mathfrak{g}' \times_{c'} M'
            \end{tikzcd}, \hspace{1 em}
            \mathcal{B} \defeq
            \begin{tikzcd}
                \mathfrak{g} \times_c M \arrow[r] \arrow[d, "\sigma"]& \mathfrak{g} \arrow[d, "G"] \\
                \mathfrak{g}' \times_{c'} M' \arrow[r] & \mathfrak{g}'
            \end{tikzcd},
    \end{equation*}
    \textbf{are commutative:} In particular, equation (\ref{delta_r}) also implies that
    \begin{equation}\label{f_r}
        -f^*r=\alpha\omega-\omega'H.
    \end{equation}
    Let $h \in \mathfrak{h}$. To check the commutativity of $\mathcal{A}$, notice that
    \begin{equation*}
        \begin{split}
            \sigma(f(h), \omega(h))&=(Gf(h), \alpha\omega(h)+rf(h))\\
            &=(f'H(h), \omega'H(h)),
        \end{split}
    \end{equation*}
    where on the second equality we used the condition $Gf=f'H$, as well as equation (\ref{f_r}). The commutativity of $\mathcal{B}$ is straightforward.
    \item \textbf{The map $\mathrm{F}: \mathbf{RelCo} \rightarrow \mathbf{RelE}$ respects composition:} Suppose we have two morphisms $\mathcal{X}, \mathcal{Y}$ in $\mathbf{RelCo}$:
    \begin{equation*}
        \begin{tikzcd}
            \left(
        \mathfrak{h} \xrightarrow{f} \mathfrak{g}, M, \omega \oplus c
        \right) \arrow[r, "\mathcal{X}"]& \left(
        \mathfrak{h}' \xrightarrow{f'} \mathfrak{g}', M', \omega' \oplus c'
        \right) \arrow[r, "\mathcal{Y}"] & \left(
        \mathfrak{h}'' \xrightarrow{f''} \mathfrak{g}'', M'', \omega'' \oplus c''
        \right)
        \end{tikzcd}.
    \end{equation*}
    To verify that $\mathrm{F}(\mathcal{X} \circ_{\mathbf{RelCo}}\mathcal{Y})=\mathrm{F}(\mathcal{X})\circ_\mathbf{RelE} \mathrm{F}(\mathcal{Y})$, we essentially only need to check that 
    \begin{equation*}
        \sigma_{\mathrm{F}(\mathcal{X})} \circ \sigma_{\mathrm{F}(\mathcal{Y})} = \sigma_{\mathrm{F}(\mathcal{X}\circ_{\mathbf{RelCo}} \mathcal{Y} )}.
    \end{equation*}
    Indeed, let us abbreviate $\sigma =\sigma_{\mathrm{F}(\mathcal{X})}$ and $\sigma'=\sigma_{\mathrm{F}(\mathcal{Y})}$. Then given any $(g, m) \in \mathfrak{g} \times_c M$,
    \begin{equation*}
        \begin{split}
            \sigma' \circ \sigma (g,m)&=(G(g), \alpha(m)+r(g))\\
            &=(G'  G(g), \alpha' \alpha (m)+\alpha' r(g)+r'G(g))\\
            &= (G'  G(g), \alpha' \alpha (m)+r' \circ_{\mathbf{RelCo}}r(g)).
        \end{split}
    \end{equation*}
\end{enumerate}

\section{Automorphisms in $\mathbf{Fox}^\eta$}\label{aut_in_relco}

In this subsection, we prove Theorem \ref{intro_KRV}, offering a different characterization of the higher genus $\mathrm{KRV}$ groups. 

Consider the following commutative diagram of Lie algebras $\mathcal{C}^f$:
\begin{equation} \label{maindiag}
    \begin{tikzcd}
    & L \arrow[ld, "\Delta \times T^f" swap] \arrow[d, "\Delta"]\\
    (L \oplus L) \times_G A \arrow[r, "p"]& L \oplus L
    \end{tikzcd}
\end{equation}
Recall that $L$ is the degree completion of the free Lie algebra in generators $\{x_i, y_i, z_j\}$, where $1\leq i \leq g$ and $1 \leq j \leq n$. We are regarding $A \defeq UL$ as a commutative Lie algebra. As a vector space, $(L \oplus L) \times_G A  = (L \oplus L) \oplus A $, and the bracket in the former is given by
\[
    [(x, a), (y,b)]= ([x,y], y \cdot a - x \cdot b + G(x,y)).
\]
The map $G: \bigwedge^2 L \oplus L \rightarrow A$ is the Chevalley-Eilenberg cochain in $Z^2(L \oplus L, A)$ that corresponds to the Fox pairing $\rho_G$ under the equivalence from Theorem \ref{quasi-iso}. Likewise, $T^f: L \rightarrow A$ is the Chevalley-Eilenberg 1-cochain uniquely determined by the requirement that $T^f \oplus G$ correspond to the quasi-derivation $q^f$. In formulas,
\begin{gather*}
    G(v,w)=\rho_G(v_1, w_2)-\rho_G(w_1, v_2) \hspace{1em} \text{for all }v,w\in L \oplus L;\\
    T^f(v)=q^f(v)\hspace{1 em} \text{for all }v\in L.
\end{gather*}
Lastly, the projection $p$ is given by
\begin{gather}
    p(x, a)=x\, \text{for all }(x,a)\in (L \oplus L) \times_c A;
\end{gather}
we can identify its kernel with the abelian Lie algebra $A$.

Diagram $\mathcal{C}^f$ is an extension of $L \oplus L$ by $A$ relative to the diagonal map $\Delta: L \rightarrow L \oplus L$. The object that $\mathcal{C}^f$ defines in the category $\mathbf{RelE}$ corresponds to
\begin{equation*}
    \begin{split}
        \mathcal{C}^f &= ( \Delta: L \rightarrow L \oplus L,\, A,\, T^f \oplus G \in \mathcal{Z}^2(L \oplus L, L; A)) \in \mathrm{Ob}(\mathbf{RelCo})\\
        &=(L,\, \rho_G \oplus q^f \in Z^2(\mathcal{M}(\mathrm{id}_L))) \in \mathrm{Ob}(\mathbf{Fox}),
    \end{split}
\end{equation*}
with respect to the categorical equivalences of Proposition \ref{cat-equiv} and \ref{def_E}, where we have abusively denoted all these objects by the same symbol. Our goal is to prove the result:
\begin{theorem}\label{biiig_theo}
    Let $\mathrm{Aut}(L, [-,-]_{\mathrm{gr}}, \delta_{\mathrm{gr}}^f)$ denote the set of Lie automorphisms of $L$ preserving the associated graded to the Goldman bracket and the Turaev cobracket on $\Sigma_{g,n+1}$. Then the following is a split exact sequence:
    \begin{equation}
        \begin{tikzcd}
            \mathrm{LFox}(A) \cap \mathrm{RFox}(A) \arrow[r] & \mathrm{Aut}_{\mathbf{Fox}}(\mathcal{C}^{f}) \arrow[r, "\Psi"] & \mathrm{Aut}(L, [-,-]_{\mathrm{gr}}, \delta_{\mathrm{gr}}^{f})
        \end{tikzcd}.
    \end{equation}
\end{theorem}
To accomplish this, we need to understand how special automorphisms act on Fox pairings and quasi-derivations. 

\subsection{Special automorphisms}

\begin{definition}
    Recall an automorphism $f\in \mathrm{Aut}_{\mathbf{Lie}}(L)$ of positive grading is special if there exist $f_j \in \exp(L)$, for $j=1, \cdots,n$, such that
    \begin{equation*}
        f(z_j)=f_j^{-1}z_j f_j
    \end{equation*}
    for all $j$, and also $f(\omega)=\omega$, where
    \begin{equation*}
        \omega \defeq \sum_{i=1}^g[x_i, y_i]+\sum_{j=1}^nz_j.
    \end{equation*}
\end{definition}

\begin{proposition}[\cite{Flor1, MasTu}] \label{spec_fix_rho}
Any special automorphism $g$ preserves the Fox pairing $\rho_G$, in the sense that the following diagram is commutative:
\begin{equation} \label{spec_rho_G}
    \begin{tikzcd}
        A \otimes A \arrow[r, "\rho_G "] \arrow[d, "g \otimes g" swap] & A \arrow[d, "g"]\\
        A \otimes A \arrow[r, "\rho_G"] & A
    \end{tikzcd},
\end{equation}
\end{proposition}
Moreover, any $g\in \mathrm{Aut}(L, [-,-]_\mathrm{gr})$ is conjugate to a special automorphism \cite[Theorem 2.7]{Flor3}. That is, there exist a group-like $x\in A$ and some special automorphism $g_0$ such that
\begin{equation*}
    g(a)=x^{-1}g_0(a)x\quad \text{for all }a\in A.
\end{equation*}
This prompts us to understand how conjugation automorphisms act on Fox pairings:
\begin{lemma}
Let $x \in A$ be a group-like element, and let $h:A \rightarrow A$ denote conjugation by $x$, so that $h(a)=x^{-1}ax$ for all $a\in A$. Then any Fox pairing $\rho: A \otimes A \rightarrow A$ fits in a commutative diagram
\begin{equation} \label{conj_on_fox}
    \begin{tikzcd}
        A \otimes A \arrow[r, "\rho"] & A \\
        A \otimes A \arrow[r, "\rho + \rho_h" swap]\arrow[u, "h\otimes h"] & A \arrow[u, "h" swap]
    \end{tikzcd},
\end{equation}
where $\rho_h$ is an exact Fox pairing. 
\end{lemma}
\begin{proof}
    Let $\alpha_i, \alpha_j$ be any two generators in A. By direct computation we may verify that
    \begin{equation*}
        \rho \circ (h\otimes h)(\alpha_i, \alpha_j)=h(\alpha_i\rho(x, x^{-1})\alpha_j+\alpha_i\rho(x, \alpha_j)+\rho(\alpha_i, x^{-1})\alpha_j + \rho(\alpha_i, \alpha_j)).
    \end{equation*}
    It follows that diagram (\ref{conj_on_fox}) is commutative if we define the Fox pairing $\rho_h$ as
    \begin{equation}\label{rho_f}
        \rho_h(a,b)=D(a)\rho(x, x^{-1})D(b)+D(a)\partial^h_R(b)+\partial^h_L(a)D(b),
    \end{equation}
    where $\partial^h_L$ and $\partial^h_R$ are left and right Fox derivatives, respectively, defined by
    \begin{equation}\label{conj_Fox_der}
        \partial_L^{h}(a)=\rho(a,x^{-1}), \quad \partial_R^h(a)=\rho(x, a)
    \end{equation}
    for all $a\in A$. The first summand in equation (\ref{rho_f}) is an inner Fox pairing. Since all inner Fox pairings are exact (cf. Definition \ref{inner}), the Fox pairing $\rho_h$ is so too, expressly.
\end{proof}

We now turn our attention to quasi-derivations. We only need to be concerned with special automorphisms that also preserve the cobracket $\delta_{\mathrm{gr}}^f$. Before we properly begin, we will need the following lemma:
\begin{lemma}[\cite{Flor2}, Proposition A.2] \label{cyc_aux}
    Let $\alpha_i$ be a generator of $A$, and let $a \in A$ be a homogeneous element of degree $m$ such that $|\alpha_i^Na|=0$ for $N \geq m-1$. Then $a \in [\alpha_i, A]$.
\end{lemma}

\begin{lemma}
Let $g$ be a special automorphism, and suppose also that $g \in \mathrm{Aut}(L, \delta_{\mathrm{gr}}^f)$. Then there exists an inner derivation $\psi_g:A \rightarrow A$ making the following diagram commutative:
\begin{equation}\label{spec_q_f}
    \begin{tikzcd}
        A \arrow[r, "q^f"] \arrow[d, "g"] & A \arrow[d, "g"] \\
        A \arrow[r, "q^f +\psi_g" swap] & A
    \end{tikzcd}.
\end{equation}
\end{lemma}
\begin{proof}
     By inner derivation, we mean that there exits $v_g \in A$ such that
    \begin{equation*}
        \psi_g(a)=[v_g, a]
    \end{equation*}
    for all $a \in A$.
    
    We define $\psi_g \defeq g^{-1}q^fg-q^f$. Notice that $\psi_g$ is a derivation, in the sense that $\psi_g \in \mathrm{Qder}(0)$, since 
    \begin{equation*}
        g^{-1}q^fg \in \mathrm{Qder}(g^{-1}\rho_Gg)=\mathrm{Qder}(\rho_G),
    \end{equation*}
    where the previous equality of sets is a consequence of Proposition \ref{spec_fix_rho}. It remains to show that $\psi_g$ is in fact inner.

    The assumption that $g \in \mathrm{Aut}(L, \delta_{\mathrm{gr}}^f)$ means that $\delta_{\psi_g}=0$. The skew-symmetry of $q^f$ implies that of $\psi_g$, so we may write this as
    \begin{equation*}
        d_{\psi_g}-Pd_{\psi_g}=0.
    \end{equation*}
    In particular, this means that
    \begin{equation} \label{psi_null}
        0=(\varepsilon \otimes \mathrm{id})|\delta_{\psi_g}(a)|=|\psi_g(a)-S\psi_g(a)|
    \end{equation}
    for all $a\in A$. Let $\alpha_i$ denote the generators of $A$. We now evaluate (\ref{psi_null}) on $a=\alpha_i^N$, with the intention of using Lemma \ref{cyc_aux}. By alternating between arbitrarily large odd or even numbers, we learn that there must exist $b_i, c_i \in A$ such that
    \begin{equation*}
        \begin{split}
            \psi_g(\alpha_i) + S\psi_g(\alpha_i)& =[b_i, \alpha_i];\\
            \psi_g(\alpha_i)-S\psi_g(\alpha_i) &= [c_i, \alpha_i].
        \end{split}
    \end{equation*}
    Together, the last two equalities imply that $\psi_g(\alpha_i)=[a_i, \alpha_i]$ for $a_i=\frac{1}{2}(b_i+c_i).$ It remains to show that $a_i=a_j$ for all $i,j$. Equivalently, we will assume $a_1=0$ and show that this implies $a_i=0$ for all $i$. If we now evaluate equation (\ref{psi_null}) on $a= \alpha_1^N \alpha_j^M$, we get
    \begin{equation} \label{NM-null}
        |\alpha_j^M \alpha_1^N a_j-\alpha_1^N \alpha_j^M a_j|=0.
    \end{equation}
    The only way this last equality can hold for arbitrarily large $M$ and $N$ is if $a_j=0$. More formally, given any $n \geq 1$, let us define cyclic partial derivatives $D_i^n: |A| \rightarrow A$ by setting
    \begin{equation*}
        D_i^n(|\alpha_{i_1} \dots\alpha_{i_k}|) \defeq \sum_{l=1}^k \partial_j^n(\sigma^l \cdot \alpha_{i_1} \dots\alpha_{i_k}),
    \end{equation*}
    where $\partial_i:A \rightarrow A$ is the right Fox derivative defined in equation (\ref{fox_base}). The cyclic permutation $\sigma \in S_l$ is $\sigma \defeq (i_i \cdots i_k)$. A permutation $\gamma \in S_l$ acts on a monomial of degree $l$ in A by
    $\gamma \cdot \alpha_{i_1} \dots\alpha_{i_k} = \alpha_{\gamma(i_1)} \dots\alpha_{\gamma(i_k)}$.
    If we apply the transformation $\partial_i^N\circ D_j^M$ to equation (\ref{NM-null}), we obtain
    \begin{equation*}
        0=a_j +\partial_j^M \partial_1^N(a_j)[ \alpha_j^M \alpha_1^N+  \alpha_1^N \alpha_j^M].
    \end{equation*}
    However, we can always ensure the term $\partial_j^M \partial_1^N(a_j)$ is zero by choosing large enough numbers $N, M \in \mathbb{N}$, so necessarily $a_j=0.$
\end{proof}

We are similarly interested in the action of conjugation automorphisms on quasi-derivations.
\begin{lemma}
    Let $x \in A$ be a group-like element, and let $h:A \rightarrow A$ denote the conjugation by $x$, so that $h(a)=x^{-1}ax$ for all $a\in A$. Then any quasi-derivation $q: A  \rightarrow A$ fits in the commutative diagram
    \begin{equation}\label{conj_on_q}
        \begin{tikzcd}
            A \arrow[r, "q"] & A \\
            A \arrow[r, "q + q_h" swap] \arrow[u, "h"] & A \arrow[u, "h"]
        \end{tikzcd},
    \end{equation}
    where $q_h$ is an exact quasi-derivation.
\end{lemma}
\begin{proof}
    Let $a \in L \subset A$. By direct computation, we may verify that
    \begin{equation}\label{g_invqg}
        \begin{split}
            h^{-1}qh(a)=q(a)+xq(x^{-1})a+aq(x)x^{-1}+\rho_G(a,x)x^{-1}+x\rho_G(x^{-1},a).
        \end{split}
    \end{equation}
    We may use the equalities
    \begin{gather*}
    0=q(xx^{-1})=q(x)x^{-1}+xq(x^{-1})+\rho_G(x,x^{-1});\\
        0=\rho_G(a, xx^{-1})=\rho(a,x)x^{-1}+\rho(a,x^{-1});\\
        0=\rho_G(xx^{-1},a)=x\rho_g(x^{-1},a)+\rho_{G}(x, a);
    \end{gather*}
    to substitute for the relevant terms in equation (\ref{g_invqg}). Simplifying, our result follows if we set
    \begin{equation*}
        q_h(a)\defeq - D(a)q(x)x^{-1}-xq(x^{-1})D(a) - \partial_L^h(a)-\partial_R^h(a),
    \end{equation*}
    where the Fox derivatives $\partial_L^h$ and $\partial_R^h$ are defined in equation (\ref{conj_Fox_der}).
\end{proof}
We can summarize all our observations from this section in the following lemma:
\begin{lemma} \label{spec_on_q_p}
    Let $g \in \mathrm{Aut}(L, [-,-]_{\mathrm{gr}},\delta^f_\mathrm{gr})$ be expressible as the composition $g=h \circ g_0$, where $g_0$ is a special automorphism and $h \in \mathrm{Aut}(A)$ denotes conjugation by a group-like element $x \in A$. Then the following diagrams are commutative: 
    \begin{equation*}
        \begin{tikzcd}
            A \otimes A \arrow[d, "g \otimes g" swap] \arrow[r, "\rho_G"] & A \arrow[d, "g"]\\
            A \otimes A \arrow[r, "\rho_G + \rho_h" swap] & A
        \end{tikzcd}, \hspace{3 em}
        \begin{tikzcd}
            A \arrow[r, "q^f"]  & A\\
            A \arrow[r, "q^f + \psi_g + q_h " swap] \arrow[u, "g"] & A \arrow[u, "g"]
        \end{tikzcd}.
    \end{equation*}
    \begin{proof}
        The commutativity of the first diagram follows from diagrams (\ref{spec_rho_G}) and (\ref{conj_on_fox}). The commutativity of the second diagram in turn follows from (\ref{spec_q_f}) and (\ref{conj_on_q}).
    \end{proof}
\end{lemma}

\begin{proof}[Proof (of Theorem \ref{biiig_theo}).]
    Recall we have also denoted
    \begin{equation*}
        \mathcal{C}^f \defeq ( L,\, q^f \oplus \rho_G \in Z^2(\mathcal{M}(\mathrm{id}_\mathfrak{h}))).
    \end{equation*} 
    According to Theorem \ref{RelCo--GoTu}, if
    \begin{equation*}
        ( g\in \mathrm{Aut}_{\mathbf{Lie}}(L),\, r \in \mathcal{C}^1(\mathcal{M}(\mathrm{id}_{L})) ) \in \mathrm{Aut}_{\mathbf{Fox}}(\mathcal{C}^f),
    \end{equation*}
    then $g \in \mathrm{Aut}_\mathbf{GoTu}(A, [-,-]^{\rho_G}, \delta_{q^f})$. But according to Corollaries \ref{rho_G=gold} and \ref{qf=Tur}, respectively,
    \begin{equation*}
        [-,-]^{\rho_G}=[-,-]_{\mathrm{gr}} \hspace{1 em} \text{and}\hspace{1em}\delta_{q^f}=\delta_\mathrm{gr}^{f}.
    \end{equation*}
    For the other half of the statement, notice that $\Psi: \{f,r\} \mapsto\mathrm{id}_L$ if and only if $g=\mathrm{id}_L$ and 
    \begin{equation*}
        \delta r=-\mu( r)\oplus -\tau(r)=0.
    \end{equation*}
    Set $r=\partial_L \oplus \partial_R$, then the requirement that $-\mu(r)=0$ means is equivalent to the statement $\partial_L = - \partial_R$, which implies that $\partial_L$ is a right and left Fox derivative at once.

    All that remains to show is that the map $\psi$ is surjective. Let $g \in \mathrm{Aut}(L, [-,-]_{\mathrm{gr}}, \delta_{\mathrm{gr}}^f)$ be expressible as the composition $g=h \circ g_0$, where $g_0$ is a special automorphism and $h:A \rightarrow A$ denotes conjugation by a group-like element $x\in A$. By Lemma \ref{spec_on_q_p}, the quasi-derivation $q^f$ and the Fox pairing $\rho_G$ are preserved up to exact terms, in the sense that
    \begin{equation*}
        g^{-1}\rho_G(g \otimes g)=\rho_g + \rho_h \hspace{1 em} \text{and} \hspace{1em} g^{-1}q^fg=q^f+q_h+ \psi_g.
    \end{equation*}
    Now we notice that the term $q_h\oplus \rho_h + \psi_{g_0} \in \mathrm{Qder}(A) \oplus \mathrm{Fox}(A)$ is exact with respect to the total complex (\ref{goodone}). That is, if we define left and right Fox derivatives by
    \begin{equation*}
        \partial_L^{x}(a) \defeq  D(a)q(x)x^{-1}, \hspace{1 em} \partial_R^{x}(a) \defeq xq(x^{-1})D(a),
    \end{equation*}
    then we can see that
    \begin{equation*}
        q_h \oplus\rho_h=(\mu \oplus\tau)[(\partial_L^{x}+\partial_L^h)\oplus(\partial_R^{x}+\partial_R^h)],
    \end{equation*}
    where $\partial_L^h$ and $\partial_R^h$ are defined in equation (\ref{conj_Fox_der}). Similarly, if we set 
    \begin{equation*}
        \partial_L^v(a) \defeq D(a)v_{g_0}, \hspace{1em} \partial_R^v(a) \defeq v_{g_0} D(a)
    \end{equation*}
    for all $a \in A$, then 
    \begin{equation*}
        \psi_g \oplus0= (\mu \oplus\tau)[\partial_L^v\oplus\partial_R^v].
    \end{equation*}
    In other words, $g \in \mathrm{Aut}_{\mathbf{Fox}}(\mathcal{C}^f)$, and $\Psi(g)=g$.
\end{proof}

\begin{theorem}\label{diff_char}
    There exists an isomorphism
    \begin{equation*}
       \mathrm{Aut}_{\mathbf{Fox}^\eta}(\mathcal{C}^f)/ \mathbb{K} \cong \overline{\mathrm{KRV}}^f_{(g,n+1)},
    \end{equation*}
\end{theorem}
\begin{proof}
    We know
    \begin{equation*}
       \mathrm{Aut}_{\mathbf{Fox}}(\mathcal{C}^f)/ (\mathrm{LFox(A) \cap \mathrm{RFox(A}}) ) \cong \mathrm{Aut}_{\mathbf{GoTu}}(L, [-,-]_{\mathrm{gr}}, \delta_{\mathrm{gr}}^f) =\overline{\mathrm{KRV}}^f_{(g,n+1)},
    \end{equation*}
    where the first isomorphism follows from Theorem \ref{biiig_theo}. Lastly, in accordance with Lemma \ref{LcapR=D}, $(\mathrm{LFox(A) \cap \mathrm{RFox(A}}))\cong\mathbb{K}$.
\end{proof}

\section{The case of $\mathrm{KRV}_{(0,n)}^{f}$}\label{app_KRV}

Morally speaking, we can treat the genus zero case by formally setting $\mathbf{PaCD}^f_{g=0} \defeq \mathbf{PaCD}^f$. This is less than ideal, as it is not immediately obvious how to adapt some of our previous statements to the instance when $g=0$.  Instead, we will devote this appendix to handling this case. On the one hand, this is justified by the historical significance of the group $\mathrm{KRV}_{(0,3)}^{f^{\mathrm{adp}}}$, which acts freely and transitively on the set of solutions to the original Kashiwara-Vergne conjecture. On the other hand, building the genus zero case from the ground up will allow us to highlight a connection to the work of Kuno \cite{kuno} regarding the \textit{emergent Drinfeld-Kohno braid algebras} (see Definition \ref{em_drin}).

For the convenience of the reader, we recall some notation from Section \ref{krv_chap}. Let $\Sigma \defeq \Sigma_{0,n+1}$ be a compact oriented surface of genus zero with $n+1$ boundary components. We denote $\pi=\pi_1(\Sigma, *)$, where $*$ is a boundary point in $\Sigma$. Given a \textit{generating system} $\{\gamma_1, \dots, \gamma_n \}$ of $\pi$, we define a weight filtration on $\mathbb{K}\pi$ by setting
\begin{equation*}
    \mathrm{wt}(\gamma_j -1)=2
\end{equation*}
for $j=1, \dots, n$. Let $A$ denote the algebra of power series $A \defeq \mathbb{K}\langle \langle z_1, \dots, z_n \rangle \rangle$. We recall that there is a natural isomorphism
\begin{equation*}
    \mathrm{gr^{wt}}(\mathbb{K}\pi) \cong A
\end{equation*}
which maps $\mathrm{gr}(\gamma_j)$ to $z_j$ (cf. Proposition \ref{A_is_grad}). Via this isomorphism, we view the associated graded of the Goldman bracket and Turaev cobracket as maps
\begin{equation*}
    [-,-]_{\mathrm{gr}}:|A| \otimes|A| \rightarrow|A|, \hspace{2em} \delta^f_{\mathrm{gr}}: |A| \rightarrow|A| \otimes|A|,
\end{equation*}
where the Turaev cobracket depends on a choice of framing $f$ on $\Sigma$,
\begin{equation*}
    \begin{tikzcd}
        f: T\Sigma \arrow[r, "\cong" maps to] & \Sigma \times \mathbb{R}^2
    \end{tikzcd}.
\end{equation*}

We already described the Goldman bracket and Turaev cobracket in Section \ref{GT_bialg}, but we summarize the relevant information in the next two lemmas:

\begin{lemma}\label{rho_adp}
    Let $\rho_G:A \times A \rightarrow A $ be the Fox pairing with non-vanishing entries given by
    \begin{equation*}
        \rho_G(z_j, z_j)=-z_j
    \end{equation*}
    for $j=1, \dots, n$. Then $[-,-]^{\rho_G}=[-,-]_{\mathrm{gr}}$.
\end{lemma}

\begin{lemma}\label{q_adp}
    Given a framing $f$ on $\Sigma$, let $q^{f}$ be the quasi-derivation associated to the Fox pairing $-\rho_G$, and with values on generators given by
    \begin{equation*}
        q^{f}(z_j)=r^{f}(z_j)=\mathrm{rot}^f(\gamma_j)+1
    \end{equation*}
    for $j=1, \dots,n$. Then $\delta_{q^{f}}=\delta^{f}_{\mathrm{gr}}$.
\end{lemma}

Let $L$ denote the degree completed Lie algebra in generators $\{z_1, \dots, z_n\}$, so that $UL \cong A$. As before, we consider the object in $\mathbf{Fox}^\eta$
\begin{equation} \label{Cfadp_delta}
    \mathcal{C}^{f} \defeq (L, q^{f} \oplus \rho_G \in Z^2(\mathcal{M}(\mathrm{id_L))} ),
\end{equation}
and its counterpart in $\mathbf{RelE}$, which we denote by the same symbol:
\begin{equation*}
        \mathcal{C}^{f}=
    \begin{tikzcd}
        & L \arrow[d, "\Delta"] \arrow[dl, "\Delta \times T^{f}" swap]\\
         L \oplus L \times_G A \arrow[r, "P_1" swap] & L \oplus L
    \end{tikzcd}
    \end{equation*}
Here, $G \in \mathrm{Hom}_\mathbb{K}(\bigwedge^2 L \oplus L, A)$ and $T^{f} \in \mathrm{Hom}_{\mathbb{K}}(L,A)$ are the Chevalley-Eilenberg cochains corresponding to the relative cochain $q^{f} \oplus \rho_G$ under the equivalence of Theorem \ref{quasi-iso}. 

\begin{remark}\label{big_rem}
    By how we have defined $\mathcal{C}^{f}$, the functor $\Psi: \mathbf{Fox} \rightarrow \mathbf{GoTu}$ gives a non-trivial morphism
    \begin{equation*}
        \begin{tikzcd}
             \mathrm{Aut}_{\mathbf{Fox}^\eta}(\mathcal{C}^{f}) \arrow[r, "\Psi"] & \mathrm{Aut}_{\mathbf{GoTu}}(A, [-,-]_{\mathrm{gr}}, \delta^{f}_{\mathrm{gr}})= \overline{\mathrm{KRV}}_{(0,n+1)}^{f},
        \end{tikzcd}
    \end{equation*}
where the last equality follows from Lemmas \ref{rho_adp} and \ref{q_adp}.
\end{remark}

\subsection{Braid algebras}

We can naturally regard $\mathbf{PaCD}^f$ as a right $\mathbf{PaCD}^f$-operad module. Then, by definition,
\begin{equation*}
    \Pi_{n+1}(\mathbf{PaCD}^f, \mathbf{PaCD}^f)=
    \begin{tikzcd}[column sep=large]
        K_{n+1} \arrow[r, "\pi"] \arrow[d, "d_{n+1}"] & K_{n+1}/\langle t_{(n+1)(n+1)} \rangle \arrow[d, "\Delta"] \\
        \mathfrak{e} \arrow[r, "s_{n+2} \oplus s_{n+1}" swap] & (K_{n+1}/\langle t_{(n+1)(n+1)}\rangle)^{\oplus 2}
    \end{tikzcd},
\end{equation*}
where 
\begin{equation*}
    \mathfrak{e}\defeq (H_{n+2}/\langle t_{(n+1)(n+1)},t_{(n+2)(n+2)} \rangle)/[\ker (s_{n+2} \oplus s_{n+1}),\ker (s_{n+2} \oplus s_{n+1}) ], 
\end{equation*}
and $K_{n+1}$ and $H_{n+2}$ are part of short exact sequences
 \begin{equation*}
            \begin{tikzcd}[column sep= large, row sep= tiny]
            K_{n+1} \arrow[r] & \mathfrak{t}_{n+1}^f \arrow[r, "s_{n+1}"] & \mathfrak{t}_{n}^f,\\
            H_{n+2} \arrow[r] & \mathfrak{t}_{n+2}^f \arrow[r, "s_{n+1}\circ s_{n+2}"] & \mathfrak{t}_{n}^f.
            \end{tikzcd}
        \end{equation*}

In the same way that we handled the higher genus cases, we can use Proposition \ref{main_app} to write $K_{n+1}$ in terms of generators and relations. This yields that 
\begin{equation*}
    K_{n+1}/\langle t_{(n+1)(n+1)} \rangle=\widehat{\mathrm{Lie}}(t_{1(n+1)}, \dots , t_{n(n+1)})\eqdef L^n.
\end{equation*}
We also define
\begin{equation*}
    \overline{H}_{n+2} \defeq H_{n+2}/\langle t_{(n+1)(n+1)},t_{(n+2)(n+2)} \rangle.
\end{equation*}

\begin{lemma}\label{ker_sis}
    Let $N$ denote the abelian Lie subalgebra of $\mathfrak{e}$,
    \begin{equation*}
        N \defeq \langle t_{(n+1)(n+2)} \rangle_{\overline{H}_{n+2}}/ [\langle t_{(n+1)(n+2)} \rangle, \langle t_{(n+1)(n+2)} \rangle]_{\overline{H}_{n+2}}.
    \end{equation*}
    The following sequence is short-exact:
    \begin{equation*}
        \begin{tikzcd}
            0 \arrow[r]& N \arrow[r, "\iota"]& \mathfrak{e} \arrow[r, "P"] & L^{n} \oplus L^{n}
        \end{tikzcd},
    \end{equation*}
    where $P=s_{n+2} \oplus s_{n+1}$ and $\iota$ stands for the inclusion.
\end{lemma}
\begin{proof}
    As in the higher genus case, the injectivity of $\iota$ and surjectivity of $P$ are straightforward; our task is to show that $\ima \iota = \ker P$. Given that $t_{(n+2)(n+1)} \in \ker P$, we already know that $\ima \iota \subset \ker P$. For the reverse inclusion, notice the equality of vector spaces
    \begin{equation} \label{decomp}
        \mathfrak{e}= \langle t_{(n+1)(n+2)} \rangle_{\mathfrak{e}} \oplus (t_{1(n+1)}, \dots t_{n(n+1)})_{\mathfrak{e}} \oplus ( t_{1(n+2)},\dots, t_{n(n+2)})_{\mathfrak{e}},
    \end{equation}
    where the parentheses $( - )_{\mathfrak{g}}$ stand for the subalgebra generated in $\mathfrak{e}$. It is clear no element in $\mathfrak{e}-\langle t_{(n+1)(n+2)} \rangle_{\mathfrak{g}}$ maps to zero under $P$.
\end{proof}
As before, we point out $N$ is an abelian $L^n \oplus L^n$-module with respect to the action $g \cdot m =[m, u(g)]$
for $g \in L^n \oplus L^n$ and $m\in N$, where $u$ is any section\footnote{Since $N$ is abelian, this action is independent of the choice of section} of the projection $P: \mathfrak{e} \rightarrow L^n \oplus L^n$.

Notice also how the proof of Lemma \ref{ker_sis} implies that 
\begin{equation}
    \ker (s_{n+2} \oplus s_{n+1}: \overline{H}_{n+2} \rightarrow L^n \oplus L^n)=\langle t_{(n+1)(n+2)} \rangle.
\end{equation}
This observation is key, in that it relates $\mathfrak{e}$ to the emergent Drinfeld-Kohno Lie algebra of type (n,2). 

\begin{definition}[\cite{kuno}, Definition 3.3]\label{em_drin}
    The \textit{Drinfeld-Kohno Lie algebra of type (m,l)}, $\mathrm{dk}_{m,l}$, is the graded Lie algebra in generators $t_{i(m+j)}$ for $1\leq i \leq m$, $1 \leq j \leq l$, and $t_{(m+i)(m+j)}$ for $1 \leq i \neq j \leq l$, subject to relations (\ref{FL}) and (\ref{F4T}).
    
    Let $\mathbf{c}_{m,l}$ be the ideal in $\mathrm{dk}_{m,l}$ generated by the symbols $t_{(m+i)(m+j)}$ for $1 \leq i \neq j \leq l$. The \textit{emergent Drinfeld-Kohno Lie algebra of type (m,l)}, $\mathrm{edk}_{m,l}$, is the quotient algebra
    \begin{equation*}
        \mathrm{edk}_{m,l} \defeq \mathrm{dk}_{m,l}/[\mathbf{c}_{m,l},\mathbf{c}_{m,l}].
    \end{equation*}
\end{definition}

As a direct consequence of Proposition \ref{main_app}, we have that, by definition of $\mathrm{edk_{n,2}}$,
\begin{equation*}
    \widehat{\mathrm{edk}}_{n,2} \cong \mathfrak{e}.
\end{equation*}
In \cite{kuno}, Kuno describes the Lie bracket in $\mathrm{edk}_{n,2}$ in terms of a decomposition analogous to (\ref{decomp}). We rephrase their result in our notation, and in terms of $\mathfrak{e}$:

\begin{lemma}[\cite{kuno}, \S 3.2] The map $\varphi: \mathfrak{e} \rightarrow (L \oplus L) \times_G A$ given by
\begin{equation*}
    \begin{split}
        t_{j(n+1)}& \longmapsto(z_{j} \oplus 0, 0),\\
        t_{j(n+2)}& \longmapsto( 0 \oplus z_{j}, 0),\\
        t_{(n+1)(n+2)} & \longmapsto ( 0 \oplus 0, 1),
    \end{split}
\end{equation*}
 for $j=1, \dots, n$ is an isomorphism of Lie algebras.
\end{lemma}

We close this subsection with the genus zero analogue of Proposition \ref{Cgamma_is_C}:
\begin{lemma} \label{D_is_C_adp}
    The two commutative diagrams
    \begin{equation*}
        \mathcal{D}^{f} \defeq
        \begin{tikzcd}
         & L^{2} \arrow[d, "\Delta"] \arrow[dl, "d_{{n+1}}^{f}" swap]\\
         \mathfrak{e} \arrow[r, "P" swap] & L^{n} \oplus L^{n}
    \end{tikzcd} \hspace{2em} \text{and} \hspace{2em}
    \mathcal{C}^{f} \defeq
    \begin{tikzcd}
        & L \arrow[d, "\Delta"] \arrow[dl, "\Delta \times T^{f}" swap]\\
         L \oplus L \times_G A \arrow[r, "P_1" swap] & L \oplus L
    \end{tikzcd}
    \end{equation*}
    are isomorphic as objects in $\mathbf{RelE}$.
\end{lemma}
\begin{proof}
    As in the higher genus case, we can construct an explicit isomorphism in $\mathbf{RelE}$:
    \begin{equation*}
        \begin{tikzcd}
            \mathfrak{e} \arrow[r, "P"] \arrow[d, "\varphi"]& L^n \oplus L^n \arrow[d, "\mu \oplus \mu"] \\
            (L \oplus L) \times_G A \arrow[r, "P_1"] & L \oplus L
        \end{tikzcd}, \hspace{1em}
        \begin{tikzcd}
            L^n \arrow[r, "\Delta"] \arrow[d, "\mu"] & L^n \oplus L^n \arrow[d, "\mu \oplus \mu"]\\
            L \arrow[r, "\Delta"] & L \oplus L
        \end{tikzcd}, \hspace{1 em}
        \begin{tikzcd}
            L^n \arrow[r, "{d_{n+1}^{f}}"] \arrow[d, "\mu"] & \mathfrak{e} \arrow[d, "\varphi"]\\
            L \arrow[r, "\Delta \times T^{f}" swap] & (L \oplus L ) \times_G A
        \end{tikzcd},
    \end{equation*}
    where $\mu: L^n \rightarrow L$ is the isomorphism specified by $\mu(t_{j{n+1}})=z_j$ for $j=1, \dots, n$. The commutativity of the first two diagrams can be checked by direct computation. To check the commutativity of the third, notice that, as a map in $\mathrm{Hom}_\mathbb{K}(L, A)$,
    \begin{equation*}
        T^{f}(z_j)=q^{f}(z_j)=r^f(z_j)
    \end{equation*}
    for $j=1, \dots, 2$. Then
    \begin{equation*}
        \begin{split}
            \varphi d_{n+1}^{f}(t_{j(n+1)})&= \varphi(t_{j(n+1)}+t_{j(n+2)}+r^f(t_{j(n+1)})t_{(n+1)(n+2)})\\
            &=(z_j \oplus z_j,r^f(t_{j(n+1)}))\\ 
            &=\Delta \times T^{f}(z_j)=\Delta \times T^{f} (\mu (t_{j(n+1)})).
        \end{split}
    \end{equation*}
\end{proof}
\begin{remark}\label{imp_rem}
    Consider the section $u: L\oplus L \rightarrow (L \oplus L) \times_G A$, $u(x\oplus y)= (x \oplus y, 0)$ of the projection $P_1$. We can express the $L \oplus L$-action on $\ker P_1$ as
\begin{equation*}
    (x\oplus y) \cdot(0,a)=(0,xa-ay)
\end{equation*}
for all $a \in A$. This immediately implies that $N$ is free as an $L \oplus L$-module.
\end{remark}

\begin{lemma}
    Let $n \geq 1$. Given any framing $f$ of $\Sigma$,
    \begin{equation*}
        \mathcal{M}^f \defeq ((\mathrm{PaCD}^f, \mathrm{PaCD}^f, t_{11}),\, \gamma^f: \overline{K}_{n+1} \rightarrow K_{n+1},\, \phi: N \rightarrow U\overline{K}_{n+1})\in \mathrm{OpR}^\Delta_f(n+1),
    \end{equation*}
    where $\phi$ is the $\overline{K}_{n+1} \oplus \overline{K}_{n+1}$-module isomorphism uniquely defined by setting $\phi(t_{(n+1)(n+2)})=1$.
\end{lemma}
\begin{proof}
    We can see that $(\mathrm{PaCD^f}, \mathrm{PaCD}^f, t_{11})\in \mathrm{OpR}^\Delta$, and we have already observed how $\overline{K}_{n+1}=K_{n+1}/\langle t_{(n+1)(n+1)}\rangle$ is isomorphic to a free Lie algebra. We just need to verify that $\phi$ is a well defined isomorphism. By Remark \ref{imp_rem}, we already know $N$ is free as an $L^n \oplus L^n$-module. The element $t_{(n+1)(n+2)}$ is indeed a generator, as $\varphi(t_{(n+1)(n+2)})=(0,1)$ is the generator of $\ker P_1 \cong A$ as an $L \oplus L$-module.
\end{proof}

\begin{theorem}\label{big_adp}
    For $n \geq 1$, let $f$ be a framing of $\Sigma_{0,n+1}$. Then $\Lambda_{n+1}(\mathcal{M}^f)$ is isomorphic to the associated graded Goldman-Turaev Lie bialgebra on $\Sigma_{0,n+1}$ with respect to $f$.
\end{theorem}
\begin{proof}
    Recall that the functor $\Lambda_{n+1}$ consists of the composition
    \begin{equation*}
        \begin{tikzcd}[row sep=small]
            \mathrm{OpR}^\Delta_f(n+1)  \arrow[r, "\mathrm{D}_{n+1}"]    & \mathbf{RelE}^\eta \arrow[r, "\mathrm{F}^{-1}", "\cong"'] & \mathbf{RelCo}^\eta \arrow[r, "\mathrm{E}^{-1}", "\cong"'] & \mathbf{Fox}^\eta \arrow[r, "\Psi"] & \mathbf{GoTu}
        \end{tikzcd}.
    \end{equation*}
    By definition, the bracket and cobracket on $\Lambda(\mathcal{M}^f)$ are induced by the Fox pairing and quasi-derivation that correspond to the object of $\mathbf{RelE}^\eta$
    \begin{equation*}
    \left(\mathrm{P}\left(
    \begin{tikzcd}[column sep=large]
        K_{n+1} \arrow[r, "\pi"] \arrow[d, "d_{n+1}"] & \overline{K}_{n+1}\arrow[d, "\Delta"] \arrow[l, bend right=30, swap, "\gamma^f"]\\
        \mathfrak{e} \arrow[r, "s_{n+2} \oplus s_{n+1}" swap] & 
        \overline{K}_{n+1} \oplus \overline{K}_{n+1}
    \end{tikzcd} \right),\,
    \begin{tikzcd}
        N \arrow[d, "\phi"]\\
        U\overline{K}_{n+1}
    \end{tikzcd}
    \right) = (\mathcal{D}^f, \phi).
    \end{equation*}
    By Lemma \ref{D_is_C_adp}, the bialgebra $\Lambda_{n+1}(\mathcal{M}^f)$ will be isomorphic to the bialgebra corresponding to $(L, q^f \oplus \rho_G) \in \mathrm{Ob}(\mathbf{Fox})$, but by Lemmas \ref{rho_adp} and \ref{q_adp}, this is the desired Goldman-Turaev bialgebra.
\end{proof}

\subsection{$\mathrm{GRT}^f$}

\begin{definition}[\cite{gonz}, Definition 2.16]
    The \textit{graded framed Grothendieck-Teichm\"uller} group is the group
    \begin{equation*}
        \mathrm{GRT}^f \defeq \mathrm{Aut}^+_{\mathrm{Op\, \mathbf{Cat(CoAlg)}}}(\mathrm{PaCD}^f),
    \end{equation*}
    where the superscript $+$ indicates that maps in $\mathrm{GRT}^f$ are the identity on objects. We also define $\mathrm{GRT}^f_1$ to be the subgroup of elements $F \in \mathrm{GRT}^f$ that satisfy
    \begin{equation}\label{fix_t11}
        F(t_{11})=t_{11}.
    \end{equation}
\end{definition}

\begin{definition}
    Define the group
    \begin{equation*}
        \mathrm{GRT}^{0} \defeq \mathrm{Aut}^+_{\mathrm{OpR\, \mathbf{Cat(CoAlg)}}}(\mathrm{PaCD}^f, \mathrm{PaCD}^f).
    \end{equation*}
    Similarly to before, we set $\mathrm{GRT^0_1}$ to be the subgroup
    \begin{equation*}
        \mathrm{GRT}^{0}_1 \defeq \{(F,G) \in \mathrm{GRT}^0 \,| \, F(t_{11})=1\}.
    \end{equation*}
\end{definition}
\begin{lemma} \label{GRT_is_GRT0}
    The map:
    \begin{equation*}
        \begin{split}
            \mathrm{GRT}^0 & \xrightarrow{\iota} \mathrm{GRT}^f\\
            (F,G) & \mapsto F
        \end{split}
    \end{equation*}
    is a group isomorphism.
\end{lemma}
\begin{proof}
    Recall the notation $\Upsilon_n = \mathrm{End}_{\mathrm{PaCD}^f}(d^{n-1}(1))$. The map $\iota$ is clearly surjective. To prove that it is also injective, assume $(\mathrm{id}, G) \in \mathrm{GRT^0}$, and let $\sigma$ be a morphism of arity $n$ in $\mathbf{PaCD}^f$. Then
    \begin{equation*}
        G(\sigma)=G(\mathbf{1}\circ_1^{1,n} \sigma)=G(\mathbf{1})\circ \mathrm{id}(\sigma)=\sigma,
    \end{equation*}
    where $\mathbf{1}$ stands for the multiplicative identity in $\mathrm{End}_{\mathbf{PaCD}^f(1)}(1)$. This implies that also $G = \mathrm{id}$.  
\end{proof}

Let $\gamma: L^n \rightarrow K_{n+1}$ be a section of the canonical projection
\begin{equation*}
    \pi: K_{n+1} \rightarrow K_{n+1}/\langle t_{(n+1)(n+1)} \rangle.
\end{equation*}
Given $(F,G) \in \mathrm{GRT}_1^0$, $g$ induces automorphisms in $K_{n+1}$ and $L^n$; we distinguish them as
\begin{equation*}
    \overline{G}_{n+1} \in \mathrm{Aut_{\mathbf{Lie}}}(K_{n+1}), \hspace{2 em} G_{n+1} \in \mathrm{Aut}_\mathbf{Lie}(L^n).
\end{equation*}
Let $\mathrm{GRT}_1^{0,\gamma}$ denote the subgroup of $\mathrm{GRT}_1^0$ that preserves the section $\gamma$, in the sense that
\begin{equation}\label{fix_adp}
    \gamma \circ G_{n+1} = \overline{G}_{n+1} \circ \gamma
\end{equation}
for all $(F,G) \in \mathrm{GRT}_1^{0,\gamma}$. Almost as a direct consequence of Theorem \ref{big_adp}, we have that
\begin{proposition}
    Given a framing $f$ of $\Sigma_{0,n+1}$, the map
    \begin{equation*}
        \begin{split}
            \mathrm{GRT}_1^{0, \gamma^f} &\longrightarrow \overline{\mathrm{KRV}}_{0,n+1}^f\\
            (F,G) & \longmapsto(G_{n+1}: L \rightarrow L)
        \end{split}
    \end{equation*}
    is a group morphism.
\end{proposition}
\begin{proof}
    We only need to show that $\mathrm{GRT}_1^{0, \gamma^f} \subset \mathrm{Aut}_{\mathrm{OpR}^\Delta_f(n+1)}(\mathcal{M}^f)$. By definition, any element of $\mathrm{GRT}_1^{0,\gamma^f}$ will fix $t_{11}$ and preserve the section $\gamma^f$. Lastly, let $\tilde{G}: \mathfrak{e} \rightarrow \mathfrak{e}$ be the automorphism of $\mathfrak{e}$ induced by $(F,G) \in \mathrm{GRT}_1^{0, \gamma^f}$. Lemma \ref{phi_is_fix} still implies that $\tilde{G}(t_{(n+1)(n+2)})=t_{n(n+1)(n+2)}$, so that the $L \oplus L$-isomorphism $\phi: N \rightarrow A$ uniquely defined by $\phi(t_{(n+1)(n+2)})=1$ is indeed \textit{fixed}.
\end{proof}

Lastly, we would like to give special treatment to the group $\overline{\mathrm{KRV}}_{0,3}^{f^{\mathrm{adp}}}$. Recall that the \textit{adapted framing} $f^{\mathrm{adp}}$ on $\Sigma_{0,3}$ is specified by the requirement
    \begin{equation*}
        \mathrm{rot}^{f^\mathrm{adp}}(\gamma_j)=-1
    \end{equation*}
for $j=1,2$.
\begin{lemma}\label{GRTgamma_is_GRT}
    Let $\gamma^{\mathrm{adp}}:L^2 \rightarrow K_3$ be the section induced by the adapted framing $f^{\mathrm{adp}}$ in the sense of Lemma \ref{sec_from_frames}, so that
    \begin{equation*}
        \gamma^{\mathrm{adp}}(t_{j3})=t_{j3}
    \end{equation*}
    for $j=1,2$. Then $\mathrm{GRT}_1^{0, \gamma^{\mathrm{adp}}}=\mathrm{GRT}_1^0$.
\end{lemma}
\begin{proof}
    Let $(F,G) \in \mathrm{GRT}_1^0$. We would like to prove equality (\ref{fix_adp}) for $\gamma=\gamma^{\mathrm{adp}}$. We know $G_3$ and $\overline{G}_3$ are tangential automorphisms. Since 
    \begin{equation*}
        \mathfrak{t}_3^f= \left(\bigoplus_{i=1}^3 \mathbb{K}t_{ii} \right) \oplus \mathfrak{t}_3, 
    \end{equation*}
    there exists a group-like element $\phi=\phi(t_{13},t_{23}) \in UL_2$ such that
    \begin{equation*}
        \overline{G}_3 \gamma^{\mathrm{adp}}(t_{j3})=\overline{G}_3(t_{j3})=\phi^{-1}t_{j3}\phi=G_3(t_{j3})=\gamma^{\mathrm{adp}}G_3(t_{j3})
    \end{equation*}
    for $j=1,2$.
\end{proof}
In view of the isomorphism $\mathrm{GRT}_1^0 \cong \mathrm{GRT}^f_1$ of Lemma \ref{GRT_is_GRT0}, the preceding lemma implies our version of the Alekseev-Torossian injection of $\mathrm{GRT}$ into $\mathrm{KRV}$:
\begin{corollary}
    The map
    \begin{equation*}
        \begin{split}
            \mathrm{GRT}^f_1 &\longrightarrow \overline{\mathrm{KRV}}_{(0,3)}^{f^\mathrm{adp}}\\
            G & \longmapsto (G_3: L \rightarrow L)
        \end{split}
    \end{equation*}
    is a group morphism.
\end{corollary}

\printbibliography

\end{document}

%% file: Packages.tex
\usepackage[T1]{fontenc}
\usepackage[latin1]{inputenc}
\usepackage[english]{babel}
\usepackage{siunitx}
\usepackage{tipa} 
\usepackage{times} 
\usepackage{amsmath}
\usepackage{amssymb}
\usepackage{latexsym}
\usepackage{amscd}
\usepackage{mathalfa}
\usepackage{mathrsfs} 
\usepackage{relsize}
\usepackage{delarray}
\usepackage{pstricks}
\usepackage{changepage}
\usepackage{euscript}
\usepackage{textcomp}
\usepackage{esvect}
\usepackage{parskip}
\usepackage{placeins}
\usepackage{subfigure}
\usepackage{amsthm}
\usepackage{amsfonts}
\usepackage{musicography}
\usepackage{hyperref}
\usepackage{tensor}
\usepackage{todonotes}
\usepackage{scalerel}

\usepackage{multirow}%
\usepackage{amsmath,amssymb,amsfonts}%
\usepackage{amsthm}%
\usepackage{mathrsfs}%
\usepackage{xcolor}%
\usepackage{textcomp}%
\usepackage{manyfoot}%
\usepackage{booktabs}%
\usepackage{algorithm}%
\usepackage{algorithmicx}%
\usepackage{algpseudocode}%
\usepackage{listings}%

\usepackage[backend=biber, style=alphabetic, sorting = nyt]{biblatex} 
\DeclareFieldFormat[article,inbook,incollection,inproceedings,patent,thesis,unpublished]{title}{\textit{#1}}
\renewbibmacro{in:}{}
\DeclareFieldFormat{journaltitle}{#1}

\usepackage{mathtools}
\newcommand{\defeq}{\vcentcolon=}
\newcommand{\eqdef}{=\vcentcolon}
\newcommand{\dbl}{\{\!\!\{}
\newcommand{\dbr}{\}\!\!\}}

\usepackage{array}
\usepackage{delarray}
\usepackage{stmaryrd}
\usepackage{fancyhdr}
\usepackage{graphpap}
\usepackage{makeidx}
\usepackage{enumerate}
\usepackage{esint}
\usepackage{datetime}
\usepackage{caption}
\usepackage{smartdiagram}
\usesmartdiagramlibrary{additions}
\usepackage{lastpage}

\usepackage{color}
\definecolor{igreen}{rgb}{0.0, 0.56, 0.0}
\usepackage{xcolor, colortbl}
\colorlet{gred}{-red!75!green!65!}
\colorlet{mamber}{-red!75!green!15!blue!50!}
\colorlet{grown}{-red!75!blue!20!green}
\colorlet{bled}{-red!85!blue!40!green!45!}
\colorlet{waters}{cyan!25} 
\colorlet{water}{cyan!25!green!20!} 
\definecolor{grin}{HTML}{00F9DE}
\usepackage{rotating}

\usepackage{booktabs}
\usepackage{graphicx}
\usepackage{tikz}
\usepackage{tikz-3dplot}
\usetikzlibrary{
arrows,
arrows.meta,
automata,
backgrounds,
calc,
decorations,
decorations.pathmorphing,
decorations.pathreplacing,
decorations.fractals,
external,
fit,
matrix,
petri,
positioning,
shadows,
shapes,
shapes.multipart,
topaths,
intersections
}
\usepackage{eso-pic}
\def\ba{\begin{array}}
\def\ea{\end{array}}
\def\beann{\begin{eqnarray*}}
\def\eeann{\end{eqnarray*}}
\def\bea{\begin{eqnarray}}
\def\eea{\end{eqnarray}}

\usepackage{lipsum}
\usepackage{tikz-cd}
\usepackage{float}
\usepackage{titling}
\usepackage{epigraph}
\usepackage[title, titletoc]{appendix}
\newtheorem{theorem}{Theorem}[section]
\newtheorem{proposition}[theorem]{Proposition}
\newtheorem{corollary}{Corollary}[theorem]
\newtheorem{lemma}[theorem]{Lemma}

\theoremstyle{definition}
\newtheorem{definition}[theorem]{Definition}
\newtheorem{example}[theorem]{Example}

\theoremstyle{remark}
\newtheorem{remark}[theorem]{Remark}

\theoremstyle{conjecture}
\newtheorem{conjecture}[theorem]{Conjecture}

\DeclareMathOperator{\ima}{Im}